\documentclass[11pt]{amsart}
\usepackage[T1]{fontenc}
\usepackage{times,mathptm}
\usepackage{amssymb,verbatim}
\usepackage[all]{xy}
\usepackage{subfig}
\usepackage{tikz}
\usepackage{color}
\usepackage{a4wide}
\usepackage[normalem]{ulem}
\usepackage{hyperref}
\usepackage{wrapfig}
\usetikzlibrary{arrows}

\newcommand{\R}{\mathbb{R}}
\newcommand{\C}{\mathbb{C}}

\newcommand\Z{\mathbb{Z}}
\newcommand\Prym{\textrm{Prym}}
\newcommand{\N}{\mathbb{N}}
\newcommand{\Q}{\mathbb{Q}}

\newcommand{\SL}{{\rm SL}}
\newcommand{\GL}{{\rm GL}}

\renewcommand{\S}{\mathbb{S}}
\newcommand{\lra}{\longrightarrow}
\newcommand{\ra}{\rightarrow}
\newcommand{\id}{\mathrm{id}}
\newcommand{\Ord}{\mathcal{O}}
\newcommand{\ety}{\emptyset}
\newcommand{\Mod}{\mathfrak{M}}
\newcommand{\Pcal}{\mathcal{P}}
\newcommand{\Qcal}{\mathcal{Q}}

\newcommand{\Id}{\mathrm{Id}}
\newcommand{\SD}{\mathcal{S}_D}
\renewcommand\mod{\text{ mod }}
\newcommand{\DS}{\displaystyle}

\newtheorem{Theorem}{Theorem}[section]
\newtheorem{Corollary}[Theorem]{Corollary}
\newtheorem{Lemma}[Theorem]{Lemma}
\newtheorem{Proposition}[Theorem]{Proposition}
\newtheorem{Remark}[Theorem]{Remark}
\newtheorem{Definition}[Theorem]{Definition}
\newtheorem{Example}[Theorem]{Example}
\newtheorem*{NoNumberTheorem}{Theorem}
\newtheorem*{Notation}{Notation}

\begin{document}
\title[Teichm\"uller  curves  and  Prym eigenforms]  {Teichm\"uller
curves generated by Weierstrass Prym eigenforms in genus three and genus four}

\author{Erwan Lanneau, Duc-Manh Nguyen}

\begin{abstract}
This paper is  devoted to the classification of  the infinite families
of Teichm\"uller curves  generated by Prym
eigenforms of genus $3$ having  a  single  zero.  These curves  were  discovered  by
McMullen~\cite{Mc7}.  The   main  invariants    of  our
classification is the discriminant  $D$ of the corresponding quadratic
order, and the generators of this order. It  turns out  that for $D$  sufficiently large, there  are two
Teichm\"uller curves  when $D\equiv 1 \mod 8$,  only one Teichm\"uller
curve when  $D \equiv  0,4 \mod 8$,  and no Teichm\"uller  curves when
$D\equiv 5 \mod 8$. For small values of $D$, where this classification
is not  necessarily true,  the number of  Teichm\"uller curves  can be
determined  directly.  The  ingredients  of our  proof  are  first,  a
description of  these curves  in terms of  prototypes and  models, and
then  a careful  analysis of  the combinatorial  connectedness  in the
spirit of McMullen~\cite{Mc4}. As a consequence, we obtain a description of cusps 
of Teichm\"uller curves given by Prym eigenforms.

We would like also to emphasis that even though we have the same statement 
compared to~\cite{Mc7}, when $D\equiv 1 \mod 8$, the reason for this disconnectedness is different.

The   classification  of  these Teichm\"uller  curves plays  a key  role in  our investigation  
of the dynamics of $\textrm{SL}(2,\R)$ on the intersection of  the   Prym   eigenform  
locus   with  the   stratum $\Omega\mathfrak{M}(2,2)$, which is the object of a forthcoming paper.
\end{abstract}

\address{
Unit\'e Mixte de Recherche (UMR 6207) du CNRS et de l'Universit\'e d'Aix-Marseille \newline
et de l'Universit\'e du Sud Toulon-Var.
Unit\'e affili\'ee \`a la FRUMAM F\'ed\'eration de Recherche 2291 du CNRS.\newline
Luminy, Case 907, F-13288 Marseille Cedex 9, France
}

\email{lanneau@cpt.univ-mrs.fr}

\address{
IMB Bordeaux-Universit\'e Bordeaux 1\newline
351, Cours de la Lib\'eration \newline 
33405 Talence Cedex, FRANCE}

\email{duc-manh.nguyen@math.u-bordeaux1.fr}

\date{\today}


\maketitle
\setcounter{tocdepth}{1}
\tableofcontents

\section{Introduction}

For a  long time,  it has  been known that  the ergodic  properties of
linear  flows on  a translation  surface are  strongly related  to the
behavior of  its $\textrm{SL}(2,\mathbb R)$-orbit in  the moduli space
$\Omega\mathfrak{M}_g$  (see  {\it e.g.}~\cite{Masur2002,Zorich:survey}  for
surveys    of    the     literature    on    this    subject).     The projection of the
$\textrm{SL}(2,\mathbb  R)$-orbit of a  translation surface  into the Teichm\"uller space is 
a  {\it  Teichm\"uller disc}.   It  has  been  known that  when  the
stabilizer of  a surface is  a lattice in  $\textrm{SL}(2,\mathbb R)$,
the   Teichm\"uller   disc   projects   onto  a   {\it   Teichm\"uller
  curve}~\cite{Veech1992,Smillie2006}    in     the    moduli    space
$\mathfrak{M}_g$. \medskip

Since the work  of Veech much effort has gone into  the study of these
Teichm\"uller discs  and their closures. Two decades  later, thanks to
the   seminal  work   of   McMullen~\cite{Mc1,Mc2,Mc4,Mc5,Mc6,Mc3},  a
complete  classification of  the closures  of Teichm\"uller  discs, as
well  as  a  complete  list  of the  Teichm\"uller  curves  has  been
established   in  genus  two   (see~\cite{Calta2004} for a partial 
classification involving different ideas). See~\cite{Moller2006,
Moller2008,Bainbridge:Moller,Eskin:preprint}        for        related
results in higher genera. \medskip

McMullen's analysis  relates Teichm\"uller curves to  the locus $\Omega
E_{D}(2)\subset  \Omega\mathfrak{M}_2$ (respectively, $\Omega  E_{D}(1,1)$) 
corresponding          to         pairs
$(X,\omega)\in\Omega\mathfrak{M}_2$  where $X$  is  a Riemann  surface
whose  the Jacobian  $\textrm{Jac}(X)$ admits  real  multiplication by
some order $\mathcal O_{D}$ and $\omega$  being then an eigenform for the
real  multiplication with a single zero (respectively, two simple zeroes) 
(see  Section~\ref{sec:background}  for  precise definitions). \medskip

Roughly  speaking, one  can  single out  two  key facts  of genus  two
surfaces, playing a crucial role in McMullen's approach:

\begin{enumerate}
\item The  existence of the so-called  hyperelliptic involution $\rho$
  on $X$, and
\item the complex dimension of $H^{1}(X,\C)$ is $2$.
\end{enumerate}

Later~\cite{Mc7} McMullen extended these results to the Prym eigenform
loci $\Omega E_{D}$  in higher genera, that can  be thought as natural
loci where these two above properties remains true. Then it was proven
in~\cite{Mc7} that these loci are closed $\textrm{SL}(2,\mathbb
R)$-invariant.  Moreover  by  a  dimension  count,  the  intersections
$\Omega  E_{D}(4) \subset  \Omega\mathfrak{M}_3$ and  $\Omega E_{D}(6)
\subset  \Omega\mathfrak{M}_4$  of  $\Omega  E_{D}$ with  the  minimal
strata, consist entirely of Teichm\"uller curves. \medskip

The  family  of Teichm\"uller  curves  in  $\Omega  E_D(2)$ has  been
classified by McMullen~\cite{Mc4}. In this paper we prove a classification type 
result for $\Omega E_D(4)$:

\begin{Theorem}
\label{theo:main:intro}
For $D \geq 17$,  $\Omega E_D(4)$ is non empty if and
only  if $D\equiv 0,1, \text{ or } 4 \mod  8$, and all the loci $\Omega E_D(4)$ are pairwise disjoint. 
Moreover, for those discriminants, the following dichotomy holds. Either
\begin{enumerate}
\item $D$ is odd and then $\Omega E_D(4)$ has exactly two connected components, or
\item $D$ is even and $\Omega E_D(4)$ is connected.
\end{enumerate}
In addition, each component of $\Omega E_D(4)$ corresponds to a closed $\GL^+(2,\R)$-orbit.
\end{Theorem}

For $D<17$, $\Omega E_D(4)$ is non-empty if and only if $D\in\{8,12\}$ 
and in this case, it is connected (see Theorem~\ref{MainTh2}). As a direct consequence, we 
get a surprising fact: all surfaces in $\Omega E_8(4)$ have no simple 
cylinders. Note that translation surfaces with no simple cylinders are quite rare, as generic ones 
always have simple cylinders. All  examples of such surfaces (e.g. the
Wollmilchsau surface)  known to the authors  are square-tiled surfaces
(see~\cite{Herrlich:Schmithusen}). Since $8$  is not a perfect square,
surfaces of $\Omega E_8(4)$ are not square-tiled. \medskip

\begin{Remark}
It is not difficult to see that the parity of the spin structures determined by Abelian differentials in 
$\Omega E_D(4)$ are odd (see~\cite{Kontsevich2003}). Theorem~\ref{theo:main:intro} 
is thus a crucial step in our 
attempt  to obtain an accurate count of the number of  components of the intersection 
$\Omega E_D\cap \Omega \mathfrak{M}(2,2)^{\mathrm{odd}}$. \medskip

This count, together with other problems such as the characterization of surfaces in 
$\Omega E_D\cap \Omega \mathfrak{M}(2,2)^{\mathrm{odd}}$, and the classification 
of the $\GL^+(2,\R)$-orbits (Ratner type theorem) will be addressed in a forthcoming 
paper~\cite{Lanneau:Mahn:ratner}.
\end{Remark}

The strategy we develop can also be used to investigate the 
connectedness of the loci $\Omega E_D(6)$, namely:

\begin{Theorem}
\label{theo:H6}
For any $D\in \N$, $D\equiv 0,1  \mod 4$, and $D \not\in \{4,9\}$, the
locus   $\Omega  E_D(6)$   is  non   empty   and  has   at  most   two
components.  Moreover if $D$ is odd then $\Omega E_D(6)$ is connected.
\end{Theorem}

%
%

\begin{Remark}
Unfortunately, we do not succeed to obtain an accurate count of $\GL^+(2,\R)$-orbits in $\Omega E_D(6)$ for 
$D$ even. There are reasons to believe that the locus $\Omega E_D(6)$ is always connected 
({\it  i.e} there  is  only {\it  one}  $\GL^+(2,\R)$-orbit). This  is
strongly supported by the fact that $\Omega E_{d^{2}}(6)$ is connected
for small values of even $d$  (we have checked for $d\leq 20$) and for
small values of $D$ (e.g. $D < 53$) that are not a square.

Since it seems to us that such a result would require other tools and ideas than what has been 
introduced  in  the present  paper,  we  will  adress the  topological
classification of $\Omega E_D(6)$ in a forthcoming paper.
\end{Remark}

\subsection*{Euler characteristics of Prym loci}

In the recent paper~\cite{Moller2011}, M\"oller provide a way to
calculate  the  Euler  characteristics  of  the  Teichm\"uller  curves
obtained by the Prym construction.

From  this  point  of  view  this  paper is,  jointly  with  the  work
of~\cite{Moller2011}, a continuation of~\cite{Mc7}.

\subsection*{Cusps of Teichm\"uller curves}

The projection of $\Omega E_{D}(2g-2)$ to the moduli space $\mathcal M_{g}$ leads 
to Teichm\"uller curves. Let $\SL(X,\omega)$ denote the Veech group of $(X,\omega)$ which 
is a lattice of $\SL(2,\R)$. A Teichm\"uller curve can never be compact, since any periodic direction 
of $(X,\omega)$ gives rise to a cusp, each cusps corresponds to the $\SL(X,\omega)$-orbits of a 
periodic direction of $(X,\omega)$. \medskip

For $g=2,3,4$, let $W_D(2g-2)$ be the projection of $\Omega E_D(2g-2)$ into $\mathfrak{M}_g$. By~\cite{Mc4}, Theorem~\ref{theo:main:intro}
and Theorem~\ref{theo:H6} we know that  $W_D(2g-2)$ is the union of at
most two Teichm\"uller curves.  \medskip

A corollary of our result is a description of the number of cusps of theses curves for $g=3$ and $g=4$. 
For        a        more        detailed       description,        see
Appendix~\ref{appendix:cusps}.\\
Table~\ref{tab:cusps} represents the number of cusps for the first discriminants. 
Observe  that the  first data  ({\it  i.e.} $g=2$)  have been  already
established by McMullen (see~\cite{Mc4}, Table~$3$).

\begin{Remark}
For a combinatorial description of the cusps of $W_{D}(4)$, see Propositions~\ref{NormA1Prop},~\ref{NormA2Prop} 
\& \ref{NormBProp} for the $3-$cylinder decompositions and Theorem~\ref{theo:12cyl:prototype} for 
the $1$ and $2$-cylinder decompositions.

When $D$ is not a square, the number of cusps of
$W_{D}(2)$ and $W_{D}(6)$ is the same (see Proposition~\ref{prop:cusps:egalite}).
\end{Remark}

\begin{table}[htbp]
$$
\begin{array}{ccc}
\textrm{Genus } $2$ & \textrm{Genus } $3$ & \textrm{Genus } $4$ \\
\begin{array}{|c|c|c|}
\hline 
D & |\mathcal P_{D}| & |C(W_{D}(2))| \\
\hline
5 & 1 & 1 \\
8 & 2 & 2 \\
{\bf 9} & {\bf 1} & {\bf 1+1} \\
12 & 3 & 3 \\
13 & 3 & 3 \\
{\bf 16} & {\bf 2} & {\bf 2+1} \\
17 & 6 & 6 \\
20 & 5 & 5 \\
21 & 4 & 4 \\
24 & 6 & 6 \\
{\bf 25} & {\bf 6} & {\bf 6+2} \\
28 & 7 & 7 \\
29 & 5 & 5 \\
32 & 7 & 7 \\
33 & 12 & 12  \\
{\bf 36} & {\bf 5} & {\bf 5+3} \\
37 & 9 & 9 \\
40 & 12 & 12 \\
41 & 14 & 14 \\
44 & 9 & 9 \\
45 & 8 & 8 \\
48 & 11 & 11 \\
{\bf 49} & {\bf 13} & {\bf 13+5} \\
52 & 15  & 15 \\
\hline
\end{array}
&
\begin{array}{|c|c|c|}
\hline 
|\mathcal P_{D}| & |\mathcal P'_{D}| & |C(W_{D}(4))| \\
\hline
0 & 0 & 0\\
0 & 1 & 1\\
{\bf 0} & {\bf 0} & {\bf 0}\\
1 & 0 & 2\\
0 & 0 & 0\\
{\bf 0} & {\bf 0 } & {\bf 0}\\
2 & 2& 6\\
1 & 2& 4\\
0 & 0& 0\\
2 & 0& 4\\
{\bf 2} & {\bf 0}& {\bf 4+2}\\
1 & 2& 4\\
0 & 0& 0\\
3 & 2& 8\\
4 & 6& 14\\
{\bf 1} & {\bf 0}& {\bf 2+2}\\
0 & 0& 0\\
2 & 2& 6\\
7 & 2& 16\\
3 & 0& 6\\
0 & 0& 0\\
3 & 4& 10\\
{\bf 4} & {\bf 2}& {\bf 10+6}\\
5 & 2& 12\\
\hline
\end{array}
&
\begin{array}{|c|c|c|}
\hline 
|\mathcal P_{D}| & |\mathcal P'_{D}| & |C(W_{D}(6))| \\
\hline
0 & 1 & 1\\
1 & 1 & 2\\
{\bf 0} & {\bf 0} & {\bf 0}\\
1 & 2 & 3\\
2 & 1 & 3\\
{\bf 1} & {\bf 0}& {\bf 1+1}\\
2 & 4& 6\\
3 & 2& 5\\
2 & 2& 4\\
4 & 2& 6\\
{\bf 2} & {\bf 1}& {\bf 3+3}\\
3 & 4& 7\\
4 & 1& 5\\
4 & 3& 7\\
6 & 6& 12\\
{\bf 3} &{\bf 0}& {\bf 3+5}\\
4 & 5& 9\\
6 & 6& 12\\
8 & 6& 14\\
7 & 2& 9\\
4 & 4& 8\\
7 & 4& 11\\
{\bf 6} & {\bf 3}& {\bf 9+9}\\
7 & 8& 15\\
\hline
\end{array}
\end{array}
$$
\caption{
\label{tab:cusps}
The number of cusps $C(W_{D}(2g-2))$ of the Weierstrass curve in genus $2$, $3$ and $4$. Lines in bold
correspond to square tilde surfaces ($\sqrt{D}\in \N$); In this case the number of cusps is broken down as 
the sum of the number of cusps for decomposition in models $A\pm,B$ and cusps for others decompositions.
When $D$ is not a square, the number of cusps is given, respectively for $g=2,3,4$, by $|\mathcal P_{D}|$,
$2|\mathcal P_{D}|+|\mathcal P'_{D}|$, and $|\mathcal P_{D}|+|\mathcal P'_{D}|$.
}
\end{table}

\newpage
\subsection*{Outline}

We  briefly  sketch  the  proof  of  Theorem~\ref{theo:main:intro}.  It
involves  decompositions  of  surfaces  into  cylinders,  and  then  a
combinatorial analysis of the  space of such decompositions. This last
step is tackled using number theory arguments.

\begin{enumerate}

\item Associated to any  Abelian differential $(X,\omega)\in \Omega
  E_D(4)$   is    a   flat   metric   structure    (with   cone   type
  singularities).  Since $(X,\omega)$  is a  Veech surface~\cite{Mc7},  
  the Veech dichotomy ensures that there are
  infinitely many {\em completely periodic} directions {\it i.e.}  each trajectory is either a
  saddle  connection or  a closed  geodesic. The  surface is  then the
  union of finitely many open cylinders and saddle connections in this
  direction.  The  parameters of a  cylinder are denoted  by $(w,h,t)$
  for,  respectively,   the  {\it   width},  {\it  height}   and  {\it
  twist}.  It turns out that a surface in $\Omega E_D(4)$ always admits a decomposition 
  into three cylinders following one of three topological models (named $A+$, $A-$, and $B$) 
  presented in Figure~\ref{fig:3cyl:models}. 
  This corresponds to Proposition~\ref{prop:decomp:cylinders} and Corollary~\ref{cor:decomp:3cyl}.
  \medskip

\item The  first main ingredient  is a combinatoric  representation of
  every surface  in $\Omega E_D(4)$ that is  decomposed into three cylinders
  in the  horizontal direction. We  show (Proposition~\ref{NormA1Prop}
  and   Proposition~\ref{NormA2Prop})  that up   to  the   action  of
  $\textrm{GL}^+(2,\R)$ and appropriate Dehn twists, any cylinder decomposition of type 
  $A+$ or $A-$ can be encoded by the  parameters    $(w,h,t,e,\varepsilon)$,   where   $(w,h,t,e)\in \Z^4$ satisfies
$$ (\mathcal{P})\left\{
\begin{array}{l}
w>0, \ h>0,\ e+2h< w,\ 0\leq t<\gcd(w,h), \\ \gcd(w,h,t,e)=1, \textrm{
  and } D=e^2+8hw,
\end{array}
\right.
$$
and   $\varepsilon   \in   \{\pm\}$ (compare with~\cite{Mc4}).   The   set   of   all
$p=(w,h,t,e)\in\Z^4$ satisfying $(\Pcal)$ is denoted by $\Pcal_D$, and
the set of $(p,\varepsilon)$  with $p\in \Pcal_D$ and $\varepsilon \in
\{\pm\}$ is  denoted by  $\mathcal{Q}_D$. The set  $\Qcal_D$ naturally
parametrizes the  set of all decompositions of type $A\pm$ for a fixed $D$.  For $D\neq 8$, 
since any surface in $\Omega E_D(4)$ always admits a decompositions of type $A\pm$ 
(see Proposition~\ref{prop:D8:topo}), this  provides a (huge)  upper bound  for the  number of
components of $\Omega E_D(4)$ by the cardinal of  $\mathcal Q_{D}$. \medskip

\item In general, neither the Prym involution nor the quadratic order 
  is uniquely determined by  the Prym eigenform. However, the analysis
  on the prototypes allows us to show that, for
Prym eigenform in $\Omega \mathfrak{M}(4)$, the Prym involution and the quadratic order are 
unique. It follows immediately that $\Omega E_{D_1}(4)$ and $\Omega E_{D_2}(4)$ are disjoint 
if $D_1\neq D_2$ (see Theorem~\ref{UqeThm} and Corollary~\ref{cor:disjoint}). \medskip 

\item  In  Section~\ref{sec:butterfly}  we  introduce  an  equivalence
  relation  $\sim$  on $\mathcal  Q_{D}$ which are generated by  changes  of periodic  directions
   called  {\em Butterfly moves}.
  A Butterfly move is the operation of switching from a cylinder decomposition of type $A+$ or $A-$ 
  to another one such that the simple cylinders in the two decompositions are disjoint.  If we have a 
  decomposition of type $A+$, then a Butterfly move yields a decomposition of type $A-$, and  vice 
  versa. Two elements of $\Qcal_D$ are equivalent if they parametrize two cylinder decompositions 
  on the same surface which can be connected by a sequence of Butterfly moves.     Thus
$$    \#\   \{\textrm{Components   of    }   \Omega    E_D(4)\}   \leq
  \#\ \left(\mathcal Q_{D}/\sim \right).
$$ 
Given the parameters of the initial cylinder decomposition and the parameter of the Butterfly 
move, one can write down  the parameters of the new decomposition rather explicitly
 (see Proposition~\ref{BM1Prop} and Proposition~\ref{BM2Prop}). It turns out that the changing 
 rules are the same for decompositions of type $A+$ and $A-$, therefore the equivalence relation 
 $\sim$ descends  to an equivalence relation (still  denoted   by $\sim$) in $\mathcal P_{D}$:  
 for any $p, p' \in \Pcal_D$, $p   \sim  p'$  if there exist $\varepsilon, \varepsilon'\in\{\pm\}$, such
  that $(p,\varepsilon) \sim (p',\varepsilon')$ in $\Qcal_D$. Obviously
$$
 \#\left(\mathcal Q_{D}/\sim\right) \leq 2\cdot \#\left(\mathcal
  P_{D}/\sim\right).
$$
\bigskip

{\it The   remaining   parts    consist   in    showing   that $\#\left(\mathcal P_{D}/\sim\right) = 1$.} \bigskip

\item To compute the number of equivalence classes in $\Pcal_D$, we use McMullen's 
approach \cite{Mc4}. We first show that every prototype $(w,h,t,e)\in \Pcal_D$ can be  connected
 to a reduced one, {\it i.e.} a prototype $(w,h,t,e)$ where
  $h=1$. For fixed  $D$ such reduced prototypes are uniquely
  determined by its $e$ and thus they are parametrized by
$$
\mathcal S_{D}  =  \{ e\in  \Z,\  e^{2} \equiv  D  \mod 8,\  e^{2}
  \textrm{ and } (e+4)^{2} < D \}.
$$
The prototype in $\Pcal_D$ associated to an element $e\in\mathcal{S}_D$ will be denoted by $[e]$. 
We define an equivalence relation in $\mathcal S_D$ which is generated by the following condition 
$e \sim -e - 4q$ whenever $-e-4q\in \mathcal S_{D}$ and $\gcd(w,q)=1$, where $q>0$ and 
$w =  (D-e^{2})/8$. By definition, we have $e\sim e'$ in $\mathcal S_D$ implies that $[e]\sim [e']$ in
 $\Pcal_D$ but the converse is not necessarily true. Thus, the number of equivalence classes in 
 $\mathcal S_D$ gives us an upper bound for the number of equivalence classes in $\Pcal_D$. \medskip

\item  In Section~\ref{sec:connectedness:PD}  we  show that, for $D$ sufficiently large ({\it i.e.} 
$D\geq 83^2$)  $\mathcal   S_{D} /  \sim$ consists of  a single component  if 
$D \not \equiv 4 \mod 16$,  otherwise two.  It follows immediately that for $D\not\equiv 4 \mod 16$,
 there is  only one equivalence class in $\Pcal_D$. For the  remaining  cases, we  show that  it is
  possible   to   connect  the   two   components of $\mathcal{S}_D$  through   $\mathcal
  P_{D}$, hence   there is actually only one equivalence class in $\Pcal_D$,  which implies
  that $\Qcal_D$ (and thus $\Omega E_D(4)$) has at most two components.  \medskip

\item  For $D\equiv 0,4\mod 8$, two components of $\Qcal_D$ can be connected by an 
explicit path in $\Omega E_{D}$ (see Theorem~\ref{theo:connect:odd:steps}),
therefore $\Omega E_D(4)$ consists of only one $\GL^+(2,\R)$-orbit and our 
statement is proven in this case. When $D \equiv 1\mod 8$, since $\Omega E_D(4)$ contains at most 
two components, it remains to show that $\Omega E_D(4)$ is non connected. \medskip

\item When $D \equiv 1\mod 8$ (i.e. $D$ is odd since $D\not \equiv 5 \mod 8$) 
it turns out that the locus $\Omega E_D(4)$ contains at least two distinct 
 $\GL^+(2,\R)$-orbits (Theorem~\ref{theo:disconnect:odd}). Roughly
speaking, the two components correspond to two distinct complex lines in the space 
$\Omega(X,\rho)^- \simeq H^1(X,\R)^-$ (see Section~\ref{sec:OddDisc}).   \medskip

\item Note  that the  number theory arguments  that we use  only apply
  when $D$ is sufficiently large. For small
  values of $D$, Theorem~\ref{theo:main:intro} is proven with computer
  assistance (see Table~\ref{tab:exceptional:cases} page~\pageref{tab:exceptional:cases}, and
  Table~\ref{table:exceptionnal:link} page~\pageref{table:exceptionnal:link}).  Actually, the number of
  components of $\Omega  E_D(4)$ is not always equal  to the number of
  components of  $\Qcal_D$, there are several  exceptions, namely when
  $D\in \{41,\ 48,\ 68,\ 100\}$. We discuss and prove Theorem~\ref{theo:main:intro} for those exceptional cases in
  Section~\ref{sec:diagrams}.

\end{enumerate}

\subsection*{Reader's guide}

In  Section~\ref{sec:background} we review  basic definitions  on real
multiplication         and         state         precisely         the
classification.    Section~\ref{sec:topology}   is   devoted    to   a
classification of cylinder decompositions.   This allows us to obtain
a  combinatorial description  of  Prym eigenforms,  which  is achieved in
Section~\ref{sec:prototypes}.    In  Section~\ref{sec:uniqueness},    and
Section~\ref{sec:OddDisc} the combinatorial description of cylinder decompositions is used to prove that
for different  values of $D$,  the loci $\Omega E_D(4)$  are disjoint,
and when  $D$ is odd, $\Omega E_D(4)$ contains at least  two distinct $\GL^+(2,\R)$-
orbits.\\
In Section~\ref{sec:butterfly},  we introduce  the  spaces of
prototypes $\Qcal_D$ and $\Pcal _D$.  We define the ``Butterfly  Move'' operations, and compute 
the induced transformations on the sets $\Qcal_D$ and $\Pcal_D$.  Section~\ref{sec:connectedness:PD} 
can be  read independently from the
others.   We prove  the combinatorial  connectedness of  the  space of
prototypes $\Pcal_D$. Then in the last part (Section~\ref{sec:diagrams}) we give
the proof of our main result.  In the Appendix we treat the particular
cases separately, and give a quick r\'esum\'e of the classification problem for Prym eigenforms in 
$\Omega\mathfrak{M}(6)$. We also derive the number of cusps of Teichm\"uller curves given by Prym 
eigenforms (Appendix~\ref{appendix:cusps}).

\subsection*{Acknowledgments}

The authors  are grateful  to the Hausdorff  Institute in  Bonn, where
this work  began, for its hospitality, excellent  ambiance and working
conditions. The  authors thank the  Centre de Physique  Th\'eorique in
Marseille and the Universit\'e Bordeaux 1.

We thank  Vincent Delecroix, Pascal Hubert,  Samuel Leli\`evre, Carlos
Matheus, Curt McMullen and  Martin M\"oller for useful discussions and
remarks on this paper.

We also thank~\cite{sage} for computational help.

\section{Background and tools}
\label{sec:background}

We review basic notions  and results concerning Abelian differentials,
translation surfaces  and real multiplication.  For general references
see {\it e.g.}~\cite{Masur2002,Zorich:survey,Mc1,Mc7}.

\subsection{Prym varieties}

If $X$ is  a Riemann surface, and  $\rho : X \lra X$  is a holomorphic
involution then  the {\em Prym  variety} of $(X,\rho)$ is  the abelian
variety defined by
$$ \Prym(X,\rho)=(\Omega(X)^-)^*/H_1(X,\Z)^-,
$$          where          $\Omega(X)^-=\{\omega\in         \Omega(X),
\  \rho^*(\omega)=-\omega\}$,   and  $H_1(X,\Z)^-  =   \{  \gamma  \in
H_1(X,\Z),\ \rho(\gamma)=-\gamma \}$. \\
In  the rest of  this paper, we  will assume that  $ \dim_\C
\Prym(X,\rho)   =  2$.  This   assumption  can   be  easily   read  as
follows.     Since     $\Omega(X/\rho)$     is     identified     with
$\Omega(X)^+=\ker(\rho-\id)\subset \Omega(X)$ one has
$$        \dim_\C        \Prym(X,\rho)=\dim_\C       \Omega(X)-\dim_\C
\Omega(X/\rho)=\textrm{genus}(X)-\textrm{genus}(X/\rho) =2
$$

\begin{Remark}
This assumption can be thought as the natural condition coming for the
genus two case discussed in the introduction (compare with~\cite{Mc7},
Section~$3$).
\end{Remark}

In  the   present  article  we  will  concentrate   on  the  following
construction of Prym varieties.

\begin{Example}
\label{ex:main}
Let $q$ be a quadratic differential on the 2-torus having three simple
poles and a single zero (of order 3).  Let $\pi: X \rightarrow \mathbb
T^2$  be the double  orientating cover.  Then the  deck transformation
$\rho$ on  $X$ provides a  natural Prym variety  $\Prym(X,\rho)$ where
$\sqrt{\pi^\ast   q}   \in  \Omega(X)^-$.   Observe   that  the   form
$\sqrt{\pi^\ast q}$ has a unique zero of order $4$.
\end{Example}

\subsection{Real multiplication on abelian variety}

Let  $D>0$  be a  positive  integer congruent  to  $0$  or $1$  modulo
$4$. Let
$$ \Ord_D \cong \Z[X]/(X^2+bX+c)
$$ be the real quadratic  order of discriminant $D$, where $b,c\in \Z$
and $D=b^2-4c$. \medskip

If $P$ is  a polarized abelian variety, we can  then identify $P$ with
the quotient $\C^g/L$  where $L$ is a lattice  isomorphic to $\Z^{2g}$
equipped   with   a   symplectic   pairing   $\langle,\rangle$.    The
endomorphism ring $\mathrm{End}(P)$ of $P$ is then identified with the
set of complex  linear maps $T :\C^g\lra \C^g$  such that $T(L)\subset
L$. Recall  that an endomorphism is  said to be  {\em self-adjoint} if
for all  $x,y \in  L$ the relation  $\langle Tx,y \rangle=  \langle x,
Ty\rangle$ holds. \medskip

We say  that the variety $P$  admits {\em real  multiplication} by the
order $\Ord_D$, if $\dim_\C P=2$, and if there exists a representation
$\mathfrak{i} : \Ord_D \lra \mathrm{End}(P)$ which satisfies
\begin{enumerate}
\item \label{cond:1}  $\mathfrak{i}(\lambda)$ is self-adjoint  for any
  $\lambda \in \Ord_D$,
\item \label{cond:2} $\mathfrak{i}(\Ord_D)$  is a {\em proper subring}
  of $\mathrm{End}(P)$ {\it i.e.}  if $T\in \mathrm{End}(P)$ and $nT\in
  \mathfrak{i}(\Ord_D)$   where   $n   \in   \Z,\  n>0$   then   $T\in
  \mathfrak{i}(\Ord_D)$.
\end{enumerate}

\subsection{Prym eigenforms}

Let $P=\Prym(X,\rho)$ be a Prym variety.
We say that $P$ has real multiplication by the order $\Ord_D$ if there
exists a  representation $\mathfrak{i} :  \Ord_D \lra \mathrm{End}(P)$
which  satisfies  above conditions~(\ref{cond:1})  and~(\ref{cond:2}),
where the lattice $H_1(X,\Z)^{-}$  is equipped with the restriction of
the   intersection form  on $H_1(X,\Z)$. \medskip

Since $\rho$ acts on $\Omega(X)$,  it follows that we have a splitting
into a direct sum of  two eigenspaces: $\Omega(X) = \Omega(X)^+ \oplus
\Omega(X)^-$.  If  $P$  has  real multiplication  then  $\Ord_D$  acts
naturally on $\Omega(P)\cong\Omega(X)^-$. We  say that a non zero form
$\omega   \in   \Omega(X)^-$   is    a   {\em   Prym   eigenform}   if
$\Ord_D\cdot\omega \subset\C\omega$.

\subsection{Pseudo-Anosov homeomorphisms}
\label{ex:pA}

Real multiplication arises naturally with pseudo-Anosov homeomorphisms
commuting  with   $\rho$.   Let  $\phi   :  X  \rightarrow  X$   be  a
pseudo-Anosov affine with respect to  the flat metric given by $\omega
\in \Omega(X,\rho)^-$. Since $\phi$  commutes with $\rho$ it induces a
homomorphism
$$ T = \phi_{\ast}  + \phi_{\ast}^{-1} : H_1(X,\Z)^{-} \longrightarrow
H_1(X,\Z)^{-},
$$ that is self-adjoint.  Observe  that $T$ preserves the complex line
$S$   in    $(\Omega(X,\rho)^{-})^*$   spanned   by    the   dual   of
$\rm{Re}(\omega)$ and $\rm{Im}(\omega)$, and the restriction of $T$ to
this vector space is $\rm{Tr}(D\phi)\cdot \textrm{id}_{S}$.

Now the crucial assumption on the dimension comes into play.  Since $\dim_\C
\Omega(X,\rho)^{-} = 2$ one has $\dim_\C S^{\perp} = 1$.  Since $T$ is
self-adjoint,  it preserves  the  splitting $(\Omega(X,\rho)^{-})^*  =
S\oplus S^{\perp}$, acting by real scalar multiplication on each line,
hence  $T$ is  $\C$-linear,  {\it i.e.}  $T\in \mathrm{End}(P)$.  This
equips $\Prym(X,\rho)$  with the real multiplication  by $\Z[T] \simeq
\mathcal O_{D}$ for a convenient discriminant $D$. Since $T^\ast\omega
= \rm{Tr}(D\phi)\omega$, the form $\omega$ becomes an eigenform for this
real multiplication. \bigskip

We now summarize results on the moduli space of all forms, its stratification and the
action of $\GL^+(2,\R)$ upon it.

\subsection{Stratification of the space of Prym eigenforms}

As usual we denote  by $\Omega\mathfrak{M}_g$ the Abelian differential
bundle over the moduli space of Riemann surfaces of genus $g$, that is
the moduli space of pairs $(X,\omega)$, where $X$ is a Riemann surface
of  genus $g$, and  $\omega$ is  an Abelian  differential on  $X$. For
$g>1$ the natural stratification given by the orders of the zeroes of $\omega$
is denoted by:
$$   \Omega\mathfrak{M}_g=\bigsqcup_{\begin{array}{c}0<k_1\leq\dots\leq
    k_n,               \\               k_1+\dots+k_n=2g-2\end{array}}
\Omega\mathfrak{M}(k_1,\dots,k_n),
$$     where    $\Omega\mathfrak{M}(k_1,\dots,k_n)=\{(X,\omega)    \in
\Omega\mathfrak{M}_g,\  \hbox{the  zeroes   of  $\omega$  have  orders
  $(k_1,\dots,k_n)$}             \}$.              We            refer
to~\cite{Hubbard:Masur,Masur1982,Kontsevich1997,Veech1990}   for  more
details.  These   strata  are  not  necessarily   connected,  but  the
classification    has     been    obtained    by     Kontsevich    and
Zorich~\cite{Kontsevich2003}. \medskip

Let $\Omega  E_D \subset \Omega\mathfrak{M}_g$  (respectively, $\Omega
E_D(k_1,\dots,k_n) \subset  \Omega\mathfrak{M}(k_1,\dots,k_n)$) be the
space  of  Prym eigenforms  (respectively,  with  marked zeroes)  with
multiplication  by  $\Ord_D$.  Note   that  in  general,  neither  the
involution  $\rho$, nor  the  representation of  $\Ord_D$ is  uniquely
determined by  the eigenform $\omega$. We discuss  this uniqueness for
the case $\Omega E_D(4)$ in Section~\ref{sec:uniqueness}.

\begin{Example}
Coming back  to Example~\ref{ex:main}  one has $\sqrt{\pi^\ast  q} \in
\Omega\mathfrak{M}(4)$,  the underlying  Riemann surface  having genus
$3$. Thus, combining with Section~\ref{ex:pA}, this furnishes examples
in $\Omega E_{D}(4)$.
\end{Example}

\begin{Remark}
\label{rk:fixed:pts}
Applying Riemann-Hurwitz formula to the condition
$$ \dim_\C \Prym(X,\rho) = g(X)-g(X/\rho) = 2
$$ we  get that  $\Omega E_D=\emptyset$ unless  $ 2\leq g(X)  \leq 5$.
Moreover if $g(X)=5$ then $\rho$ has no fixed points {\it i.e.}  the
projection $X\lra Y=X/\rho$ is unramified.
\end{Remark}

\subsection{Dynamics on moduli spaces}

The group $\textrm{GL}^+(2,\R)$ acts on the set of translation surfaces
by postcomposition in the charts of the translation structure. 
The subgroup $\textrm{SL}(2,\R)$ preserves the area of a translation surface. The dynamics of 
the one-parameter diagonal subgroup of $\textrm{SL}(2,\R)$ has been studied by Masur~\cite{Masur1982}
 and Veech~\cite{Veech1982}. One important conjecture in Teichm\"uller dynamics is that the closure of
  any $\textrm{GL}^+(2,\R)$-orbit is an algebraic orbifold. This conjecture has been  proven for genus $g=2$ 
by McMullen~\cite{Mc3}. More recently, Eskin-Mirzakhani~\cite{Eskin:preprint} announced some important
 breakthrough toward the establishment   of this conjecture in the general case. \medskip

The strata $\Omega\mathfrak{M}(k_1,\dots,k_n)$ are obvious $\GL^+(2,\R)$-orbit closures, 
and the $\GL^+(2,\R)$-orbits of Veech surfaces are also closed~\cite{Smillie2006}

The  main important  fact about  the  Prym eigenforms  is that  they
furnishes new examples of closed $\textrm{GL}^+(2,\R)$-invariant subsets
of $\Omega\mathfrak{M}_{g}$. Namely one has:

\begin{Theorem}[McMullen~\cite{Mc7}]
\label{ThMcM1}
The locus $\Omega  E_D$ of Prym eigenforms for  real multiplication by
$\Ord_D$    is    a    closed,   $\GL^+(2,\R)$-invariant    subset    of
$\Omega\mathfrak{M}_g$.
\end{Theorem} 

\subsection{Weierstrass curves}

Following  McMullen, we  call the  locus $\Omega  E_D(2g-2)$  the {\em
  Weierstrass locus}.   Remark that the unique zero  of the eigenforms
in $\Omega  E_D(2g-2)$ must  be a fixed  point of $\rho$.   Since when
$g(X)=5$    the    involution    $\rho$    has   no    fixed    points
(Remark~\ref{rk:fixed:pts}),  it follows  that  $\Omega E_D(2g-2)$  is
non-empty only if $g=2,3,4$.  As a corollary of Theorem~\ref{ThMcM1}, a
dimension count gives

\begin{Corollary}\label{cor:McM1}
For $g=2,3,4$, the projection of  the Weierstrass locus to $\Mod_g$ is
a  finite union  of  Teichm\"uller  curves.  Each  of  such curves  is
primitive unless $D$ is a square.
\end{Corollary}

It turns out  that for surfaces $(X,\omega)\in \Omega\Mod(2g-2)$, if there exists a Prym 
involution $\rho$ such that $\dim_{\C}\Omega(X,\rho)^-=2$, and $\rho(\omega)=-\omega$, then the following 
are equivalent (see \cite{Mc3}):
\begin{enumerate}
\item $(X,\omega)\in \Omega E_D(2g-2)$.\medskip
\item There is a hyperbolic element in $\SL(X,\omega)$.\medskip
\item The group $\SL(X,\omega)$ is a lattice.\medskip
\end{enumerate}

Teichm\"uller  curves  in   $\Mod_2$  have  been  intensively  studied
(\cite{Mc1,Mc4,Mc7,Calta2004}).  The  situation  is  now  rather  well
understood.  The question  whether  or not  the  Weierstrass locus  is
connected has been raised in~\cite{Mc4} and solved for $g=2$:

\begin{Theorem}[McMullen~\cite{Mc4}]
\label{theo:mc:connexe}
For any integer $D\geq 5$ with $D \equiv 0 \textrm{ or } 1 \mod 4$
\begin{enumerate}
\item Either the Weierstrass locus $\Omega E_{D}(2)$ is connected, or
\item  $D  \equiv  1  \mod  8$  and  $D  \not  =  9$,  in  which  case
  $\Omega E_{D}(2)$ has exactly two components.
\end{enumerate}
For $D<5$, $\Omega_{D}(2) = \emptyset$.
\end{Theorem}

We are finally in a position to state precisely our results.
\subsection{Statements of the results}
\label{sec:statements}

\begin{Theorem}[Generic case]
\label{MainTh1}
Let $D\geq  17$ be  a discriminant. The  locus $\Omega E_D(4)$  is non
empty if and only if $D\not\equiv 5 \mod 8$. In this case, one has the
two possibilities:
\begin{enumerate}
\item Either $D$ is even, then the locus $\Omega E_D(4)$ consists of a
  single $\GL^+(2,\R)$-orbit,
\item or  $D$ is odd  then the locus  $\Omega E_D(4)$ consists  of two
  $\GL^+(2,\R)$-orbits.
\end{enumerate}
Moreover,  if  $D_1\neq   D_2$,  then  $\Omega  E_{D_1}(4)\cap  \Omega
E_{D_2}(4)=\ety$.
\end{Theorem}

\begin{Remark}\label{Rmk:MainTh1}\hfill
\begin{enumerate}
\item An important difference between $\Omega E_D(4)$ and $\Omega E_D(2)$ is the symplectic
 form of the $\Prym$ varieties. In the case $\Omega E_D(2)$, the $\Prym$ variety is the Jacobian 
 variety of a Riemann surface, therefore the symplectic form is given by 
 $\left(
\begin{smallmatrix}
  J & 0 \\
  0 & J 
\end{smallmatrix} \right)$, 
where $J=\left(%
\begin{smallmatrix}
  0 & 1 \\
  -1 & 0
\end{smallmatrix}\right)$, 
but in the case $\Omega E_D(4)$, the symplectic form of the $\Prym$ variety is given by the matrix 
$\left(
\begin{smallmatrix}
  J & 0 \\
  0 & 2J 
\end{smallmatrix}\right)$.
This difference is responsible for the non-existence of $\Omega E_D(4)$, when $D \equiv 5 \mod 8$.

\item We would like also to emphasis on the fact that, even though we have the same statement 
compared to Theorem~\ref{theo:mc:connexe} when $D\equiv 1 \mod 8$, {\it i.e.} $\Omega E_D(4)$
 has two components, the reason for this disconnectedness is different in the two cases. Roughly
  speaking, in our case, the two components correspond to two distinct complex lines in the space 
  $\Omega(X,\rho)^- \simeq H^1(X,\R)^-$ (see Section~\ref{sec:OddDisc}), but in the case $\Omega E_D(2)$,
   the two components correspond to the same complex line, they can only be distinguished by the spin 
   invariant (see \cite[Section 5]{Mc4}).

\item Using similar ideas, we obtain a partial classification of Teichm\"uller curves in $\Omega E_D(6)$. 
See Appendix, Section~\ref{appendix:other:strata} for more precise detail.

\end{enumerate}

\end{Remark}

There  are only $4$  admissible values for $D$  smaller than
$17$, namely $D \in \{8,9,12,16\}$.  For these small values of $D$ one
has:

\begin{Theorem}[Small discriminants]
\label{MainTh2}
\hfill
\begin{enumerate}
\item $\Omega E_9(4)=\Omega E_{16}(4)=\ety$.

\item     $\Omega     E_{12}(4)$      consists     of     a     single
  $\mathrm{GL}^+(2,\R)$-orbit,   the   associated  Teichm\"uller   curve
  having~$2$ cusps.

\item $\Omega E_8(4)$  consists of a single $\mathrm{GL}^+(2,\R)$-orbit,
  the  associated Teichm\"uller curve  having only one cusp.   Moreover, if
  $(X,\omega)\in\Omega  E_8(4)$   then  $(X,\omega)$  has   no  simple
  cylinders.
\end{enumerate}

\end{Theorem}

Theorem~\ref{MainTh2} is a direct consequence of the classification of cylinder decompositions in 
$\Omega E_D(4)$, its proof is given in Section~\ref{sec:proof:MainTh2}.

\begin{Remark}
In the appendix we prove a similar result  (Theorem~\ref{theo:H6}) for the Prym locus of eigenforms in genus $4$ 
with a single zero, namely $\Omega E_{D}(6)$.
\end{Remark}


\section{Cylinder decompositions of Prym eigenforms}
\label{sec:topology}

In this  section we give  a complete topological description of  the cylinder decompositions  of
 Prym eigenforms. A  good introduction to  the geometry of translation  surfaces
    is~\cite{Troyanov1986}; See also~\cite{Masur2002,Zorich:survey}.

\subsection{Cylinder decompositions}

Associated to any Abelian differential is a flat metric structure with
cone  type   singularities  whose  transition   maps  are  translation
$z\mapsto z+c$. The singularities of the flat metric structure are the
zeroes of the  holomorphic $1$-form.  On such surfaces,  a {\em saddle
  connection} is a geodesic  segment whose endpoints are singularities
(the endpoints  might coincide),  a {\em cylinder}  is an  open subset
isometric to $\R\times]0,h[/\Z$, where the action of $\Z$ is generated
    by $(x,y) \mapsto  (x+w,y), \; w >0,$ and  maximal with respect to
    this property,  $h$ and  $w$ are called  the {\em height}  and the
    {\em  width} of  the cylinder  (see~\cite{Hubbard:Masur}  for more
    details).   A  cylinder is  bounded  by  concatenations of  saddle
    connections  freely homotopic to  the waist  curve. Note  that, in
    general,   the  two  boundary   components  are   not  necessarily
    disjoint. If  each boundary  component of a  cylinder is  a single
    saddle connection, we say that the cylinder is {\em simple}.\medskip

For any direction $\theta\in \S^1$,  we have a flow on the translation
surface  whose trajectories are  geodesics in  this direction.  We say
that the flow  in direction $\theta$ is {\em  completely periodic } if
each trajectory  is either a  saddle connection or a  closed geodesic.
The  surface is then  the union  of finitely  many open  cylinder and
saddle connections in this direction. \medskip

We  say  that  the  flow  in  direction  $\theta$  is  {\em  uniformly
  distributed } if each  trajectory is dense and uniformly distributed
with respect to the natural Lebesgue measure on $\Sigma$.

\medskip

Observe that  surfaces that are completely periodic  in some direction
are very  rare in  a stratum. But  in the Prym  locus this is  the typical
case. Indeed, the surfaces in $\Omega E_D(2)$ and $\Omega E_D(1,1)$ are completely
 periodic, that is, if there is a closed geodesic in some direction $\theta$, then the surface is completely
  periodic in this direction (see \cite{Calta2004}, and \cite{Mc3}). Following~\cite{Mc7}, surfaces 
  $(X,\omega)\in   \Omega  E_D(4)$  are Veech surfaces {\em i.e.} the  Veech group 
   $\SL(X,\omega):=\textrm{Stab}_{\textrm{SL}(2,\R)}(X,\omega)$   
is  a lattice. Thus the central result from~\cite{Veech1989} applies:

\begin{Theorem}[Veech~\cite{Veech1989}]
Let $(X, \omega)$ be a Veech surface. Then for
any $\theta$:
\begin{enumerate}
\item Either the flow in direction $\theta$ is completely periodic, or
\item the flow in direction $\theta$ is uniformly distributed.
\end{enumerate}
\end{Theorem}

For the rest of this section let $(X,\omega) \in  \Omega E_{D}(4)$  be a  Prym eigenform  for some
discriminant $D$. Recall that $\rho: X\rightarrow X$ is a holomorphic involution of the 
genus $3$ Riemann surface $X$ and  $\omega$ is anti-invariant {\it i.e.} 
$\rho^*(\omega)=-\omega$. Let  $\Sigma$ denote the flat surface  associated to the pair 
$(X,\omega)$, then $\rho$ is an isometry  of $\Sigma$ whose  differential is $-\id$. Note  that the
unique singular  point of $\Sigma$  (which corresponds to the  zero of
$\omega$)  is   obviously  a  fixed   point  of  $\rho$   (compare  to Example~\ref{ex:main}).


%
%

\subsection{Topological classification of cylinder decompositions}

The next proposition furnishes a classification of topological configurations of cylinder decompositions
of Prym eigenforms.

\begin{Proposition}
\label{prop:decomp:cylinders}
Let  $(X,\omega) \in  \Omega E_{D}(4)$  be a  Prym eigenform  for some
discriminant $D$.  If the horizontal direction is completely
periodic then the  horizontal flow on $X$ decomposes  the surface into
cylinders following one of the following five models (models A+, A-, B, C, D):

\begin{itemize}

\item three  cylinders: one is fixed, two  are exchanged by
  the involution (see Figure~\ref{fig:3cyl:models}).

\item  two  cylinders  exchanged  by  the  involution  (see
  Figure~\ref{fig:2cyl:model}, left).

\item   one   cylinder  fixed   by   the  involution   (see
  Figure~\ref{fig:2cyl:model}, right).

\end{itemize}

\end{Proposition}

\begin{figure}[htbp]
\begin{minipage}[t]{0.3\linewidth}
\begin{tikzpicture}[scale=0.45]
\fill[fill=yellow!80!black!20,even  odd rule]  (0,0)  rectangle (2,2);
\draw[thick] (0,0) -- (0,2) --  (-3,2) -- (-2.5,3) -- (2.5,3) -- (2,2)
-- (2,0)  -- (5,0)  -- (4.5,-1)  -- (-0.5,-1)  --  cycle; \draw[thick]
(0,0) -- (2,0); \draw[thick] (0,2) -- (2,2); 
\draw[very thick, dashed] (0,2) -- (0.5,3) (1.5,-1) -- (2,0);

\draw  (-1,3) +(0,-0.1)  -- +(0,0.1)  (-1.5,2) +(0,-0.1)  -- +(0,0.1);
\draw   (3.5,0)   +(-0.05,-0.1)   --  +(-0.05,0.1)   +(0.05,-0.1)   -- +(0.05,0.1); 
\draw  (3,-1) +(-0.05,-0.1) --  +(-0.05,0.1) +(0.05,-0.1) -- +(0.05,0.1); 
 \draw (1.5,3) +(0,-0.1)  -- +(0,0.1)  +(-0.1,-0.1) -- +(-0.1,0.1)  +(0.1,-0.1) --  +(0.1,0.1); \draw  (0.5,-1)  +(0,-0.1) --
+(0,0.1) +(-0.1,-0.1) -- +(-0.1,0.1) +(0.1,-0.1) -- +(0.1,0.1);

\filldraw[fill= white, draw=black, inner sep=0pt] (0,0) circle (1.2mm)
(2,0)  circle (1.2mm)  (5,0) circle  (1.2mm) (-0.5,-1)  circle (1.2mm)
(1.5,-1) circle  (1.2mm) (4.5,-1) circle (1.2mm)  (0,2) circle (1.2mm)
(2,2)  circle (1.2mm)  (2.5,3) circle  (1.2mm) (0.5,3)  circle (1.2mm)
(-2.5,3)   circle  (1.2mm)   (-3,2)  circle   (1.2mm);   \draw  (1,-2)
node[below] {Model $A+$};
\end{tikzpicture}
\end{minipage}
%
\begin{minipage}[t]{0.3\linewidth}
\centering
\begin{tikzpicture}[scale=0.45]
\fill[fill=yellow!80!black!20,even odd  rule] (0,0) --  (5,0) -- (6,2)
-- (1,2) -- cycle;

\draw[thick] (0,0) --  (4,0) -- (4,-1) -- (5,-1) --  (5,0) -- (6,2) --
(2,2) -- (2,3) -- (1,3) -- (1,2) -- cycle; \draw[thick] (1,2) -- (2,2)
(4,0) -- (5,0);

\draw  (1.5,3)  +(0,-0.1)  --  +(0,0.1); \draw  (3.5,0)  +(0,-0.1)  --
+(0,0.1);

\draw   (2.5,2)   +(-0.05,-0.1)   --  +(-0.05,0.1)   +(0.05,-0.1)   --
+(0.05,0.1); \draw (4.5,-1) +(-0.05,-0.1) -- +(-0.05,0.1) +(0.05,-0.1)
-- +(0.05,0.1);

\draw  (4.5,2)  +(0,-0.1)  --  +(0,0.1)  +(-0.1,-0.1)  --  +(-0.1,0.1)
+(0.1,-0.1)  --  +(0.1,0.1);   \draw  (1.5,0)  +(0,-0.1)  --  +(0,0.1)
+(-0.1,-0.1) -- +(-0.1,0.1) +(0.1,-0.1) -- +(0.1,0.1);

\filldraw[fill=white, draw=black, inner  sep=0pt] (0,0) circle (1.2mm)
(3,0) circle (1.2mm) (4,0) circle (1.2mm) (4,-1) circle (1.2mm) (5,-1)
circle (1.2mm) (5,0) circle  (1.2mm) (6,2) circle (1.2mm) (3,2) circle
(1.2mm) (2,2) circle (1.2mm) (2,3) circle (1.2mm) (1,3) circle (1.2mm)
(1,2) circle (1.2mm); \draw (3,-2) node[below] {Model $A-$};
\end{tikzpicture}
\end{minipage}
%
\begin{minipage}[t]{0.3\linewidth}
\centering
\begin{tikzpicture}[scale=0.45]
\fill[fill=yellow!80!black!20,even odd  rule] (0,0) --  (4,0) -- (4,2)
-- (0,2) -- cycle; \draw[thick] (0,0) -- (1,0) -- (0.5,-1) -- (3.5,-1)
-- (4,0) --  (4,2) -- (3,2) --  (3.5,3) -- (0.5,3) --  (0,2) -- cycle;
\draw[thick] (0,2) -- (3,2) (1,0) -- (4,0);

\draw  (1,3)  +(0,-0.1)  --   +(0,0.1);  \draw  (0.5,0)  +(0,-0.1)  --
+(0,0.1);

\draw   (3.5,2)   +(-0.05,-0.1)   --  +(-0.05,0.1)   +(0.05,-0.1)   --
+(0.05,0.1); \draw  (3,-1) +(-0.05,-0.1) --  +(-0.05,0.1) +(0.05,-0.1)
-- +(0.05,0.1);

\draw  (2.5,3)  +(0,-0.1)  --  +(0,0.1)  +(-0.1,-0.1)  --  +(-0.1,0.1)
+(0.1,-0.1)  --  +(0.1,0.1);  \draw  (1.5,-1)  +(0,-0.1)  --  +(0,0.1)
+(-0.1,-0.1) -- +(-0.1,0.1) +(0.1,-0.1) -- +(0.1,0.1);

\filldraw[fill=white, draw=black, inner  sep=0pt] (0,0) circle (1.2mm)
(1,0) circle  (1.2mm) (0.5,-1) circle (1.2mm)  (2.5,-1) circle (1.2mm)
(3.5,-1)  circle (1.2mm)  (4,0)  circle (1.2mm)  (4,2) circle  (1.2mm)
(3,2)  circle (1.2mm)  (3.5,3) circle  (1.2mm) (1.5,3)  circle (1.2mm)
(0.5,3) circle (1.2mm) (0,2)  circle (1.2mm); \draw (2,-2) node[below]
         {Model $B$};

\end{tikzpicture}
\end{minipage}
\caption{
\label{fig:3cyl:models}
Three-cylinder decompositions for  periodic directions  on Prym
eigenforms (the cylinder fixed by $\rho$ is colored in gray).  }
\end{figure}
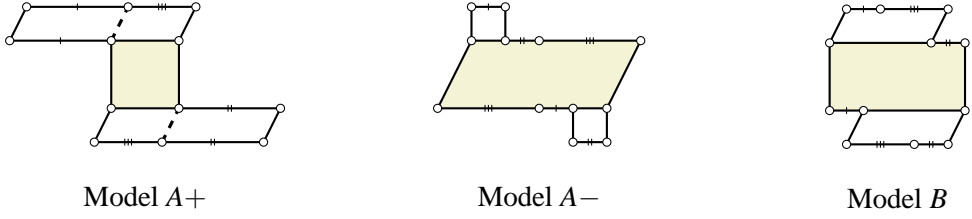

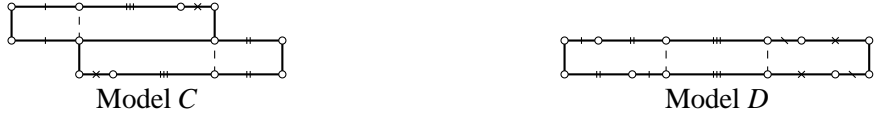
\begin{figure}[htbp]
\begin{minipage}[t]{0.4\linewidth}
\begin{tikzpicture}[scale=0.45]

\draw[thick] (0,0)  -- (6,0) -- (6,1)  -- (4,1) -- (4,2)  -- (-2,2) --
(-2,1)  -- (0,1) --  cycle; \draw[thick]  (0,1) --  (4,1); \draw[thin,
  dashed] (0,1) -- (0,2) (4,0) -- (4,1);

\filldraw[fill=white,  draw=black]  (0,0)  circle (3pt)  (1,0)  circle
(3pt) (4,0) circle  (3pt) (6,0) circle (3pt) (6,1)  circle (3pt) (4,1)
circle (3pt) (4,2) circle (3pt)  (3,2) circle (3pt) (0,2) circle (3pt)
(-2,2) circle (3pt) (-2,1) circle (3pt) (0,1) circle (3pt);

\draw  (-1,2)  +(0,-0.1)  --   +(0,0.1);  \draw  (-1,1)  +(0,-0.1)  --
+(0,0.1);

\draw (5,1) +(-0.05,-0.1) -- +(-0.05,0.1) +(0.05,-0.1) -- +(0.05,0.1);
\draw (5,0) +(-0.05,-0.1) -- +(-0.05,0.1) +(0.05,-0.1) -- +(0.05,0.1);

\draw (3.5,2)  +(-0.1,-0.1) -- +(0.1,0.1)  +(0.1,-0.1) -- +(-0.1,0.1);
\draw (0.5,0) +(-0.1,-0.1) -- +(0.1,0.1) +(0.1,-0.1) -- +(-0.1,0.1);

\draw  (1.5,2)  +(0,-0.1)  --  +(0,0.1)  +(-0.1,-0.1)  --  +(-0.1,0.1)
+(0.1,-0.1)  --  +(0.1,0.1);   \draw  (2.5,0)  +(0,-0.1)  --  +(0,0.1)
+(-0.1,-0.1) -- +(-0.1,0.1) +(0.1,-0.1) -- +(0.1,0.1);

\draw (2,-1.3) node[above] {Model $C$};

\end{tikzpicture}
\end{minipage}
%
\begin{minipage}[t]{0.4\linewidth}
\centering
\begin{tikzpicture}[scale=0.45]
\draw[thick] (0,0)  -- (9,0) --  (9,1) -- (0,1) --  cycle; \draw[thin,
  dashed] (3,0) -- (3,1) (6,0) -- (6,1);

\filldraw[fill=white,  draw=black]  (0,0)  circle (3pt)  (2,0)  circle
(3pt) (3,0) circle  (3pt) (6,0) circle (3pt) (8,0)  circle (3pt) (9,0)
circle (3pt) (9,1) circle (3pt)  (7,1) circle (3pt) (6,1) circle (3pt)
(3,1) circle (3pt) (1,1) circle (3pt) (0,1) circle (3pt);

\draw  (0.5,1)  +(0,-0.1)  --  +(0,0.1); \draw  (2.5,0)  +(0,-0.1)  --
+(0,0.1);

\draw (2,1) +(-0.05,-0.1) -- +(-0.05,0.1) +(0.05,-0.1) -- +(0.05,0.1);
\draw (1,0) +(-0.05,-0.1) -- +(-0.05,0.1) +(0.05,-0.1) -- +(0.05,0.1);

\draw  (8,1) +(-0.1,-0.1)  -- +(0.1,0.1)  +(0.1,-0.1)  -- +(-0.1,0.1);
\draw (7,0) +(-0.1,-0.1) -- +(0.1,0.1) +(0.1,-0.1) -- +(-0.1,0.1);

\draw (6.5,1) +(-0.1,0.1) -- +(0.1,-0.1); \draw (8.5,0) +(-0.1,0.1) --
+(0.1,-0.1);

\draw  (4.5,1)  +(0,-0.1)  --  +(0,0.1)  +(-0.1,-0.1)  --  +(-0.1,0.1)
+(0.1,-0.1)  --  +(0.1,0.1);   \draw  (4.5,0)  +(0,-0.1)  --  +(0,0.1)
+(-0.1,-0.1) -- +(-0.1,0.1) +(0.1,-0.1) -- +(0.1,0.1);

\draw (4.5,-1.3) node[above] {Model $D$};
\end{tikzpicture}
\end{minipage}
\caption{
\label{fig:2cyl:model}
Two-cylinder decomposition (the cylinders  must  be exchanged by $\rho$) on the 
left, and one-cylinder decomposition on the right.
}
\end{figure}

The proof of the proposition will use the following lemma, easily derived
from the Riemann-Hurwitz's formula

\begin{Lemma}\label{Lm1}
Let $X$ be a  Riemann surface of genus $3$, and $\rho:  X \lra X$ be a
holomorphic  involution. Suppose  that  $\dim_\C\Prym(X,\rho)=2$, then
$\rho$ has exactly $4$ fixed points.
\end{Lemma}

\begin{proof}[Proof of Proposition~\ref{prop:decomp:cylinders}]
Since   the   cone  angle   at   the   singularity   of  $\Sigma$   is
$(4+1)2\pi=10\pi$,   there   are   exactly  $5$   horizontal   saddle
connections.  Each  of these saddle  connections appears in  the lower
boundary of a unique cylinder, thus  we have a partition of the set of
horizontal saddle  saddle connections into  $k$ subsets, where  $k$ is
the number of cylinders. Clearly we have $k\leq 5$. Note that a saddle
connection can not be the upper  boundary of a cylinder, and the lower
boundary of another cylinder, since  this would imply that this saddle
connection  is  actually  a   simple  closed  geodesic  containing  no
singularities.\medskip

Since $\rho$  is an isometry, it  sends a cylinder  isometrically to a
cylinder,  therefore  $\rho$  induces  a  permutation on  the  set  of
cylinders.  As $D\rho=-\id$,  $\rho$  sends the  lower  boundary of  a
cylinder to the  upper boundary of another cylinder,  hence a cylinder
which is invariant by $\rho$  contains exactly two fixed points in its
interior. Recall that the singularity of $\Sigma$ is already one fixed
point, thus there are a most one cylinder invariant by $\rho$. 

\begin{enumerate}

\item If $k=5$ or $k=4$,  then there always exists a saddle connection
  which is the lower boundary  of one cylinder, and the upper boundary
  of another one, therefore these cases are excluded.\medskip

\item If $k=3$, then $\rho$  preserves one cylinder, and exchanges the
  other  two.  Let  $C_0$ be  the  cylinder invariant  by $\rho$,  and
  $C_1,C_2$  the two  permuted  cylinders.  Let  $n_0$  be the  number
  saddle connections  contained in the lower boundary  of $C_0$, since
  the  upper boundary  and lower  boundary of  $C_0$ are  exchanged by
  $\rho$,  there  are  also  $n_0$  saddle connections  in  the  upper
  boundary of  $C_0$. Note  also that the  lower boundary of  $C_1$ is
  mapped onto the upper boundary of $C_2$ and vice versa.

\begin{enumerate}
\item Case $n_0=1$:  in this case $C_0$ is a  simple cylinder, and
  the lower boundaries of both $C_1$ and $C_2$ must contain two saddle
  connections. The corresponding configuration is given by Model $A+$.
\item Case  $n_0=2$: in this case,  none of the  cylinders are simple,
  and there is also only  one possible configuration which is given by
  Model $B$.
\item  Case $n_0=3$: in  this case,  both $C_1,  C_2$ are  simple, and
  the unique possible configuration is given by Model $A-$.
\end{enumerate}

\item If $k=2$,  then the two cylinders are  permuted by $\rho$. Since
  the  number of  saddle  connections  in the  lower  boundary of  one
  cylinder is the same as the number of those in the upper boundary of
  the other  one, it follows that  the partition of the  set of saddle
  connections must  be $\{2,3\}$ (otherwise,  there would be  a saddle
  connection which is a lower  boundary of one cylinder, and the upper
  boundary  of the  other  one).  Hence, there  is  only one  possible
  configuration which is given by Model $C$. \medskip

\item If $k=1$,  both of the lower and upper  boundaries of the unique
  cylinder contain $5$ saddle connections. Observe that $\rho$ induces
  a permutation  on the  set of saddle  connections.  Since  there are
  already two fixed points of  $\rho$ in the interior of the cylinder,
  there  is  only  one fixed  point  in  the  interior of  the  saddle
  connections,  which  means  that   only  one  saddle  connection  is
  invariant  by $\rho$.  Therefore,  $\rho$ must  preserve one  saddle
  connection,  and  exchange the  other two  pairs.  Again,  there is  one
  possible configuration, which is given by Model $D$.

\end{enumerate}

Proposition~\ref{prop:decomp:cylinders} is now proved.
\end{proof}

An immediate consequence of Proposition~\ref{prop:decomp:cylinders} is

\begin{Corollary}
\label{cor:decomp:3cyl}
For  any  Abelian  differential  in  the locus  $\Omega  E_D(4)$,  the
associated flat surface admits a three-cylinder decomposition.
\end{Corollary}

\begin{proof}
Let $\Sigma$ be the flat surface associated to an Abelian differential
in $\Omega E_D(4)$. By Corollary~\ref{cor:McM1}, we know that $\Sigma$
is  a Veech surface,  therefore it  admits infinitely  many completely
periodic directions.  Without loss of  generality, we can  assume that
the  horizontal direction  is completely  periodic for  $\Sigma$. From
Proposition~\ref{prop:decomp:cylinders}, we only  have to consider the
cases  $C$  and $D$  where  $\Sigma$ is  decomposed  into  one or  two
cylinders. But in  those cases, one can easily  find a simple cylinder
in another  direction $\theta\neq (1,0)\in \S^1$. Since  $\Sigma$ is a
Veech surface, it  is also decomposed into cylinders  in the direction
$\theta$. But since there is at  least one simple cylinder in that direction, 
the new decomposition must  belong to the cases $A+$ or $A-$.
\end{proof}

\begin{Remark}
It turns out (see Proposition~\ref{prop:D8:topo}) that for  all but  one value  of $D$,  the  surfaces in
$\Omega E_D(4)$ always admit a cylinder decomposition in Model $A\pm$.
\end{Remark}

\section{Prototypes}
\label{sec:prototypes}

The main goal of this section is to provide a canonical representation
of any three-cylinder decomposition of a surface in $\Omega  E_{D}(4)$ in
terms    of    {\it    prototype}     (up    to    the    action    of
$\textrm{GL}^+(2,\R)$).   More  precisely,   for  each such decomposition (see 
Proposition~\ref{prop:decomp:cylinders}) we  will attach parameters satisfying 
some specific conditions, which provide a necessary and sufficient condition  to be
a  surface in  $\Omega  E_{D}(4)$. As a  consequence, we derive the following finiteness result.

\begin{Theorem}
\label{theo:onto:map}
Let  $D$ be  a fixed  positive integer.  Let $\mathcal  Q_{D}$  be the
(finite) set of tuples $(w,h,t,e,\varepsilon)\in \Z^5$ satisfying
$$   \left\{\begin{array}{l}   w>0,   \   h>0,\  \varepsilon   =   \pm
1,\\  e+2h<w,\ 0\leq  t<\gcd(w,h),\\ \gcd(w,h,t,e)=1,  \textrm{  and }
D=e^2+8w h.\\
\end{array}
\right.
$$ If $D\not =  8$ then there is an onto map  from $\mathcal Q_{D}$ on
the components of $\Omega E_{D}(4)$.
\end{Theorem}

\subsection{Normalizing  cylinder decompositions}

Recall  that Corollary~\ref{cor:decomp:3cyl} tells  us that  a surface
$(X,\omega) \in \Omega E_D(4)$ always admits a cylinder decomposition
into Model $A+$, $A-$ or $B$. We will examine separately each of the three cases.

\begin{Notation}
For all $\gamma \in H_1(X,\Z)$ we set $\omega(\gamma):=\int_\gamma\omega$.
\end{Notation}

\subsubsection{Cylinder decompositions of type $A+$} \hfill \medskip

\begin{Proposition}
\label{NormA1Prop}
Let $(X,\omega) \in \Omega E_D(4)$  be a Prym eigenform which admits a
cylinder       decomposition      in       Model       $A+$.      Let
$\alpha_1,\beta_1,\alpha_{2,1},\beta_{2,1},\alpha_{2,2},    \beta_{2,2}
\in   H_1(X,\Z)$    be   a   symplectic   basis    as   presented   in
Figure~\ref{fig:modelA+:basis}.                 We                 set
$\alpha_2:=\alpha_{2,1}+\alpha_{2,2}$                               and
$\beta_2:=\beta_{2,1}+\beta_{2,2}$. Then
\begin{itemize}

\item[(i)]
There exists  a unique generator $T$ of $\Ord_D$ which is written  
in  the  basis   $\{\alpha_1,\beta_1,\alpha_2,\beta_2\}$  by a matrix of the from $\DS{\left(\begin{smallmatrix}
e\cdot\rm{id}_2 & 2B\\ B^* & 0\\
\end{smallmatrix}\right)}$, $B\in \mathbf{M}_{2\times 2}(\Z)$, such that 
$T^*(\omega)=\lambda(T)\omega$ with $\lambda(T)>0$.


\item[(ii)]
Up    to    the   $\mathrm{GL}^+(2,\R)$-action and Dehn twists 
$\beta_1\mapsto \beta_1+n\alpha_1$, $\beta_{2,i}\mapsto \beta_{2,i}+m\alpha_{2,i}$, $n,m\in\Z$, there exist
 $w,h,t \in \N$ such that the tuple $(w,h,t,e)$  satisfies

$$
(\mathcal{P})\left\{\begin{array}{l}        w>0,h>0,\;        0\leq
  t<\gcd(w,h),\\    \gcd   (w,h,t,e)    =1,\\    D=e^2+8w   h,\\    0<
  \lambda:=\frac{e+\sqrt{D}}{2}<w\\
\end{array}
\right.,
$$
and the matrix of $T$ is  given by $\left(%
\begin{smallmatrix}
  e & 0 & 2w & 2t \\ 0 & e & 0 & 2h \\ h & -t & 0 & 0 \\ 0 & w & 0 & 0  \\
\end{smallmatrix}%
\right)$. Moreover, in these coordinates we have

\begin{equation}\label{normalize:A+}
\left\{ \begin{array}{l}
  \omega(\Z\alpha_1+\Z\beta_1)=\lambda\cdot\Z^2
  \\ \omega(\Z\alpha_{2,1}+\Z\beta_{2,2})=\omega(\Z\alpha_{2,2}+\Z\beta_{2,2})=\Z(w,0)+\Z(t,h)
\end{array}
\right. 
\end{equation}
\end{itemize}
Conversely,  let $(X,\omega)$ be an  Abelian differential in
$\Omega\mathfrak{M}(4)$ having a decomposition into three cylinders in
model $A+$.  Suppose that there exists $(w,h,t,e)  \in \Z^4$ verifying
$(\mathcal{P})$  such  that,  after  normalizing by  $\GL^+(2,\R)$,  the
conditions in $(\ref{normalize:A+})$  are satisfied, then $(X,\omega)$
belongs to $\Omega E_D(4)$.
\end{Proposition}

\begin{figure}[htbp]
\centering \subfloat{%
\begin{tikzpicture}[scale=0.8]
\fill[fill=yellow!80!black!20,even  odd rule]  (0,0)  rectangle (2,2);
\draw (0,0)  -- (0,2)  -- (-3,2)  -- (-2.5,3) --  (2.5,3) --  (2,2) --
(2,0) -- (5,0) -- (4.5,-1) -- (-0.5,-1) -- cycle; \draw (0,0) -- (2,0)
(0,2) -- (2,2);

\draw[->,>= angle 45, thick,  dashed] (0,0.5) -- (1,0.5); \draw[thick,
  dashed]  (1,0.5) --  (2,0.5); \draw[thick,  dashed, ->,  >=angle 45]
(2,0)  .. controls (1.5,0.5)  and (1.5,  1.5) ..   (2,2); \draw[thick,
  dashed,  ->,  >=  angle  45] (-2.75,2.5)  --  (1,2.5);  \draw[thick,
  dashed] (1,2.5) -- (2.25,2.5); \draw[thick, dashed, ->, >= angle 45]
(-1.5,2) -- (-1,3); \draw[thick,  dashed, ->, >= angle 45](-0.25,-0.5)
-- (2,-0.5);    \draw[thick,   dashed]   (2,-0.5)    --   (4.75,-0.5);
\draw[thick, dashed, ->, >= angle 45] (3,-1) -- (3.5,0);

\filldraw[fill=white,  draw=black]  (0,0)  circle (2pt)  (2,0)  circle
(2pt) (5,0) circle (2pt)  (-0.5,-1) circle (2pt) (1.5,-1) circle (2pt)
(4.5,-1) circle  (2pt) (0,2) circle  (2pt) (2,2) circle  (2pt) (2.5,3)
circle (2pt) (0.5,3) circle  (2pt) (-2.5,3) circle (2pt) (-3,2) circle
(2pt);

\draw  (0,0.5) node[left]  {$\scriptstyle \alpha_1$}  (1.25,1.25) node
      {$\scriptstyle \beta_1$}  (2.25, 2.5) node[right] {$\scriptstyle
        \alpha_{1,1}$} (-1,3) node[above] {$\scriptstyle \beta_{1,1}$}
      (-0.25,-0.5)  node[left]  {$\scriptstyle \alpha_{2,1}$}  (3.5,0)
      node[above] {$\scriptstyle \beta_{2,1}$};

\end{tikzpicture}
}
\caption{Basis
  $\{\alpha_1,\beta_1,\alpha_{2,1},\beta_{2,1},\alpha_{2,2},
  \beta_{2,2}\}$   of   $H_1(X,\Z)$    associated   to   a   cylinder
  decomposition  in Model  $A+$  (the fixed  cylinder  is colored  in
  grey).       If       $\alpha_2:=\alpha_{2,1}+\alpha_{2,2}$      and
  $\beta_2:=\beta_{2,1}+\beta_{2,2}$,                            then
  $\{\alpha_1,\beta_1,\alpha_2,\beta_2\}$     is     a  symplectic  basis     of
  $H_{1}(X,\Z)^{-}$.  }
\label {fig:modelA+:basis}
\end{figure}
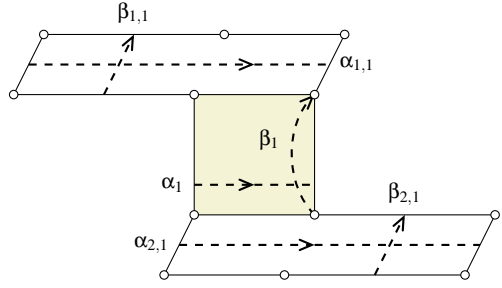

\begin{Remark}\label{Rk:condition:P}
Since $\lambda$  is the positive root of  the polynomial $X^2-eX-2wh$,
the   condition  $\lambda   <  w$  can be read $w^2-ew-2wh>0$, or equivalently 
$e+2h<w$.
\end{Remark}

For the proof of Proposition~\ref{NormA1Prop}, we need the following straightforward lemma

\begin{Lemma}\label{Lm2}
Let  $P\cong  \C^2/L$ be  a  polarized  Abelian  variety of  dimension
$2$.  Suppose  that  $L=L_1\oplus  L_2$, where  $L_i\cong  \Z^2$,  and
$L_1^\perp=L_2$ with respect to the symplectic form $\langle, \rangle$
on  $L$.  Let $(a_i,b_i),\;  i=1,2,$  be a  basis  of  $L_i$, and  set
$\langle    a_i,b_i\rangle=\mu_i\in   \N\setminus\{0\}$.    If   $T\in
\mathrm{End}(P)$ is self-adjoint, then the  matrix of $T$ in the basis
$(a_1,b_1,a_2,b_2)$ is given by

$$T=\left(%
\begin{array}{cc}
  e\cdot\id_2 & B \\ \frac{\mu_1}{\mu_2}B^* & f\cdot \id_2 
\end{array}%
\right),
$$
with    $e,f   \in   \Z,    B,  \;  \frac{\mu_1}{\mu_2}B^*\in
\mathbf{M}_{2\times 2}(\Z)$, where $\displaystyle{\left(%
\begin{smallmatrix}
  a & b \\ c & d \\
\end{smallmatrix}%
\right)^* =\left(%
\begin{smallmatrix}
  d & -b \\ -c & a \\
\end{smallmatrix}%
\right)}$.
\end{Lemma}

%
%
%

\begin{proof}[Proof of Proposition~\ref{NormA1Prop}]
Since $\rho$ permutes  $\alpha_{2,1}$ and $-\alpha_{2,2}$ and permutes
$\beta_{2,1}$  and  $-\beta_{2,2}$,  we have  $\alpha_2,\beta_2\in
H_1(X,\Z)^-$.  Thus we have  a splitting  $H_1(X,\Z)^-=L_1\oplus L_2,$
where       $L_i=\Z\alpha_i+\Z\beta_i$.       In       the       basis
$(\alpha_1,\beta_1,\alpha_2,\beta_2)$,    the   restriction    of   the
intersection form is given by the matrix $\displaystyle{\left(%
\begin{smallmatrix}
  J & 0 \\ 0 & 2J \\
\end{smallmatrix}%
\right)}$.


Let  $T$  be a  generator  of  $\Ord_D$,  since $T$  is  self-adjoint,
Lemma~\ref{Lm2}       implies      that,       in       the      basis
$(\alpha_1,\beta_1,\alpha_2,\beta_2)$, the matrix of $T$ has the form
$$ T=\left(%
\begin{array}{cccc}
  e & 0 & 2w & 2t \\ 0 & e & 2c & 2h \\ h & -t & f & 0 \\ -c & w & 0 &
  f\\
\end{array}%
\right)_{(\alpha_1,\beta_1,\alpha_2,\beta_2)}
$$ for some $w,h,t,c,e,f \in \Z$.  By replacing $T$ by $T-f$, which is
still  a  generator of  $\Ord_D$, we can assume that $f=0$. Since $\omega$ is an eigenform,
 we have $T^*(\omega)=\lambda(T)\omega$. 
Using the fact that $(\omega(\alpha_1),\omega(\beta_1))$ is a basis of $\R^2$, it is straightforward to 
verify that $\lambda(T)\neq 0$. Thus, by changing  the  sign  of $T$  if necessary, we can  assume that
 $\lambda(T)>0$.   The uniqueness of $T$ follows immediately from the fact that any generator of $\Ord_D$ can 
  be written as $a\cdot T+b\cdot\rm{id}_4, \; a,b \in \Z$.

Using  $\mathrm{GL}^+(2,\R)$,  we can  assume  that $\omega(\alpha_1)=(\lambda,0),  \omega(\beta_1)=(0,\lambda)$.  In  the
basis       $(\alpha_1,\beta_1,\alpha_2,\beta_2)$,       we       have
$\mathrm{Re}(\omega)=(\lambda,0,x,y),
\mathrm{Im}(\omega)=(0,\lambda,0,z)$,    with   $x>0,    z>0$.   Since
$T^*(\omega)=\lambda\omega$, it follows

\begin{equation}\label{eq1}
(\lambda,0,x,y)\cdot T=\lambda(\lambda,0,x,y)
\end{equation}
and
\begin{equation}\label{eq2}
(0,\lambda,0,z)\cdot T=\lambda(0,\lambda,0,z)
\end{equation}

From~(\ref{eq1})   we  draw   $x=2w$,   and  $y=2t$,   and
from~(\ref{eq2}), we  draw $c=0$, and $z=2h$. We  deduce in particular
that  $w>0,  h>0$.   We  can  renormalize further  using  Dehn  twists
$\beta_1\mapsto       n\alpha_1+\beta_1$      and      $\beta_2\mapsto
m\alpha_2+\beta_2$ so that $0\leq t<\gcd(w,h)$. Properness of $\Ord_D$
implies $\gcd(w,h,t,e)=1$.\medskip

Remark that $T$ satisfies
\begin{equation}
\label{GenEq}
T^2=eT+2w h\textrm{Id}_{\R^{4}}
\end{equation} 
Therefore   $\lambda$  satisfies  the   quadratic  equation
$\lambda^2-e\lambda-2w h=0$.  Moreover, since $T$  generates $\Ord_D$,
Equation~(\ref{GenEq}) implies that  $D=e^2+8wh$. Since $\lambda$ is a
positive       algebraic        number,       we       must       have
$\displaystyle{\lambda=\frac{e+\sqrt{D}}{2}}$. By construction we have
$0<\lambda<w$.   All  the  conditions   of  $(\mathcal{P})$   are  now
fulfilled.\medskip

Conversely, if $(w,h,t,e)$ satisfies $(\Pcal)$, and all the conditions
in  $(\ref{normalize:A+})$ hold  then the  construction using  model $A+$  gives  us an
Abelian  differential $(X,\omega)$  in  $\Omega\mathfrak{M}(4)$, which
admits    an     involution    $\rho    :X     \ra    X$    satisfying
$\dim_\C\Omega(X,\rho)^-=2$   (since   $H_1(X,\Z)^-\cong  \Z^4$)   and
$\omega  \in \Omega(X,\rho)^-$. The  endomorphism $T:  H_1(X,\Z)^- \longrightarrow
H_1(X,\Z)^-$  constructed as  above is  clearly self-adjoint,  and its
restriction      to     complex     line      $S=\C\cdot\omega$     is
$\lambda\cdot \rm{Id}_S$. Let  $S'=S^\perp$ be the  orthogonal complement of
$S$ in $\Omega(X,\rho)^-$ with  respect to the intersection form, then
$S'$ is also a complex line in $\Omega(X,\rho)^-$. Since $T$ satisfies
Equation~(\ref{GenEq}),   the   restriction   of   $T$  to   $S'$   is
$\lambda'\cdot\rm{Id}_{S'}$,  where $\lambda'$  is  the other  root of  the
polynomial $X^2-eX-2wh$ (note that $\lambda'<0$). Consequently, $T$ is
a $\C$-linear endomorphism of  $\Omega(X,\rho)^-$, that is $T$ belongs
to     $\rm{End}(\Prym(X,\rho))$.      Since     the    subring     of
$\rm{End}(\Prym(X,\rho))$ generated by  $T$ is isomorphic to $\Ord_D$,
this completes the proof of the proposition.
\end{proof}

\subsubsection{Cylinder decompositions of type $A-$}

The next result parallels Proposition~\ref{NormA1Prop}.

\begin{Proposition}\label{NormA2Prop}
Let $(X,\omega)$  be an Abelian differential in  $\Omega E_D(4)$ which
admits  a decomposition  into  cylinders in  the horizontal  direction
in              Model                $A-$.                Let
$\alpha_{1,1},\beta_{1,1},\alpha_{1,2},\beta_{1,2},\alpha_2,\beta_2
\in H_1(X,\Z)$  be as in Figure~\ref{fig:modelA-:basis}  below. We set
$\alpha_1=\alpha_{1,1}+\alpha_{1,2}, \beta_1=\beta_{1,1}+\beta_{1,2}$.
Then  
\begin{itemize}
\item[(i)] There exists a unique generator $T$ of $\Ord_D$ which is written 
in  the  basis  $\{\alpha_1,\beta_1,\alpha_2,\beta_2\}$ by the matrix
  $\left(%
\begin{smallmatrix}
  e\cdot\rm{id}_2 & B \\ 2B^* & 0\\
\end{smallmatrix}%
\right)$ such that $T^*(\omega)=\lambda(T)\omega$ with $\lambda(T)>0$.

\item[(ii)]  Up   to  the action $\mathrm{GL}^+(2,\R)$  and  Dehn   twists,  there  exist
$w,h,t \in \N$ such that the tuple  $(w,h,t,e)$ satisfies  condition  $(\mathcal{P})$ of
Proposition~\ref{NormA1Prop}, and the matrix of $T$ is given by $\DS{\left(
\begin{smallmatrix}
e & 0 & w & t\\ 0 & e & 0 & h \\ 2h & -2t & 0 & 0 \\ 0 & 2w & 0 & 0\\
\end{smallmatrix}
\right)}$. Moreover, in these coordinates we have

\begin{equation}\label{normalize:A-} 
\left\{ \begin{array}{l}
  \omega(\Z\alpha_2+\Z\beta_2)=\Z(w,0)+\Z(t,h)
  \\ \omega(\Z\alpha_{1,1}+\Z\beta_{1,1})=\omega(\Z\alpha_{1,2}+\Z\beta_{1,2})=\frac{\lambda}{2}\cdot
  \Z^{2}
\end{array}
\right.
\end{equation}
\end{itemize}
Conversely,  let $(X,\omega)$ be an  Abelian differential in
$\Omega\mathfrak{M}(4)$ having a decomposition into three cylinders in
model $A-$.  Suppose that there exists $(w,h,t,e)  \in \Z^4$ verifying
$(\Pcal)$,  such  that  after  normalizing  by  $\GL^+(2,\R)$,  all  the
conditions in $(\ref{normalize:A-})$  are satisfied, then $(X,\omega)$
belongs to $\Omega E_D(4)$.
\end{Proposition}

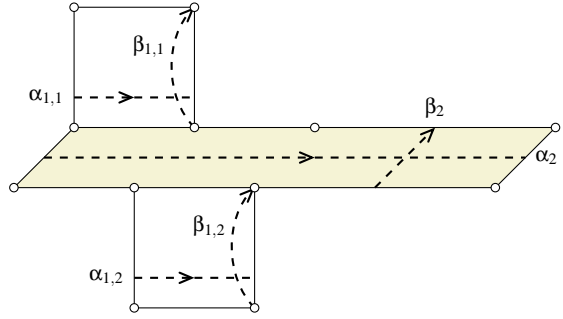
\begin{figure}[htbp]
\centering \subfloat{%
\begin{tikzpicture}[scale=0.8]
\fill[fill=yellow!80!black!20,even odd  rule] (0,0) --  (8,0) -- (9,1)
-- (1,1); \draw (0,0)  -- (2,0) -- (2,-2) -- (4,-2)  -- (4,0) -- (8,0)
-- (9,1) -- (3,1) -- (3,3) --  (1,3) -- (1,1) -- cycle; \draw (1,1) --
(3,1) (2,0) -- (4,0);

\draw[thick, dashed, ->, >= angle 45] (1,1.5) -- (2,1.5); \draw[thick,
  dashed] (2,1.5)  -- (3,1.5); \draw[thick,  dashed, ->, >=  angle 45]
(3,1) .. controls (2.5,1.5) and (2.5,2.5) .. (3,3);

\draw[thick,  dashed,   ->,  >=   angle  45]  (0.5,0.5)   --  (5,0.5);
\draw[thick, dashed] (5,0.5) -- (8.5,0.5); \draw[thick, dashed, ->, >=
  angle 45] (6,0) -- (7,1);

\draw[thick,  dashed,   ->,  >=   angle  45]  (2,-1.5)   --  (3,-1.5);
\draw[thick, dashed] (3,-1.5) -- (4,-1.5); \draw[thick, dashed, ->, >=
  angle 45] (4,-2) .. controls (3.5,-1.5) and (3.5,-0.5) .. (4,0);

\filldraw[fill=white,draw=black] (0,0) circle (2pt) (2,0) circle (2pt)
(2,-2)  circle (2pt)  (4,-2)  circle (2pt)  (4,0)  circle (2pt)  (8,0)
circle (2pt) (9,1) circle (2pt)  (5,1) circle (2pt) (3,1) circle (2pt)
(3,3) circle (2pt) (1,3) circle (2pt) (1,1) circle (2pt);

\draw   (1,1.5)  node[left]  {$\scriptstyle   \alpha_{1,1}$}  (2.25,2)
node[above]    {$\scriptstyle   \beta_{1,1}$}    (2,-1.5)   node[left]
{$\scriptstyle  \alpha_{1,2}$}  (3.25,-1)  node[above]  {$\scriptstyle
  \beta_{1,2}$} (8.5,0.5)  node[right] {$\scriptstyle \alpha_2$} (7,1)
node[above] {$\scriptstyle \beta_2$};
\end{tikzpicture}
}
\caption{Basis
  $\{\alpha_{1,1},\beta_{1,1},\alpha_{1,2},\beta_{1,2},\alpha_{2},
  \beta_{2}\}$ of $H_1(X,\Z)$ associated to a cylinder decomposition
  in  Model  $A-$  (the  fixed   cylinder  is  colored  in  grey).  If
  $\alpha_1:=\alpha_{1,1}+\alpha_{1,2}$                             and
  $\beta_1:=\beta_{1,1}+\beta_{1,2}$ ,                              then
  $\{\alpha_1,\beta_1,\alpha_2,\beta_2\}$     is     a    symplectic basis     of
  $H_{1}(X,\Z)^{-}$.  }
\label {fig:modelA-:basis}
\end{figure}

\begin{proof}[Proof of Proposition~\ref{NormA2Prop}]
We        have        $H_1(X,\Z)^-=L_1\oplus        L_2$,        where
$L_i=\Z\alpha_i+\Z\beta_i$.           In           the           basis
$(\alpha_1,\beta_1,\alpha_2,\beta_2)$, the  intersection form is given
by the matrix $\displaystyle{\left(%
\begin{array}{cc}
  2J & 0 \\ 0 & J\\
\end{array}%
\right)}$.  From Lemma~\ref{Lm2}, we know that the matrix of any element of $\Ord_D$ has the form $\DS{\left(%
\begin{array}{cc}
  e\cdot \id_2 & B \\ 2B^* & f\cdot \id_2 \\
\end{array}%
\right)}$, with $B$ in  $\mathbf{M}_{2\times 2}(\Z)$. The
remainder  of  the  proof   follows  the  same  lines  as  Proposition
\ref{NormA1Prop}.
\end{proof}

\subsubsection{Cylinder decompositions of type $B$}

\begin{Proposition}\label{NormBProp}
Suppose that $(X,\omega)\in\Omega  E_D(4)$ admits a cylinder decomposition in
 Model   $B$ and let   $\alpha_{1,1},\beta_{1,1}, \alpha_{1,2}, \beta_{1,2},  \alpha_2, \beta_2 \in H_1(X,\Z)$  be as in
Figure~\ref{fig:modelB:basis}                below.                Set
$\alpha_1=\alpha_{1,1}+\alpha_{1,2},
\beta_1=\beta_{1,1}+\beta_{1,2}$. Then
\begin{itemize}
\item[(i)] There exists a unique generator $T$ of $\Ord_D$ which is written in the basis 
$(\alpha_1,\beta_1,\alpha_2,\beta_2)$ by a matrix of the form $\DS{\left(
\begin{smallmatrix}
e\cdot\rm{id}_2 & B \\2B^* & 0\\
\end{smallmatrix}
\right)}$ such that $T^*(\omega)=\lambda(T)\omega$ with $\lambda(T)>0$.

\medskip

\item[(ii)]    Up  to   the   action   of $\GL^+(2,\R)$  and  Dehn  twists,  there exist  $w,h,t  \in \N$
 such that the   tuple         $(w,h,t,e)$         satisfies

$$(\mathcal{P'})\left\{\begin{array}{l}        w>0,h>0,\;       0\leq
  t<\gcd(w,h),\\       \gcd       (w,h,t,e)       =1,\\       D=e^2+8w
  h,\\ 0<\frac{e+\sqrt{D}}{4}< w < \frac{e+\sqrt{D}}{2}=:\lambda,
\end{array}
\right.,
$$ 
and the matrix of $T$ is given by  $ \DS{\left(%
\begin{smallmatrix}
  e & 0 & w & t \\ 0 & e & 0 & h \\ 2h & -2t & 0 & 0 \\ 0 & 2w & 0 & 0
  \\
\end{smallmatrix}%
\right)}$. Moreover, in these coordinates we have

\begin{equation}\label{normalize:B}
\left\{                                              \begin{array}{l}
  \omega(\Z\alpha_{1,1}+\Z\beta_{1,1})=\omega(\Z\alpha_{1,2}+\Z\beta_{1,2})=
  \frac{\lambda}{2}\cdot\Z^2,\\ \omega(\Z\alpha_2+\Z\beta_2)=\Z(w,0)+\Z(t,h).
\end{array}
\right.
\end{equation}
\end{itemize}
Conversely,  let $(X,\omega)$ be an  Abelian differential in
$\Omega\mathfrak{M}(4)$ having a cylinder decomposition in Model
$B$.  Suppose  that  there   exists  $(w,h,t,e)  \in  \Z^4$  verifying
$(\Pcal')$  such  that,  after  normalizing by  $\GL^+(2,\R)$,  all  the
conditions  in (\ref{normalize:B})  are  satisfied, then  $(X,\omega)$
belongs to $\Omega E_D(4)$.
\end{Proposition}

\begin{figure}[htbp]
\centering \subfloat{%
\begin{tikzpicture}[scale=0.6]
\fill[fill=yellow!80!black!20,even odd  rule] (0,0) --  (5,0) -- (6,2)
-- (1,2); \draw (0,0)  -- (2,0) -- (2,-3) -- (5,-3)  -- (5,0) -- (6,2)
-- (4,2) --  (4,5) -- (1,5)  -- (1,2) --  cycle; \draw (1,2)  -- (4,2)
(2,0) -- (5,0);

\draw[thick,  dashed,  ->,  >=  angle  45]  (0.25,0.5)  --  (2.5,0.5);
\draw[thick, dashed] (2.5,0.5) -- (5.25,0.5); \draw[thick, dashed, ->,
  >= angle 45] (5,0) .. controls (5,1) and (5,1.5) .. (6,2);

\draw[thick, dashed,  ->, >= angle 45] (1,3)  -- (2.5,3); \draw[thick,
  dashed]  (2.5,3) --  (4,3); \draw[thick,  dashed, ->,  >=  angle 45]
(4,2) .. controls (3.5,2.5) and (3.5,4.5) .. (4,5);

\draw[thick,  dashed,   ->,  >=  angle   45]  (2,  -2)   --  (3.5,-2);
\draw[thick, dashed]  (3.5,-2) -- (5,-2); \draw[thick,  dashed, ->, >=
  angle 45] (5,-3) .. controls (4.5,-2.5) and (4.5,-0.5) .. (5,0);

\filldraw[fill=white,draw=black] (0,0) circle (3pt) (2,0) circle (3pt)
(2,-3)  circle (3pt)  (3,-3) circle  (3pt) (5,-3)  circle  (3pt) (5,0)
circle (3pt) (6,2) circle (3pt)  (4,2) circle (3pt) (4,5) circle (3pt)
(3,5) circle (3pt) (1,5) circle (3pt) (1,2) circle (3pt);

\draw (1.5,0.5) node[above] {$\scriptstyle \alpha_2$} (4.75,1.25) node
      {$\scriptstyle  \beta_2$}  (1.75,3)  node[above]  {$\scriptstyle
        \alpha_{1,1}$}   (3,4)   node   {$\scriptstyle   \beta_{1,1}$}
      (2.75,-2) node[above]  {$\scriptstyle \alpha_{1,2}$} (4,-1) node
      {$\scriptstyle \beta_{1,2}$};
\end{tikzpicture}
}
\caption{Basis
  $\{\alpha_{1,1},\beta_{1,1},\alpha_{1,2},\beta_{1,2},\alpha_{2},
  \beta_{2}\}$ of $H_1(X,\Z)$ associated to a cylinder decomposition
  in  Model  $B$   (the  fixed  cylinder  is  colored   in  grey).  If
  $\alpha_1:=\alpha_{1,1}+\alpha_{1,2}$                             and
  $\beta_1:=\beta_{1,1}+\beta_{1,2}$,                               then
  $\{\alpha_1,\beta_1,\alpha_2,\beta_2\}$     is     a symplectic   basis     of
  $H_{1}(X,\Z)^{-}$.  }
\label {fig:modelB:basis}
\end{figure}
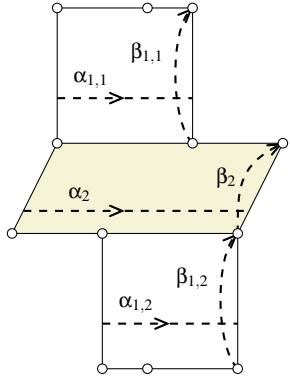

\begin{proof}[Proof of Proposition~\ref{NormBProp}]
We                  first                 observe                 that
$(\alpha_{1,1},\beta_{1,1},\alpha_{1,2},\beta_{1,2},\alpha_2,\beta_2)$
is a  canonical basis of  $H_1(X,\Z)$.  To see  this, we only  have to
check  that  $\langle   \beta_{1,1},  \beta_{1,2}  \rangle  =  \langle
\beta_{1,1},\beta_2\rangle            =\langle            \beta_{1,2},
\beta_2\rangle=0$. But this follows immediately from the fact that the
cycles  $\beta_{1,1}+\beta_2, \beta_{1,2}+  \beta_2,  \beta_2$ can  be
represented by three  disjoint simple closed curves. The  proof of the
proposition   then    follows   the   same    lines   as   Proposition
\ref{NormA2Prop}, with the exception that by construction we must have
$$ 0<\frac{\lambda}{2}< w < \lambda.
$$ We leave the details to the reader.
\end{proof}

\subsection{Surfaces having no cylinder decompositions in model $A\pm$}

\begin{Proposition}
\label{prop:D8:topo}
Let $(X,\omega)  \in \Omega E_D(4)$  be an eigenform.  If $(X,\omega)$
admits no  cylinder decompositions in  Model $A+$ or Model  $A-$ then, 
up to the action of $\GL^+(2,\R)$, the surface $(X,\omega)$ is the one 
presented in Figure~\ref{fig:D8}. In  particular, we have
$D=8$.
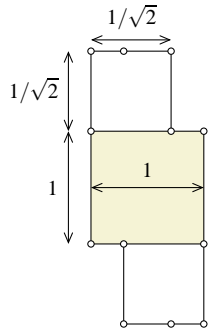
\begin{figure}[htbp]
\centering \subfloat{%
\begin{tikzpicture}[scale=1.5]
\fill[fill=yellow!80!black!20,even odd  rule] (0,0) --  (1,0) -- (1,1)
-- (0,1);  \draw (0,0)  -- (0.29,0)  -- (0.29,-0.71)  --  (1,-0.71) --
(1,1) -- (0.71,1) -- (0.71,1.71)  -- (0,1.71) -- cycle; \draw (0.29,0)
-- (1,0)  (0,1)  --  (0.71,1); \filldraw[fill=white,draw=black]  (0,0)
circle  (0.8pt) (0.29,0)  circle (0.8pt)  (0.29,-0.71)  circle (0.8pt)
(1,-0.71)  circle (0.8pt)  (1,0) circle  (0.8pt) (1,1)  circle (0.8pt)
(0.71,1)  circle (0.8pt)  (0.71,1.71) circle  (0.8pt)  (0,1.71) circle
(0.8pt) (0,1)  circle (0.8pt) (0.29,1.71)  circle (0.8pt) (0.71,-0.71)
circle (0.8pt);

\draw[thin,  <->,   >=angle  45]  (-0.2,1)  --   (-0.2,1.71)  ;  \draw
(-0.2,1.36)  node[left] {$\scriptstyle 1/\sqrt{2}$};  \draw[thin, <->,
  >=angle  45]  (-0.2,0) --  (-0.2,1)  ;  \draw (-0.2,0.5)  node[left]
     {$\scriptstyle  1$};  \draw[thin, <->,  >=angle  45] (0,1.81)  --
     (0.71,1.81)   ;  \draw  (0.36,1.81)   node[above]  {$\scriptstyle
       1/\sqrt{2}$}; \draw[thin, <->,  >=angle 45] (0,0.5) -- (1,0.5);
     \draw (0.5,0.5) node[above] {$\scriptstyle 1$};
\end{tikzpicture}
}
\caption{
\label {fig:D8}
A surface in $\Omega E_{D}(4)$ (decomposed into cylinders in Model $B$) that does not
admit a cylinder decomposition in Model $A\pm$ in any direction.  }
\end{figure}
\end{Proposition}

\begin{proof}[Proof of Proposition~\ref{prop:D8:topo}]
Since Model  $A+$ and  Model $A-$ are  characterized by the  fact that
there exists  a simple cylinder (see Proposition~\ref{prop:decomp:cylinders}), we  will show that in  all cases, but
one, we can find a direction having a simple cylinder. Thus let $(X,\omega)  \in \Omega
E_D(4)$ be an  eigenform and let us assume  $(X,\omega)$ is decomposed
into        cylinders       following        Model        $B$       (see
Figure~\ref{fig:modelB:basis}).     Using
$\GL^+(2,\R)$, we can normalize so that
$$       \left\{      \begin{array}{l}      \omega(\alpha_{1,1})=(x,0)
  \\   \omega(\beta_{1,1})=(0,x)  \\   \omega(\alpha_{2})   =  (x+y,0)
  \\ \omega(\beta_{2}) = (t,1)
\end{array} \right.
$$ where the  parameters $x,y,t\in \R$ satisfy $0 < y  < x$ and $0\leq
t<x+y$.  We   will  show   that  unless  $t=0$,   $y=x^{2}/(x+1)$  and
$x=1/\sqrt{2}$ there always exists a direction having a simple cylinder. \medskip
\begin{itemize}

\item[$\bullet$] {\bf Step 1:} $t=0$. Let us  consider the direction $\theta_{1}$ of slope  $\frac{x+1}{t-x-y}$. Clearly, if
$\frac{x+y-t}{x+1}x <  x-y$ then there  exists a simple  cylinder in direction $\theta_{1}$ (see
Figure~\ref{fig:D8:simple},     left).    Thus    let     us    assume
$\frac{x+y-t}{x+1}x \geq x-y$, or equivalently
\begin{equation}
\label{eq:simple:1}
tx \leq 2xy-x+y.
\end{equation} 
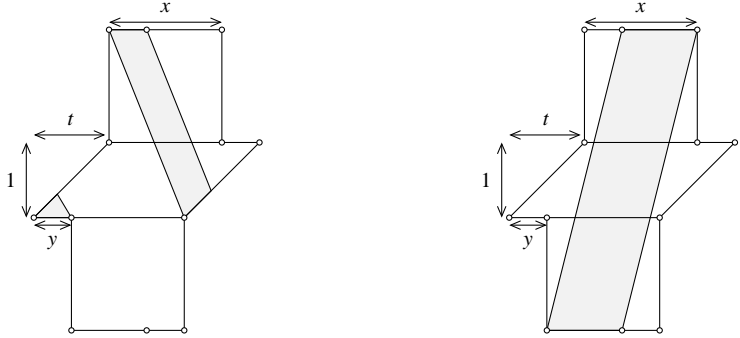
\begin{figure}[htbp]
\begin{minipage}[t]{0.4\linewidth}
\centering
\begin{tikzpicture}[scale=0.5]

\filldraw[fill=gray!10] (2,5) -- (4,0)  -- (intersection of 4,0 -- 6,2
and 5,0  -- 3,5) --  (3,5) -- cycle; \filldraw[fill=gray!10]  (0,0) --
(1,0) -- (intersection of 0,0 -- 2,2 and 1,0 -- -2,5) -- cycle;

\draw (0,0) -- (1,0) -- (1,-3) -- (4,-3) -- (4,0) -- (6,2) -- (5,2) --
(5,5) -- (2,5) -- (2,2) -- cycle; \draw (1,0) -- (4,0) (2,2) -- (5,2);

\filldraw[draw=black,  fill=white]  (0,0)  circle (2pt)  (1,0)  circle
(2pt) (4,0) circle (2pt) (1,-3) circle (2pt) (4,-3) circle (2pt) (6,2)
circle (2pt) (5,2) circle (2pt)  (5,5) circle (2pt) (2,5) circle (2pt)
(2,2) circle (2pt) (3,5) circle (2pt) (3,-3) circle (2pt) ;

\draw[thin,  <->,  >=angle 45]  (2,5.2)  --  (5,5.2); \draw  (3.5,5.2)
node[above] {$\scriptstyle x$};  \draw[thin, <->, >=angle 45] (-0.2,0)
-- (-0.2,2)   ;   \draw   (-0.2,1)  node[left]   {$\scriptstyle   1$};
\draw[thin, <->,  >=angle 45]  (0,-0.2) -- (1,-0.2);  \draw (0.5,-0.2)
node[below] {$\scriptstyle  y$}; \draw[thin, <->,  >=angle 45] (0,2.2)
-- (1.9,2.2); \draw (1,2.2) node[above] {$\scriptstyle t$};
\end{tikzpicture}
\end{minipage} 
\begin{minipage}[t]{0.4\linewidth}
\centering
\begin{tikzpicture}[scale=0.5]
\filldraw[fill=gray!10] (1,-3)  -- (3,-3)--  (5,5) -- (3,5)  -- cycle;
\draw (0,0) -- (1,0) -- (1,-3) -- (4,-3) -- (4,0) -- (6,2) -- (5,2) --
(5,5) -- (2,5) -- (2,2) -- cycle; \draw (1,0) -- (4,0) (2,2) -- (5,2);

\filldraw[draw=black,  fill=white]  (0,0)  circle (2pt)  (1,0)  circle
(2pt) (4,0) circle (2pt) (1,-3) circle (2pt) (4,-3) circle (2pt) (6,2)
circle (2pt) (5,2) circle (2pt)  (5,5) circle (2pt) (2,5) circle (2pt)
(2,2) circle (2pt) (3,5) circle (2pt) (3,-3) circle (2pt) ;

\draw[thin,  <->,  >=angle 45]  (2,5.2)  --  (5,5.2); \draw  (3.5,5.2)
node[above] {$\scriptstyle x$};  \draw[thin, <->, >=angle 45] (-0.2,0)
-- (-0.2,2)   ;   \draw   (-0.2,1)  node[left]   {$\scriptstyle   1$};
\draw[thin, <->,  >=angle 45]  (0,-0.2) -- (1,-0.2);  \draw (0.5,-0.2)
node[below] {$\scriptstyle  y$}; \draw[thin, <->,  >=angle 45] (0,2.2)
-- (1.9,2.2); \draw (1,2.2) node[above] {$\scriptstyle t$};

\end{tikzpicture}
\end{minipage}
\caption{
\label{fig:D8:simple}
Cylinders in  directions $\theta_{1}$ and $\theta_{2}$. }
\end{figure}

Let us consider a  second direction $\theta_{2}$ of slope $\frac{2x+1}{t}$. Clearly
if  $t>0$  and $\frac{t  x}{2x+1}  < y$  then  there  exists a  simple
cylinder in direction $\theta_{2}$ (see  Figure~\ref{fig:D8:simple},  right). Observe  that  if
$t>0$      then     this      case      actually     occurs      since
inequality~(\ref{eq:simple:1})    implies   $tx    \leq    2xy-x+y   <
y(2x+1)$.   Thus  the proposition   is   proven  unless
$t=0$. Hence from now on, we assume $t=0$. \medskip

\item[$\bullet$] {\bf Step 2:} $y=\frac{x^2}{x+1}$. We apply the  previous idea to the direction 
$\theta_{3}$  of slope $\frac{x+1}{x}$
(see Figure~\ref{fig:D8:simple2} for details).

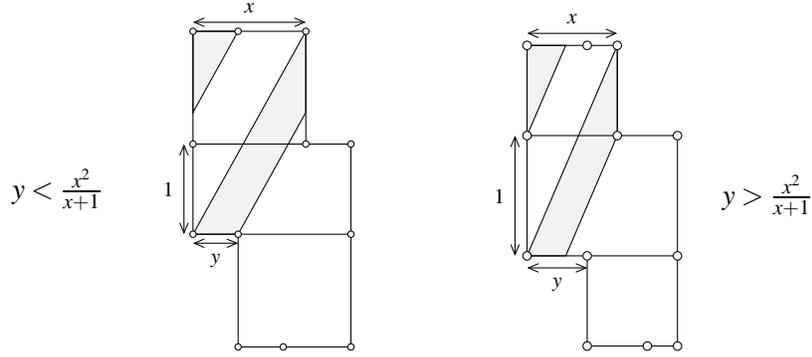
\begin{figure}[htbp]
\begin{minipage}[t]{0.4\linewidth}
\centering
\begin{tikzpicture}[scale=0.6]

\filldraw[fill=gray!10] (0,4.5)  -- (1,4.5) --  (intersection of 1,4.5
-- -1.5,0 and 0,0 -- 0,1  ) -- cycle; \filldraw[fill=gray!10] (0,0) --
(2.5,4.5) -- (intersection  of 1,0 -- 3.5,4.5 and 2.5,4  -- 2.5,2 ) --
(1,0) -- cycle;

\draw (0,0) --  (1,0) -- (1,-2.5) -- (3.5,-2.5)  -- (3.5,2) -- (2.5,2)
-- (2.5,4.5)  -- (0,4.5)  -- cycle;  \draw (1,0)  -- (3.5,0)  (0,2) --
(2.5,2);

\filldraw[draw=black,  fill=white]  (0,0)  circle (2pt)  (1,0)  circle
(2pt)  (1,-2.5) circle  (2pt) (3.5,-2.5)  circle (2pt)  (3.5,0) circle
(2pt) (3.5,2) circle (2pt) (2.5,2) circle (2pt) (2.5,4.5) circle (2pt)
(0,4.5) circle (2pt) (0,2)  circle (2pt) (1,4.5) circle (2pt) (2,-2.5)
circle (2pt) ;

\draw[thin, <->,  >=angle 45]  (0,4.7) -- (2.5,4.7);  \draw (1.25,4.7)
node[above] {$\scriptstyle x$};  \draw[thin, <->, >=angle 45] (-0.2,0)
-- (-0.2,2)   ;   \draw   (-0.2,1)  node[left]   {$\scriptstyle   1$};
\draw[thin, <->,  >=angle 45]  (0,-0.2) -- (1,-0.2);  \draw (0.5,-0.2)
node[below] {$\scriptstyle y$};

\draw (-3,1) node {$y<\frac{x^{2}}{x+1}$};

\end{tikzpicture}
\end{minipage} 
\begin{minipage}[t]{0.4\linewidth}
\centering
\begin{tikzpicture}[scale=0.8]

\filldraw[fill=gray!10] (0,0) -- (1.5,3.5) -- (1.5,2) -- (intersection
of 1.5,2 --  0,-1.5 and 0,0 -- 1,0)  -- cycle; \filldraw[fill=gray!10]
(0,2)  -- (0,3.5)  -- (intersection  of 0,2  -- 1.5,5.5  and  0,3.5 --
1,3.5) -- cycle;

\draw (0,0) --  (1,0) -- (1,-1.5) -- (2.5,-1.5)  -- (2.5,2) -- (1.5,2)
-- (1.5,3.5)  -- (0,3.5)  -- cycle;  \draw (1,0)  -- (2.5,0)  (0,2) --
(1.5,2);

\filldraw[draw=black,  fill=white]  (0,0)  circle (2pt)  (1,0)  circle
(2pt)  (1,-1.5) circle  (2pt) (2.5,-1.5)  circle (2pt)  (2.5,0) circle
(2pt) (2.5,2) circle (2pt) (1.5,2) circle (2pt) (1.5,3.5) circle (2pt)
(0,3.5) circle (2pt) (0,2)  circle (2pt) (1,3.5) circle (2pt) (2,-1.5)
circle (2pt) ;

\draw[thin, <->,  >=angle 45]  (0,3.7) -- (1.5,3.7);  \draw (0.75,3.7)
node[above] {$\scriptstyle x$};  \draw[thin, <->, >=angle 45] (-0.2,0)
-- (-0.2,2)   ;   \draw   (-0.2,1)  node[left]   {$\scriptstyle   1$};
\draw[thin, <->,  >=angle 45]  (0,-0.2) -- (1,-0.2);  \draw (0.5,-0.2)
node[below] {$\scriptstyle y$};

\draw (4,1) node {$y>\frac{x^{2}}{x+1}$};

\end{tikzpicture}
\end{minipage}
\caption{
\label{fig:D8:simple2}
Cylinders in direction  $\theta_{3}$  of slope $\frac{x+1}{x}$ when     $t=0$.    If
$y \neq \frac{x^{2}}{x+1}$ then there always
exists a  simple cylinder in direction $\theta_3$.  }
\end{figure}

\item[$\bullet$] {\bf Step 3:} $x=\frac{1}{\sqrt{2}}$.  Since  $y  =   \frac{x^{2}}{x+1}$, the
inequality~(\ref{eq:simple:1})  becomes $x\geq  \frac1{\sqrt{2}}$.  To
complete the first part of the proposition,  it remains to show that if there
is no simple  cylinders then $x \leq \frac1{\sqrt{2}}$. This is achieved 
by considering the direction $\theta_{4}$ of slope $-\frac{2x+1}{2x}$
as shown in Figure~\ref{fig:D8:simple3}.
\begin{figure}[htbp]
\begin{tikzpicture}[scale=0.6]
\filldraw[fill=gray!10] (1,-2)  -- (2,-2) -- (intersection  of 2,-2 --
-2,4  and 1,0  --  1,1) --  cycle;  \filldraw[fill=gray!10] (3,-2)  --
(intersection  of  -1,4   --  3,-2  and  0,0  --   0,1)  --  (0,4)  --
(intersection   of  0,4   --  4,-2   and   3,0  --   3,1)  --   cycle;
\filldraw[fill=gray!10] (1,4) -- (2,4) -- (intersection of 1,4 -- 5,-2
and 2,0 -- 2,1) -- cycle; \draw  (0,0) -- (1,0) -- (1,-2) -- (3,-2) --
(3,2) -- (2,2) -- (2,4) --  (0,4) -- cycle; \draw (1,0) -- (3,0) (0,2)
-- (2,2); \filldraw[draw=black,  fill=white] (0,0) circle  (2pt) (1,0)
circle  (2pt) (1,-2)  circle (2pt)  (3,-2) circle  (2pt)  (3,0) circle
(2pt) (3,2) circle  (2pt) (2,2) circle (2pt) (2,4)  circle (2pt) (0,4)
circle (2pt) (0,2) circle (2pt) (1,4) circle (2pt) (2,-2) circle (2pt)
;
\draw[thin,  <->,  >=angle  45]  (2.2,4)  --  (2.2,2);  \draw  (2.2,3)
node[right] {$\scriptstyle x$};  \draw[thin, <->, >=angle 45] (-0.2,0)
-- (-0.2,2)   ;   \draw   (-0.2,1)  node[left]   {$\scriptstyle   1$};
\draw[thin, <->,  >=angle 45]  (0,-0.2) -- (1,-0.2);  \draw (0.5,-0.2)
node[below] {$\scriptstyle y$};
\end{tikzpicture}
\caption{
\label{fig:D8:simple3}
Cylinder in direction $\theta_{4}$ of slope $-\frac{2x+1}{2x}$.  }
\end{figure}
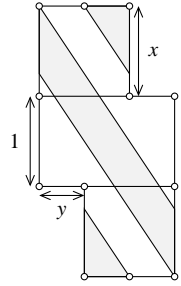
Clearly, if $\frac{2x}{2x+1}(x+1) <  x+y$ then  there exists  a simple
cylinder  in direction $\theta_{4}$ (see   the  figure  for   details).  Thus  one   can  assume
$\frac{2x}{2x+1}(x+1)    \geq    x+y$.     Substituting     $y$    by
$\frac{x^{2}}{x+1}$, we see that this  inequality  is   equivalent  to   $x  \leq
\frac1{\sqrt{2}}$, that is the desired inequality. The proof  of the first  part of the  proposition is now  complete. 

\end{itemize}

\medskip

In order to compute the discriminant,  one needs to put $(X,\omega)$ into
the form of Proposition~\ref{NormBProp}.  Since $x+y=1$ we have $w=1$,
$h=1$ and $\lambda/2 = x$. Thus $\lambda = \frac{e+\sqrt{D}}{2} = 2x =
\sqrt{2}$  and $e=0$.  Since the  tuple $(w,h,t,e)  = (1,1,0,0)$  is a
solution  to  $(\Pcal')$ the  discriminant  is $D=e^{2}+8wh=8$.   This
completes the proof of Proposition~\ref{prop:D8:topo}.
\end{proof}

\subsection{Two Consequences}
\medskip
\begin{proof}[Proof of Theorem~\ref{theo:onto:map}]
From Proposition~\ref{prop:D8:topo}, we know that, when $D\neq 8$, every surface in 
$\Omega E_D(4)$ admits a cylinder decomposition in models $A+$, or $A-$. Theorem~\ref{theo:onto:map} 
 is  then  a    direct    consequence    of 
Propositions \ref{NormA1Prop} and \ref{NormA2Prop}.
\end{proof}

From                    Propositions~\ref{NormA1Prop},~\ref{NormA2Prop}
and~\ref{NormBProp}, we also draw
\begin{Corollary}
\label{cor:admis:disc}
Let $D$ be a discriminant.
\begin{enumerate}
\item If $D \equiv 5 \mod 8$, then $\Omega E_D(4)=\ety$.
\item If $D \equiv 0, 1, 4 \mod 8$ and $D \geq 17$ then $\Omega E_D(4)
  \neq \ety$.
\end{enumerate}
\end{Corollary}

\begin{proof}
The first assertion  is immediate: if $(X, \omega)  \in \Omega E_D(4)$
then Corollary~\ref{cor:decomp:3cyl}  implies that $(X,\omega)$ admits
a  decomposition into  three cylinders  in some  direction  in model
$A+$,  $A-$   or     $B$.    Following   respectively
Proposition~\ref{NormA1Prop},   \ref{NormA2Prop}   or~\ref{NormBProp},
there  exists $(w,h,t,e) \in  \Z^{4}$ such  that $D=e^2+8wh$.  Thus $D
\equiv 0,1,4 \mod 8$. \medskip

Conversely, given any  $D \geq 17$ such that $D  \equiv 0,1,4 \mod 8$,
it is  straightforward to construct a solution  $(w,h,t,e) \in \Z^{4}$
satisfying $(\mathcal{P})$. Indeed:
\begin{itemize}
\item  if  $D\equiv  0  \mod  8$,  the tuple  $(  D/8,1,0,0  )$  is  a
  solution.\medskip

\item  if $D\equiv  1 \mod  8$, the  tuple $(  (D-1)/8,1,0,-1 )$  is a
  solution.\medskip

\item  if $D\equiv  4 \mod  8$, the  tuple $(  (D-4)/8,1,0,-2 )$  is a
  solution.
\end{itemize}
By  Propositions~\ref{NormA1Prop}, \ref{NormA2Prop}, we know
that any  solution to  $(\Pcal)$ gives rise  to a surface  in $\Omega
E_D(4)$. 
\end{proof}

\subsection{Small discriminants}\label{sec:proof:MainTh2}
We can now prove Theorem~\ref{MainTh2}, which deals with discriminants
smaller than $17$.

\begin{proof}[Proof of Theorem~\ref{MainTh2}]\hfill

There are only  $4$ admissible values for $D$  smaller than $17$, {\it
  i.e.} $D \in \{8,9,12,16\}.$

\begin{enumerate}
\item  For  $D=16$, or  $D=9$,  there  are  no $(w,h,t,e)$  satisfying
  $(\mathcal{P})$ nor $(\mathcal{P}')$.\medskip

\item  For  $D=12$, $(w,h,t,e)=(1,1,0,-2)$  is  the  only solution  to
  $(\mathcal{P})$,     and    there     are     no    solutions     to
  $(\mathcal{P}')$.  {\em  A  priori},  we  get two  surfaces  from  a
  solution to $(\mathcal{P})$, one for the Model $A+$, and one for the
  Model $A-$. But a surface admitting cylinder decomposition in Model
  $A+$  also admits cylinder  decompositions in  Model $A-$  and vice
  versa  (see also  Section~\ref{sec:butterfly}).  Therefore,  the two
  surfaces belong to the same $\GL^+(2,\R)$-orbit. Since $D=12$ is not a
  square, the surfaces in  $\Omega E_{12}(4)$ can not be square-tiled. 
  A classical theorem due to Thurston \cite{Thurston1988} then implies that they do not admit any
   decomposition into one cylinder, or two cylinders  exchanged   by the Prym involution.  Thus, the 
    corresponding Teichm\"uller curve
has  exactly two  cusps corresponding to two models of decomposition into three cylinders.\medskip

\item  For  $D=8$,  $(w,h,t,e)=(1,1,0,0)$  is  the  only  solution  to
 $(\mathcal{P}')$, and there are no solutions to $(\mathcal{P})$.  A direct consequence of this fact is that the
  surfaces in $\Omega E_8(4)$ do not have any simple cylinder, since otherwise they would have a decomposition 
  in model $A+$ or $A-$, and  there would be a solution to $(\Pcal)$.
Another consequence is that the surfaces in $\Omega E_8(4)$ do not admit any decomposition into one 
or two cylinders, since otherwise they would have a simple cylinder. Therefore,   the   corresponding 
 Teichm\"uller  curve   has  only  one  cusp corresponding to the unique model of cylinder decomposition.

\end{enumerate}

\end{proof}

\section{Uniqueness}
\label{sec:uniqueness}

As remarked previously, in general, neither the representation
of $\Ord_D$  nor the involution  $\rho$ is uniquely determined  by the
eigenform $(X,\omega)$. However, if $(X,\omega)\in \Omega E_D(4)$, 
then it does uniquely determine the pair $(\rho,\mathfrak{i})$, up to
isomorphisms of the order $\Ord_D$.

\begin{Theorem}
\label{UqeThm}
Let     $(X,\omega)$     be     an     Abelian     differential     in
$\Omega\mathfrak{M}(4)$. Suppose that there exist
\begin{itemize}
\item two involutions $\rho, \rho' : X \ra X$ such that
$$     \dim_\C\Omega(X)_\rho^-=\dim_\C\Omega(X)_{\rho'}^-=2     \qquad
  \textrm{and} \qquad \rho^*(\omega) = {\rho'}^*(\omega)= -\omega,
$$

\item two injective ring homomorphisms
$$      \mathfrak{i}:\Ord_D\ra\mathrm{End}(H_1(X,Z)_\rho^-)     \qquad
  \textrm{and}                  \qquad                  \mathfrak{i}':
  \Ord_{D'}\ra\mathrm{End}(H_1(X,\Z)_{\rho'}^-)
$$ such that their images are self-adjoint, proper subrings.
\end{itemize}
If  $\omega$  is  an  eigenform for  both  $\mathfrak{i}(\Ord_D)$  and
$\mathfrak{i}'(\Ord_{D'})$ then $\rho=\rho'$, $D=D'$, and there exists
a  ring  isomorphism  $\mathfrak{j}:  \Ord_D  \ra  \Ord_D$  such  that
$\mathfrak{i}'=\mathfrak{i}\circ\mathfrak{j}$.
\end{Theorem}

\begin{proof}
Choose a direction  for which the flat surface  $\Sigma$ associated to
$(X,\omega)$  admits  a decomposition  into  three  cylinders, such  a
direction  always exists  by  Corollary~\ref{cor:decomp:3cyl}. Let  us
assume that  the decomposition  has type $A+$.  The arguments  that we
will present also work for the  two other models $A-$ and $B$. The key
ingredient  is to  show that  the restrictions of $\rho$ and  $\rho'$ to some cylinder are the same. \medskip

The proof of Proposition~\ref{prop:decomp:cylinders} shows that there  is   exactly  one
invariant  cylinder for  each  involution. But since  there  is only  one
simple cylinder  in decomposition  $A+$ it must  be invariant  by {\it
both} $\rho$ and $\rho'$.  Hence $\rho=\rho'$. \medskip


Let
$\alpha_1,\beta_1,\alpha_{2,1},\beta_{2,1},\alpha_{2,2},\beta_{2,2}$
be as in Proposition  \ref{NormA1Prop}, so that there exists $\lambda >0$, and
a      generator       $T$      of      $\Ord_D$       such      that,
$\mathfrak{i}(T)^*(\omega)=\lambda\cdot \omega$,    and    in   the    basis
$(\alpha_1,\beta_1,\alpha_2,\beta_2)$,
$$ \mathfrak{i}(T)=\left(%
\begin{array}{cccc}
  e & 0 & 2w & 2t \\ 0 & e & 0 & 2h \\ h & -t & 0 & 0 \\ 0 & w & 0 & 0
  \\
\end{array}%
\right),
$$  where $(w,h,t,e)\in \Z^{4}$ $w>0$, $h>0$, $e\in\Z$, $\gcd(w,h,t,e)=1$.
Similarly,  there  exists  $\lambda'>0$,   and  a  generator  $T'$  of
$\Ord_{D'}$ satisfying the same conditions with adequate parameters 
$(w',h',t',e')\in    \Z^{4}$. There also exist $g, g' \in\GL(2,\R)$ such that
$$         
\left\{
\begin{array}{lll}
\mathrm{Re}(g\cdot\omega)=(\lambda,0,2w,2t),   && \mathrm{Im}(g\cdot\omega)=(0,\lambda,0,2h), \\
\mathrm{Re}(g'\cdot\omega)=(\lambda',0,2w',2t'), && \mathrm{Im}(g'\cdot\omega)=(0,\lambda',0,2h').
\end{array}
\right.
$$
It follows that $g'=s\cdot g$ for some $s\in \R_+^*$ satisfying
$$
s=\frac{\lambda}{\lambda'}=\frac{w}{w'}=\frac{t}{t'}=\frac{h}{h'}=\frac{p}{q}
\hbox{ with } p,q>0, \gcd(p,q)=1.
$$ In  particular the tuple $(w,h,t)$  (respectively, $(w',h',t')$) is
divisible by $p$ (respectively, $q$).

Recall      that      $\lambda^2=e\lambda+2wh$      and
${\lambda'}^2=e'\lambda+2w'h'$. Hence
$$   e=\frac{\lambda^2-2wh}{\lambda},   \qquad   \textrm{and}   \qquad
e'=\frac{{\lambda'}^2-2w'h'}{\lambda'},
$$ Thus we have
$$ e=\frac{p}{q}e'.
$$  Therefore  $e$  is divisible  by  $p$  and  $e'$ is  divisible  by
$q$. Putting this together with $\gcd(w,h,t,e)=\gcd(w',h',t',e')=1$ we draw
$p=q=1$.    In   conclusion    $(w,h,t,e)    =   (w',h',t',e')$    and
$\lambda=\lambda'$.  Thus $D=D'$ and one can define a ring isomorphism
$\mathfrak{j}      :\Ord_D       \ra      \Ord_D$      by      setting
$\mathfrak{j}(T')=T$.   Clearly,    the   isomorphism   $\mathfrak{j}$
satisfies               the              desired              relation
$\mathfrak{i}'=\mathfrak{i}\circ\mathfrak{j}$. This  ends the proof of
the theorem.
\end{proof}

As an immediate consequence we draw:

\begin{Corollary}\label{cor:disjoint}
If $D_1\neq D_2$ then $\Omega E_{D_1}(4)\cap \Omega E_{D_2}(4)=\ety$.
\end{Corollary}

\section{Case $D$ odd}
\label{sec:OddDisc}

In  this section,  we show  that when  $D$ is  odd, the  locus $\Omega
E_D(4)$ consists of at least two distinct $\GL^+(2,\R)$-orbits.

\begin{Theorem}
\label{theo:disconnect:odd}
Suppose that $D \equiv  1 \mod 8$, and let $p=(w,h,t,e)\in \Pcal_D$ be
an incomplete  prototype. Then the two translation surfaces constructed from  the complete prototypes
$(p,+)$ and $(p,-)$  do not belong to the same $\GL^+(2,\R)$-orbit.
\end{Theorem}

Recall that we denote by $\langle  , \rangle  $  the  restriction of the
intersection     form  to      $H_1(X,\Z)^-$. Theorem~\ref{theo:disconnect:odd} follows from the next lemma.

\begin{Lemma}
\label{lm:OddDisc:Generator}
Let $T^+$ (respectively, $T^-$) be the generator of $\Ord_D$ associated to
the prototype $(p,+)$ (respectively, $(p,-)$). Then:
$$
\begin{array}{llll}
\langle  ,  \rangle_{|\rm{Im}(T^+)}  &  \neq  & 0  \mod  2 &\textrm{ and,} \\
\langle  , \rangle_{|\rm{Im}(T^-)} & = & 0 \mod 2,
\end{array}
$$
where $\rm{Im}(T^+)$ (respectively,  $\rm{Im}(T^-))$   is  the  image  of  $T^+$   (respectively,  $T^-$)  in
$H_1(X,\Z)^-$.
\end{Lemma}

\begin{proof}
Using    the    notations    in    Proposition~\ref{NormA1Prop}    and
Proposition~\ref{NormA2Prop}, we have
$$
T^+= \left(\begin{smallmatrix}
  e & 0 & 2w & 2t \\ 0 & e & 0 & 2h \\ h & -t & 0 & 0 \\ 0 & w & 0 & 0
  \\
\end{smallmatrix} \right)= \left(%
\begin{smallmatrix}
  e & 0 & 0 & 0 \\ 0 & e & 0 & 0 \\ h & -t & 0 & 0 \\ 0 & w & 0 & 0 \\
\end{smallmatrix}
\right) \mod 2,
$$
and
$$
T^-= \left( \begin{smallmatrix}
  e & 0 & w & t \\ 0 & e & 0 & h \\ 2h & -2t & 0 & 0 \\ 0 & 2w & 0 & 0
  \\
\end{smallmatrix}\right) = \left(%
\begin{smallmatrix}
  e & 0 & w & t \\ 0 & e & 0 & h \\ 0 & 0 & 0 & 0 \\ 0 & 0 & 0 & 0 \\
\end{smallmatrix}%
\right) \mod 2,
$$
in the  bases $(\alpha^+_1,\beta^+_1,\alpha^+_2,\beta^+_2)$
and           $(\alpha^-_1,           \beta^-_1,\alpha^-_2,\beta^-_2)$, 
respectively. Since by assumption $D$ is odd, $e$ is also odd, and 
$$
\begin{array}{l}
\rm{Im}(T^+)= <\alpha^+_1+h\alpha^+_2,\beta^+_1-t\alpha^+_2+w\beta^+_2> \mod 2,\\
\rm{Im}(T^-)=<\alpha^-_1,\beta^-_1>  \mod 2
\end{array}
$$
By construction, in the case of $(p,+)$, $\langle
\alpha^+_1,\beta^+_1 \rangle =1$ and $\langle \alpha^+_2, \beta^+_2 \rangle
= 0 \mod 2$, and in  the case of $(p,-)$, $\langle \alpha^-_1,
\beta^-_1  \rangle =  0$, and  $\langle  \alpha^-_2, \beta^-_2
\rangle =1\mod  2$. The lemma is now a straightforward computation.
\end{proof}

\begin{proof}[Proof of Theorem~\ref{theo:disconnect:odd}]
Suppose that the surfaces  constructed from $(p,+)$ and $(p,-)$ belong
to  the same $\GL^+(2,\R)$-orbit  of some  Prym eigenform  $(X,\omega)$. Then 
Theorem~\ref{UqeThm} implies that there is  a  unique Prym involution
$\rho:  X \ra  X$, and a unique proper subring of $\rm{End}(\Prym(X,\rho))$ isomorphic to
$\Ord_D$ consisting  of self-adjoint endomorphisms  for which $\omega$
is  an eigenform.  It follows,  in particular,  both $T^+$  and $T^-$
belong to that subring.\medskip

Let   $S$  denote   the  subspace  of   $H^1(X,\R)^-  \cong
\Omega(X,\rho)^-$                     generated                     by
$\{\mathrm{Re}(\omega),\mathrm{Im}(\omega)\}$,  and  $S'$  denote  the
orthogonal complement of $S$ with  respect to the intersection form on
$H^1(X,\R)^-$ (dual to  the symplectic form $\langle, \rangle$).  Recall that by
construction  $T^{+}$ and $T^{-}$ satisfy the same quadratic equation given by the 
polynomial $X^2    -eX-2wh$,  and
$T^+_{|S}  = T^-_{|S}=\lambda \cdot \Id_S$, where  $\lambda$ is  the unique
positive    root    of    this    polynomial. If $\lambda'=e-\lambda$ is the other negative root then
$\lambda'$  is also  an  eigenvalue  of both  $T^+$  and $T^-$.  Since
$T^\pm$       are       self-adjoint,       it      follows       that
$T^+_{|S'}=T^-_{|S'}=\lambda' \cdot\Id_{S'}$,  and therefore  $T^+=T^-$. But
this is a contradiction with Lemma~\ref{lm:OddDisc:Generator}, and 
Theorem~\ref{theo:disconnect:odd} is proved.
\end{proof}

\begin{Remark}
The proof of Theorem~\ref{theo:disconnect:odd} actually shows that the two-dimensional
 subspaces of  $H^1(X,\R)^-$ generated by the eigenforms constructed from $(p,+)$ and $(p,-)$ are distinct.
\end{Remark}

\section{Prototypes and butterfly moves}
\label{sec:butterfly}

Fix a discriminant $D$ such that $D\equiv 0,1,4 \mod 8$. Following the
previous sections,  we naturally define the two  sets $\mathcal P_{D}$
and $\mathcal Q_{D}$ as follows
$$ \mathcal{P}_D:=\left\{ (w,h,t,e) \in \Z^4,
\begin{array}{l}
w>0,  \ h>0,\  0\leq  t<\gcd(w,h),  \ 2h+e  <  w, \\  \gcd(w,h,t,e)=1,
\textrm{ and } D=e^2+8hw.\\
\end{array}
\right\},
$$ and
$$ \mathcal{Q}_D:=\left\{ (w,h,t,e,\varepsilon)  \in \Z^5, \ (w,h,t,e)
\in \mathcal{P}_D, \textrm{ and } \varepsilon \in \{\pm 1\} \right\}.
$$

\begin{Remark}
Observe  that the condition  $\lambda =  \frac{e+\sqrt{D}}{2} <  w$ is
equivalent to $2h+e < w$.
\end{Remark}

We call an  element of $\mathcal{P}_D$ (respectively, $\mathcal{Q}_D$)
an   {\em  incomplete  prototype}   (respectively,  a   {\em  complete
  prototype})       for      the       discriminant       $D$.      
By Propositions~\ref{NormA1Prop}  and~\ref{NormA2Prop},  we  know that  a
complete   prototype  $(w,h,t,e,\varepsilon)$   produces   a Prym eigenform 
in $\Omega  E_D(4)$, and from By Theorem~\ref{theo:onto:map}
the number of  components of $\Omega E_D(4)$ is  bounded from above by
$\#\mathcal  Q_{D}$. The  goal  of  this section  is  to introduce  an
equivalence  relation  $\sim$, called  {\em  Butterfly moves},  on
$\mathcal Q_{D}$ such that

$$    \#\   \{\textrm{Components   of    }   \Omega    E_D(4)\}   \leq
\#\ \left(\Qcal_{D}/\sim\right).
$$

\subsection{Splitting and switching}

We describe two moves, called {\em Butterfly Moves}, to pass from 
Model $A+$ to Model $A-$, and vice et versa.

\subsubsection{Passing from Model $A+$ to Model $A-$}

Let $\Sigma$ be the flat surface associated to some Prym eigenform 
$(X,\omega) \in \Omega  E_D(4)$. Let us assume that  $\Sigma$  admits  a
three-cylinder decomposition in Model $A+$, and let $(C_{i})_{i=1,2,3}$ the cylinders
($C_0$ is the cylinder fixed by  $\rho$ and  $C_1,C_2$  are exchanged by $\rho$).
Observe that there are $3$ saddle connections homologous to the core
curve  of   $C_0$: there are $I_1,I_2$, the boundaries of the cylinder $C_0$, and 
$J=\partial C_1 \cap \partial C_2$ the intersection of the two cylinders $C_{1}$ and $C_{2}$.
Cutting $\Sigma$ along $I_1,I_2$, and $J$, we get three
connected  components  corresponding  to the cylinders $C_0,C_1,C_2$.  The
component  corresponding to  $C_1$ is  a torus  minus two  discs whose
boundary  circles  meet  at   one  point,  the  two  boundary  circles
correspond  respectively to  $I_1$ and  $J$. We  can split  the common
point of  the two circles into  two points, and then glue  the two segments
arising from the former circles together. The resulting surface is a torus
$T_{1}$ with a simple  geodesic segment $I_{1}$ joining two distinct  points. \medskip

We can now describe the move that will switch to a decomposition into Model $A-$. 
Let  $\gamma_1$ be a  simple closed geodesic in  $T_1$ which
does not  meet the interior of $I_1$, then  $\gamma_1$ corresponds to a  simple closed
geodesic  on $\Sigma$ which is contained  in $\overline{C}_1$.  The
simple closed geodesics homotopic to $\gamma_1$ in $\Sigma$ fill out a
simple cylinder $C_{\gamma_1}$  which is included in $\overline{C}_1$.
Since $C_2=\rho(C_1)$,  we also have a  simple cylinder $C_{\gamma_2}$
in the  same direction included  in $\overline{C}_2$. It  follows that
$\Sigma$  admits  a  decomposition  of  type  $A-$  in  the direction
of $\gamma_1$.

\subsubsection{Passing from Model $A-$ to Model $A+$}

Conversely, if we have a  decomposition of $\Sigma$ of type $A-$, then
cutting $\Sigma$ along the boundaries  of $C_1$ and $C_2$, we also get
three connected components. The one  corresponding to $C_0$ is a torus
minus $4$ discs whose boundary circles meet at one point. We can split
this common  point into  $4$ points,  we then get  a once  holed torus
whose boundary consists of $4$ segments divided into two pairs. Gluing
two segments  in each  pair we  finally obtain a  flat torus  with two
marked geodesic segments  having a common endpoint. Note  that the two
geodesic segments are parallel, and  have the same length. We call the
(closed) flat  torus $T$ and  the union of  the two segments  $I$. Let
$\gamma$ be a  simple closed geodesic in $T$ which  does not meet  the interior of $I$,
then  $\gamma$ corresponds to  a simple  closed geodesic  on $\Sigma$,
which is  contained in  $\overline{C}_0$. The simple  closed geodesics
homotopic to $\gamma$ fill out  a simple cylinder in $\Sigma$ which is
invariant by the involution $\rho$,  it follows that $\Sigma$ admits a
decomposition of type $A+$ in the same direction than $\gamma$. \medskip

We  call the operations  of switching  between decompositions  of type
$A+$ and  $A-$ described above  {\em Butterfly moves}. Here  we borrow
the  terms   from~\cite{Mc4},  even   though  the  geometric
interpretation  is  less  clear  in our  situation.   \\
The  switching from a decomposition of  type $A+$ to
another one of type $A-$ will be called a {\em Butterfly move of first
kind}, the inverse switching will be called a {\em Butterfly move of
second kind}.

\subsection{Admissibility}
We first need to know when a Butterfly move can be carried out.  If we
are    to   make    a   Butterfly    move   of    first    kind,   let
$\alpha_1,\beta_2,\alpha_{2,i},  \beta_{2,i}$  be  as  in  Proposition
\ref{NormA1Prop}. Then  there exists a  unique pair $(p,q)  \in \Z^2$,
with               $\gcd(p,q)=1$,               such              that
$\gamma=p\alpha_{2,1}+q\beta_{2,1}$.  Similarly, if we  are to  make a
Butterfly move of second kind, then letting $\alpha_{1,i},\beta_{1,i},
\alpha_2,\beta_2$   as  in   Proposition  \ref{NormA2Prop},   one  has
$\gamma=p\alpha_2+q\beta_2$.  In  both cases,  we call $(p,q)$  the
parameter of the Butterfly move. 
The following lemma is an elementary observation.

\begin{Lemma}
\label{BMCondLm}
The Butterfly  move of both kinds  can be carried out   if    and    only     if    the    prototype
$(w,h,t,e,\varepsilon)$ and the parameter $(p,q)$ satisfy
$$
0<\lambda|q|<w,
$$
or equivalently $(e+4|q|h)^2 <D$. In this case, we say that the Butterfly move is admissible.
\end{Lemma}
\begin{proof}
Identifying  any flat  torus with  a quotient  $\C/L$,  where $L\simeq
\Z^2$ is  a lattice, we  can associate to  every oriented path  on the
torus a unique vector in $\R^2\simeq  \C$. For a Butterfly move of the
first kind (respectively,  second kind), the vector associated  to the segment
$I_1$  (respectively, $I$)  is  $(\lambda,0)$, and  the  vector associated  to
$\gamma_1$ (respectively, $\gamma$) is $(pw+qt,qh)$. Recall that the Butterfly
move is admissible if and only if $\gamma_1$ (respectively, $\gamma$) does not
meet $I_1$ (respectively, $I$). In both situations, this condition is equivalent to

$$0< \left|\det\left(%
\begin{array}{cc}
  \lambda & pw+qt \\ 0 & qh\\
\end{array}%
\right)\right| < \left|\det\left(%
\begin{array}{cc}
  w & t \\ 0 & h \\
\end{array}%
\right)\right| \Leftrightarrow 0< \lambda|q|<w.
$$
To see that this condition is equivalent to  $(e+4|q|h)^2 <D$, recall that $0<
  2\lambda=e+\sqrt{D}$. Thus $-\sqrt{D}<e<e+4|q|h$. To see that $e+4|q|h <
  \sqrt{D}$, we write $8wh = D-e^2=2\lambda(\sqrt{D}-e)$, therefore
$$
\begin{array}{rccl}
 & 8\lambda  |q|h & <  & 8wh = 2\lambda(\sqrt{D}-e)\\  \Leftrightarrow &
  e+4|q|h & < & \sqrt{D}.
\end{array}
$$
\end{proof}

\begin{Definition}
For $q\in \N \setminus\{0\}$ we define $B_q$ the Butterfly move with
parameter  $(1,q)$. We also define $B_\infty$ as the Butterfly  move with
parameter $(0,1)$. 
\end{Definition}

\begin{Remark}
The Butterfly moves $B_1$ and $B_\infty$ are always admissible.
\end{Remark}

\subsection{Coding Butterfly moves}

Having a Butterfly move $B_q$ admissible for some complete prototype
$(w,h,t,e,\varepsilon)$, we obtain  a   new complete  prototype   
$(w',h',t',e',-\varepsilon) $. The goal of this section is to give a formula 
to compute the new prototype from the former one and the parameter of the Butterfly move.

\begin{Proposition}
\label{BM1Prop}
Let $(w,h,t,e,+)\in\Qcal_D$ be a  complete prototype.  Suppose that
the  Butterfly move  $B_q$ is  admissible for  this prototype  and let
$(w',h',t',e',-)$  be the  complete  prototype associated  to the  new
decomposition.

\begin{enumerate}

\item If $q\neq \infty$ then
$$
\begin{cases}
e'=-e-4qh,\\ h'=\gcd(qh,w+qt).
\end{cases}
$$
\item If $ q=\infty$ then
$$
\begin{cases}
e'=-e-4h,\\ h'=\gcd(t,h).
\end{cases}
$$
\end{enumerate}
In    both    cases   $w'$    is    determined    by   the    relation
$D=e^2+8wh={e'}^2+8w'h'$.
\end{Proposition}

\begin{proof}
Let  $\gamma_1=\alpha_{2,1}+q\beta_{2,1}$  ($\gamma_1=\beta_{2,1}$  if
$q=\infty$)  and $C_{\gamma_1}$  be the  cylinder  in $\overline{C}_1$
filled out  by simple closed geodesic freely  homotopic to $\gamma_1$.
Let  $I_{\gamma_1}$ be  a  saddle connection  such that  $I_{\gamma_1}
\subset  \overline{C}_{\gamma_1}$. Remark that  $I_1\cup I_{\gamma_1}$
is  freely  homotopic  to   a  simple  closed  curve  $\eta_1  \subset
\overline{C}_1$                        such                       that
$\Z\gamma_1+\Z\eta_1=\Z\alpha_{2,1}+\Z\beta_{2,1}$.   We   choose  the
orientation  for $\eta_1$ so  that $(\omega(\gamma_1),\omega(\eta_1))$
defines           the           same          orientation           as
$(\omega(\alpha_{2,1}),\omega(\beta_{2,1}))$. First, we set

\begin{itemize}
\item  $\tilde{\alpha}_{2,1}=\gamma_1, \tilde{\alpha}_{2,2}=-\rho(\gamma_1)$, 
$\tilde{\alpha}_2=\tilde{\alpha}_{2,1}+\tilde{\alpha}_{2,2}$, and

\item  $ \tilde{\beta}_{1,1}=\eta_1,  \tilde{\beta}_{2,2}=-\rho(\eta_1)$,   $\tilde{\beta}_2=\tilde{\beta}_{2,1}+\tilde{\beta}_{2,2}$.

\end{itemize}

Then $(\alpha_1,\beta_1,\tilde{\alpha}_2,\tilde{\beta}_2)$ is a symplectic basis of $H_1(X,\Z)^-$. Next, we set

\begin{itemize}

\item    $\alpha'_1=\tilde{\alpha}_2$,

\item    $\beta'_1 = \tilde{\beta}_2+2\alpha_1$,

\item    $\tilde{\alpha}'_2=\alpha_1$, 

\item    $\tilde{\beta}'_2=\beta_1+\tilde{\alpha}_2$.
\end{itemize}
Then  $(\alpha'_1,\beta'_1,\tilde{\alpha}'_2,\tilde{\beta}'_2)$ is another symplectic basis of
 $H_1(X,\Z)^-$. Observe that in this basis
the intersection form is written as $\displaystyle{\left(%
\begin{array}{cc}
  2J & 0 \\ 0 & J\\
\end{array}%
\right)}$.

\medskip

Recall  that we have associated to  $(w,h,t,e,+)$ a unique
generator   of    $\Ord_D$   which    is   written   in    the   basis
$(\alpha_1,\beta_1,\alpha_2,\beta_2)$ as $\DS{T=\left(%
\begin{smallmatrix}
  e & 0 & 2w & 2t \\ 0 & e & 0 & 2h \\ h & -t & 0 & 0 \\ 0 & w & 0 & 0 \\
\end{smallmatrix}%
\right)}$  such that $T^*\omega=\lambda\omega$,           with
$\lambda=\frac{e+\sqrt{D}}{2}>0$. We  now consider separately  the two
cases $q\in \N\setminus\{0\}$ first and then $q=\infty$. \medskip

\begin{enumerate}
\item         If          $q\in         \N,         q>0$,         then
  $\tilde{\alpha}_2=\alpha_2+q\beta_2$.  One  can  choose $\eta_1$  so
  that $\tilde{\beta}_2=\beta_2$. Thus, we have
$$ T=\left(%
\begin{array}{cccc}
  e & 0 & 2w+2qt &  2t \\ 0 & e & 2qh & 2h \\ h & -t  & 0 & 0 \\ -qh &
  w+qt & 0 & 0 \\
\end{array}%
\right)_{(\alpha_1,\beta_1,\tilde{\alpha}_2,\tilde{\beta}_2)}
$$
and
$$T = \left(%
\begin{array}{cccc}
  -2qh  & 0 &  h &  -e-t-2qh \\  0 &  -2qh &  -qh &  w+qt \\  2w+2qt &
  2e+2t+4qh & e+2qh & 0 \\ 2qh & 2h & 0 & e+2qh \\
\end{array}%
\right)_{(\alpha'_{1},\beta'_{1},\tilde{\alpha}'_2,\tilde{\beta}'_2)}.
$$

As  a matter of  fact, the matrix  of the basis  change from
$(\alpha_1,\beta_1,\tilde{\alpha}_2,\tilde{\beta}_2)$                to
$(\alpha'_1,\beta'_1,\tilde{\alpha}'_2,\tilde{\beta}'_2)$ is $\left(%
\begin{smallmatrix}
  0 & 2 & 1 & 0 \\ 0 & 0 & 0 & 1 \\ 1 & 0 & 0 & 1 \\ 0 & 1 & 0 & 0 \\
\end{smallmatrix}%
\right)$. Let $(\alpha'_1,\beta'_1,\alpha'_2,\beta'_2)$ be the  symplectic basis     of      $H_1(X,\Z)^-$ 
 associated to the cylinder decomposition in the direction of $\omega(\alpha'_1)$ (see Figure~\ref{fig:modelA-:basis} and 
Proposition~\ref{NormA2Prop}). Since $(\alpha'_2,\beta'_2)$ and $(\tilde{\alpha}'_2,\tilde{\beta}'_2)$ are related 
by an element of $\SL(2,\Z)$, the matrix of $T$ in the basis $(\alpha'_1,\beta'_1,\alpha'_2,\beta'_2)$ has
 the form $\DS{\left(\begin{smallmatrix}
-2qh\cdot\rm{id}_2 & B \\ 2B^* & (e+2qh)\cdot\rm{id}_2\\
\end{smallmatrix}\right)}$.\\
We set  $T'=T-(e+2qh)$, then $T^*(\omega)=(\lambda-(e+2qh))\omega$. Let us show  that
$\lambda-(e+2qh)>0 \Leftrightarrow  \lambda-e>2qh$. Since $\lambda$ is
an eigenvalue of  $T$, we have $\lambda^2=e\lambda+2wh \Leftrightarrow
\lambda(\lambda-e)=2wh$, therefore
$$    \displaystyle{\begin{array}{lccr}   &    \lambda-e    &   >    &
    2qh\\ \Leftrightarrow & 2wh/\lambda  & > & 2qh\\ \Leftrightarrow &
    w & > & q\lambda\\
\end{array}}
$$ Since the last inequality  is warranted by the admissibility of the
Butterfly move (Lemma \ref{BMCondLm}), we can conclude that $\lambda-(e+2qh)>0$. It follows that $T'$
 is the unique generator of $\Ord_D$ associated to the decomposition in direction $\omega(\alpha'_1)$. 
 Therefore, by Proposition~\ref{NormA2Prop}, up to some appropriate
Dehn twists    $\beta'_1\mapsto     \beta'_1+n\alpha'_1$,  and $\beta'_2\mapsto  \beta'_2+m\alpha'_2$, 
 the matrix of $T'$ in the basis $(\alpha'_1,\beta'_1,\alpha'_2,\beta'_2)$ has the form

$$ T'=T-(e+2qh)=\left(%
\begin{array}{cccc}
  e' & 0 &  w' & t' \\ 0 & e'  & 0 & h' \\ 2h' & -2t'  & 0 & 0
  \\ 0 & 2w' & 0 & 0 \\
\end{array}%
\right)_{(\alpha'_1,\beta'_1,\alpha'_2,\beta'_2)}
$$ 
where $e'=-e-4qh$, and $(w',h',t',e')$ satisfies the conditions in $(\Pcal)$. Note that we must have $h'=\gcd(qh,w+qt)$.

\medskip

\item  If $q=\infty$  then $\gamma_1=\beta_{2,1}$  and one  can choose
  $\eta_1= -\alpha_{2,1}$.  Using the same notations as above we have
$$ T=\left(%
\begin{array}{cccc}
  e & 0 & 2t & -2w \\ 0 & e & 2h &  0 \\ 0 & w & 0 & 0 \\ -h & t & 0 &
  0 \\
\end{array}%
\right)_{(\alpha_{1},\beta_1,\tilde{\alpha}_2,\tilde{\beta}_2)},
$$ and
$$ T=\left(%
\begin{array}{cccc}
  -2h & 0 & 0 & -e+w-2h \\ 0 &  -2h & -h & t \\ 2t & 2e-2w+4h & e+2h &
  0 \\ 2h & 0 & 0 & e+2h \\
\end{array}%
\right)_{(\alpha'_1,\beta'_1,\tilde{\alpha}'_2,\tilde{\beta}'_2)}.
$$   
Then   $T'=T-(e+2h)$ is the generator of $\Ord_D$ associated to the cylinder decomposition
 in direction $\omega(\alpha'_1)$. Let $(\alpha'_1,\beta'_1,\alpha'_2,\beta'_2)$ be the symplectic basis
  of $H_1(X,\Z)^-$ associated to the new decomposition. In this basis, we have

$$ T'=\left(%
\begin{array}{cccc}
  e' & 0 & w' & t' \\ 0 & e' & 0 & h' \\ 2h' & -2t' & 0 & 0 \\ 0 & 2w'
  & 0 & 0 \\
\end{array}%
\right)_{(\alpha'_1,\beta'_1,\alpha'_2,\beta'_2)}
$$ 
where  $e'=-e-4h, h'=\gcd(t,h)$, and $(w',h',t',e')\in\Pcal_D$.   To      see      that
$\lambda'=\lambda-(e+2h)>0$, if  suffices to follow the  same lines as
above,   and   recall   that    by   construction   we   always   have
$\lambda<w$. This completes the proof of the proposition.
\end{enumerate}
\end{proof} 

For Butterfly moves of second kind, we have the same result,

\begin{Proposition}
\label{BM2Prop}
Let $(w,h,t,e,-)\in\Qcal_{D}$ be a complete prototype. Suppose that
the  Butterfly  move  $B_q$  is  admissible for  this  prototype.  Let
$(w',h',t',e',+)$  be the  complete  prototype associated  to the  new
decomposition.

\begin{enumerate}

\item If $q\neq \infty$ then
$$
\begin{cases}
e'=-e-4qh,\\ h'=\gcd(qh,w+qt).
\end{cases}
$$

\item If $ q=\infty$ then
$$
\begin{cases}
e'=-e-4h,\\ h'=\gcd(t,h).
\end{cases}
$$
\end{enumerate}
In    both    cases   $w'$    is    determined    by   the    relation
$D=e^2+8wh={e'}^2+8w'h'$.
\end{Proposition}

\begin{proof}
We briefly  sketch the  proof since it  is similar  to the one  of the
previous   proposition.  Let   $\gamma=\alpha_{2,1}+q\beta_{2,1}$  (or
$\gamma=\beta_{2,1}$ if $q=\infty$).  There exists a saddle connection
$I_\gamma$  which is contained  in the  cylinder $C_\gamma$  such that
$I_1*I_2*I_\gamma$  is  homotopic  to  a simple  closed  curve  $\eta$
satisfying  $\Z\gamma+\Z\eta=\Z\alpha_2+\Z\beta_2$.    We  choose  the
orientation for $\eta$ so that $(\omega(\gamma),\omega(\eta))$ defines
the  same  orientation  as  $(\omega(\alpha_2),\omega(\beta_2))$. We set
$\tilde{\alpha}_2=\gamma, \tilde{\beta}_2=\eta$, and
\begin{itemize}
\item  $\alpha'_1=\tilde{\alpha}_2,$
 
 \item $\beta'_1=\tilde{\beta}_2+\alpha_1$,

\item  $\tilde{\alpha}'_2= \alpha_1$,

\item  $\tilde{\beta}'_2=\beta_1+2\tilde{\alpha}_2$.
\end{itemize}
Then $(\alpha_1,\beta_1,\tilde{\alpha}_2,\tilde{\beta}_2)$ and $(\alpha'_1,\beta'_1,\tilde{\alpha}'_2,\tilde{\beta}'_2)$ 
are symplectic bases of $H_1(X,\Z)^-$. Recall  that  we have  associated  to  the  prototype $(w,h,t,e,-)$  a
generator    $T$    of    $\Ord_D$    represented   in    the    basis
$(\alpha_1,\beta_1,\alpha_2,\beta_2)$ by the matrix $\left(%
\begin{smallmatrix}
  e & 0 & w & t \\ 0 & e & 0 & h \\ 2h & -2t & 0 & 0 \\ 0 & 2w & 0 & 0
  \\
\end{smallmatrix}%
\right)$. Now, we have two cases \medskip

\begin{enumerate}
\item     {\bf    Case     $q\in \N$:}     in     this    case
  $\tilde{\alpha}_2=\gamma=\alpha_2+q\beta_2$,   so   we  can   choose
  $\tilde{\beta}_2=\eta=\beta_2$.         In         the         basis
  $(\alpha'_1,\beta'_1,\tilde{\alpha}'_2,\tilde{\beta}'_2)$,        the
  matrix of $T$ becomes $\left(%
\begin{smallmatrix}
  -2qh &  0 & 2h & -2e-2t-4qh  \\ 0 & -2qh  & -2qh & 2w+2qt  \\ w+qt &
  e+t+2qh & e+2qh & 0 \\ qh & h & 0 & e+2qh \\
\end{smallmatrix}%
\right)$.     Set $T'=T-(e+2qh)$, then $T'$ is the generator of $\Ord_D$ associated to the cylinder decomposition 
in direction $\omega(\alpha'_1)$. Let $(\alpha'_1,\beta'_1,\alpha'_2,\beta'_2)$ be the symplectic basis of 
$H_1(X,\Z)^-$ associated this decomposition (see Proposition~\ref{NormA1Prop}). Then up to some Dehn twists
 $\beta'_1\mapsto \beta'_1+m\alpha'_1$, and $\beta'_2\mapsto \beta'_2+n\alpha'_2$, the matrix of $T'$ in this
  basis has the form

$$
T'=T-(e+2qh)=\left(%
\begin{array}{cccc}
  e' & 0 & 2w' & 2t' \\ 0 & e' & 0 & 2h' \\ h' & -t' & 0 & 0 \\ 0 & w'
  & 0 & 0 \\
\end{array}%
\right)_{(\alpha'_1,\beta'_1,\alpha'_2,\beta'_2)},
$$
where  $e'=-e-4qh,  \;  h'=\gcd(qh,  w+qt)$,  and  $(w',h',t',e')$
satisfies $(\Pcal)$. 
\medskip

\item      {\bf      Case      $q=\infty$:}     in      this      case
  $\tilde{\alpha}_2=\gamma=\beta_2$,      and     we      can     take
  $\tilde{\beta}_2=\eta=-\alpha_2$. The rest  of the proof follows the
  same lines as in  Proposition~\ref{BM1Prop} Case (2). This completes
  the proof of the proposition.

%

\end{enumerate}

\end{proof}

\subsection{Butterfly moves on incomplete prototypes}

It turns out that the equivalence relation on $\Qcal_{D}$ generated by the Butterfly
moves descends to an equivalence  relation on $\Pcal_{D}$. Indeed, we can define $p\sim  p'$ if  there exist
  $\varepsilon,\varepsilon'\in\{\pm\}$ such that  $(p,\varepsilon)\sim (p',\varepsilon')$  in  $\Qcal_D$. In the next
   section we will be interested in the classification of the equivalence  classes of this equivalence relation. For 
   that we will only  consider the Butterfly  moves $B_q, \;
q\in \{1,2,\dots\}\cup\{\infty\}$. {\em A priori}, we  have much more
possible Butterfly  moves, namely  with parameters $(p,q),  p\not= 1$,
but as we  will see, the Butterfly moves $B_q$  are sufficient for our
purpose.      \\
From    Lemma~\ref{BMCondLm}, Propositions~\ref{BM1Prop} and~\ref{BM2Prop},    we     see 
   that    the
admissibility condition, and the transformation rules are the same for 
Butterfly moves  of both  kinds, therefore, we  can regard $B_q$  as a
transformation in $\Pcal_D$, which is defined only on the subset $\left\{ (w,h,t,e)\in \Pcal_D,\; e+4qh  < \sqrt{D} \right\}$.
 This simple observation allows us to work exclusively on $\Pcal_D$, and to use the McMullen's approach in order 
 to obtain an upper bound of the number of equivalence classes.

\section{Components of the space of prototypes}
\label{sec:connectedness:PD}

The main  goal of  this section  is to set  a bound  on the  number of
classes of  the equivalence relation generated by  Butterfly moves
in $\Pcal_D$. This  section is rather independent from  the others and
can  be read  separately.  For  the reader's  convenience,  we briefly
recall  here the  relevant  definitions.

\medskip
The set $\Pcal_D$ is the family of quadruples of integers $(w,h,t,e)$ satisfying
$$
(\Pcal) \left\{
\begin{array}{l}
D  =   e^2  +   8w  h,   \  0  \leq   t  <   \gcd(w,h),\  2h+e   <  w,
\\ 0<w,\ 0<h,\ \gcd(w,h,t,e) = 1.
\end{array}\right.
$$
The elements of $\Pcal_D$  are  called   prototypes.\medskip

The Butterfly  moves introduced in  Section~\ref{sec:butterfly} define a map $B_q$ 
that  can be regarded  as a transformation acting on a subset of $\Pcal_D$. 
We recall the properties of $B_q$.
 
\begin{enumerate}
\item The Butterfly move $B_q$ is defined (we will say {\em admissible}) for all
  prototypes $p\in \Pcal_D$ satisfying $ e+4qh  < \sqrt{D}$ (see Lemma~\ref{BMCondLm}). \medskip
  
\item If $B_q(w,h,t,e) = (w',h',t',e')$ then
\begin{enumerate}
\item $ \left\{
\begin{array}{lll}
 e'=-e-4qh  & \text{  and }h'=\gcd(qh,w+qt),  &  \text{ if  } q  \not=
 \infty \\  e'=-e-4h & \text{  and } h'=\gcd(h,t),  & \text{ if }  q =
 \infty.\\
\end{array}
\right.  $,\medskip

\item    $w'$    is    determined    by    the    relation
  $D=e^{2}+8wh=e'^{2}+8w'h'$.
\end{enumerate}
\end{enumerate}

See Figure~\ref{fig:action:P68} for examples of Butterfly move transformations.

\begin{Remark}
We do not have a formula for $t'$, but it is often possible to compute
$t'$  using  the  conditions  $\gcd(w',h',t',e')=1$ and  $0\leq  t'  <
\gcd(w',h')$. In particular, if $\gcd(w',h')=1$, then $t'=0$.
\end{Remark}

The maps $B_{q}$ generate an  equivalence   relation  $\sim$  on   $\Pcal_D$:  
two prototypes are  equivalent if one can  pass from one  prototype to the
other one by  a sequence of Butterfly  moves. We will  call an equivalence
class of $\sim$ a {\it component}  of $\Pcal_D$.

We can now state the main result of this section

\begin{Theorem}
\label{theo:main:PD}
Let $D>16$  be a  discriminant with  $D \equiv 0,1,4  \mod 8$.  Let us
assume that
$$ D \not \in \{41,68,100\}.
$$ The set $\mathcal P_{D}$ has only one component. The sets $\mathcal
P_{41}$,  $\mathcal P_{68}$  and $\mathcal  P_{100}$ have  exactly two
components.
\end{Theorem}

\begin{Remark}
It  is  straightforward  to   check  that  $\mathcal  P_{D}$  has  two
components       for       $D       \in      \{41,68,100\}$,       see
Section~\ref{sec:four:discri} for  more details. We  present in
Figure~\ref{fig:action:P68} the action of Butterfly Moves on $\mathcal
P_{68}$ and $\mathcal P_{100}$.
\begin{figure}[htbp]
\begin{center}
\begin{minipage}[t]{0.45\linewidth}
\begin{tikzpicture}[fill=blue!20,scale=0.8]
\fill[fill=yellow!80!black!20,even   odd  rule]   (-2.3,-2)  rectangle
(6.3,3.3);  \path (0,2) node(P1)  [rectangle,draw,fill] {$\scriptstyle
  \left(  2,2,1,-6  \right)$}  (0,-1)  node(P4)  [rectangle,draw,fill]
     {$\scriptstyle   \left(  8,1,0,-2   \right)$}   (3.5,2)  node(P2)
     [rectangle,draw,fill]  {$\scriptstyle  \left( 4,1,0,-6  \right)$}
     (2.3,-1)  node(P5)  [rectangle,draw,fill]  {$\scriptstyle  \left(
       8,1,0,2   \right)$}  (4.7,-1)   node(P3)  [rectangle,draw,fill]
     {$\scriptstyle  \left(  4,2,1,-2 \right)$};  \draw[thick,blue,->]
     (P1)   ..   controls   +(left:2cm)   and  (left:1cm)   ..   (P4);
     \draw[thick,blue,->]   (P1)    ..   controls   +(right:2cm)   and
     (right:1cm)  ..  (P4);   \draw[thick,blue,->]  (P4)  ..  controls
     +(up:2cm)  and  +(down:1cm)  .. (P1);  \draw[thick,blue,->]  (P4)
     .. controls +(1,-1)  and (down:2cm) .. (P4); \draw[thick,blue,->]
     (P4) .. controls +(left:3cm) and +(-2,-1) .. (P4);

\draw[thick,blue,->]   (node   cs:name=P2)   --   (node   cs:name=P5);
\draw[thick,blue,<-]  (P2)   ..  controls  +(-1.5,-1)   and  +(-1.5,1)
.. (P5);  \draw[thick,blue,<-] (P2)  .. controls +(-1,-1)  and +(-1,1)
.. (P5);

\draw[thick,blue,->]   (node   cs:name=P2)   --   (node   cs:name=P3);
\draw[thick,blue,<-] (P2) .. controls  +(1.5,-1) and +(1.5,1) .. (P3);
\draw[thick,blue,<-] (P2) .. controls +(1,-1) and +(1,1) .. (P3);

\draw[thick,blue,<-] (P2) .. controls  +(-1,1.5) and +(1,1.5) .. (P2);
\draw[thick,blue,->]  (P2) ..  controls +(right:3cm)  and +(right:2cm)
..  (P3); \draw  (1,0.9) node  {$\scriptstyle  B_{\infty}$} (-0.3,0.9)
node  {$\scriptstyle B_{2}$}  (-1.5,0.9)  node {$\scriptstyle  B_{1}$}
(-1.5,-1.3)    node    {$\scriptstyle    B_{1}$}   (-0.2,-1.7)    node
{$\scriptstyle B_{\infty}$};

\draw  (2.8,3) node {$\scriptstyle  B_{3}$} (3,0)  node {$\scriptstyle
  B_{1}$} (4,0) node  {$\scriptstyle B_{2}$} (5,0) node {$\scriptstyle
  B_{\infty}$}  (5,1.3)  node   {$\scriptstyle  B_{1}$}  (2.1,0)  node
       {$\scriptstyle B_{\infty}$} (2,1.3) node {$\scriptstyle B_{1}$}
       (6,1.6)   node   {$\scriptstyle   B_{\infty}$};  \draw   (2,-2)
       node[below] {$\mathcal P_{68}$};
\end{tikzpicture}
\end{minipage}
\begin{minipage}[t]{0.45\linewidth}
\begin{tikzpicture}[fill=blue!20,scale=0.8]
\fill[fill=yellow!80!black!20,even   odd  rule]   (-2.3,-2)  rectangle
(6.3,3.3);  \path (0,2) node(P1)  [rectangle,draw,fill] {$\scriptstyle
  \left(  4,2,1,-6  \right)$}  (0,-1)  node(P4)  [rectangle,draw,fill]
     {$\scriptstyle   \left(  12,1,0,-2  \right)$}   (3.5,2)  node(P2)
     [rectangle,draw,fill]  {$\scriptstyle  \left( 8,1,0,-6  \right)$}
     (2.3,-1)  node(P5)  [rectangle,draw,fill]  {$\scriptstyle  \left(
       6,2,1,-2  \right)$}   (4.7,-1)  node(P3)  [rectangle,draw,fill]
     {$\scriptstyle  \left(  12,1,0,2 \right)$};  \draw[thick,blue,->]
     (P1)   ..   controls   +(left:2cm)   and  (left:1cm)   ..   (P4);
     \draw[thick,blue,->]   (P1)    ..   controls   +(right:2cm)   and
     (right:1cm)  ..  (P4);   \draw[thick,blue,->]  (P4)  ..  controls
     +(up:2cm)  and  +(down:1cm)  .. (P1);  \draw[thick,blue,->]  (P4)
     .. controls +(1,-1)  and (down:2cm) .. (P4); \draw[thick,blue,->]
     (P4) .. controls +(left:3cm) and +(-2,-1) .. (P4);

\draw[thick,blue,->]   (node   cs:name=P2)   --   (node   cs:name=P5);
\draw[thick,blue,<-]  (P2)   ..  controls  +(-1.5,-1)   and  +(-1.5,1)
.. (P5);  \draw[thick,blue,<-] (P2)  .. controls +(-1,-1)  and +(-1,1)
.. (P5);

\draw[thick,blue,->]   (node   cs:name=P2)   --   (node   cs:name=P3);
\draw[thick,blue,<-] (P2) .. controls  +(1.5,-1) and +(1.5,1) .. (P3);
\draw[thick,blue,<-] (P2) .. controls +(1,-1) and +(1,1) .. (P3);

\draw[thick,blue,<-] (P2) .. controls  +(-1,1.5) and +(1,1.5) .. (P2);
\draw[thick,blue,->]  (P2) ..  controls +(right:3cm)  and +(right:2cm)
..  (P3); \draw  (1,0.9) node  {$\scriptstyle  B_{\infty}$} (-0.3,0.9)
node  {$\scriptstyle B_{2}$}  (-1.5,0.9)  node {$\scriptstyle  B_{1}$}
(-1.5,-1.3)    node    {$\scriptstyle    B_{1}$}   (-0.2,-1.7)    node
{$\scriptstyle B_{\infty}$};

\draw  (2.8,3) node {$\scriptstyle  B_{3}$} (3,0)  node {$\scriptstyle
  B_{2}$} (4,0) node  {$\scriptstyle B_{1}$} (5,0) node {$\scriptstyle
  B_{\infty}$}  (5,1.3)  node   {$\scriptstyle  B_{1}$}  (2.1,0)  node
       {$\scriptstyle B_{\infty}$} (2,1.3) node {$\scriptstyle B_{1}$}
       (6,1.6)   node   {$\scriptstyle   B_{\infty}$};  \draw   (2,-2)
       node[below] {$\mathcal P_{100}$};
\end{tikzpicture}
\end{minipage}
\caption{
\label{fig:action:P68}
Action of  Butterfly moves on  the set of prototypes  $\mathcal P_{D}$
for $D=68$ and $D=100$.  }
\end{center}
\end{figure}
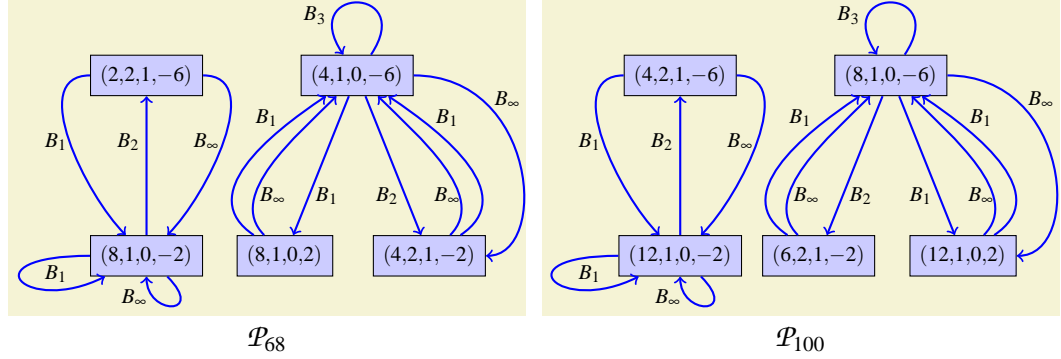
\end{Remark}

%
%


\subsection{Reduced prototypes}\label{sec:reduced:prot}
When $D$  is large,  the set $\Pcal_D$ is very big so that it is not easy to work 
directly with $\Pcal_D$. This problem is avoided by using reduced prototypes:
We say that $p=(w,h,t,e) \in \Pcal_D$ is {\em reduced} if
$h=1$ (in particular $t=0$). There is no loss of generality since we have

\begin{Proposition}
Every prototype is equivalent to a reduced prototype.
\end{Proposition}

\begin{proof}
The    proof     parallels    the    one     of    Theorem~$8.2$    in
McMullen~\cite{Mc4}. For the sake  of completeness we briefly give the
details here.  Let $p=(w,h,t,e)$ that minimize  the value of  $h$ in a
given component. We  claim that $p$ is reduced. For  that we will show
that $h$ divides  $w,t$ and $e$ so that  $h$ divides $\gcd(w,h,t,e)=1$
by definition of $\mathcal P_{D}$.
\begin{itemize}
\item  Since $B_{1}(p)=(w',h',t',e')$  where $h'=\gcd(w,h)\geq  h$ one
  has $h=h'$ divides $w$.
\item  Since $B_{\infty}(p)=(w',h',t',e')$ where  $h'=\gcd(t,h)\geq h$
  one has $t<h=h'$ and $t=0$.
\end{itemize}
Let $B_{1}(p)=(w',h,t',-e-4h)$. Now  from the relation $e'^{2}+8w'h' =
e^{2}+8wh$ one  can deduce $w' =  w-e-2h$.  But since  $h$ divides $w$
and $w'$, $h$ divides also $e$. The claim is then proven.
\end{proof}

It  will be useful  to parametrize  the set  of reduced  prototypes as
follows:
\begin{Definition}
Let  $\mathcal S_{D}  = \{  e\in \Z,\  e^{2} \equiv  D \mod  8,\ e^{2}
\textrm{ and } (e+4)^{2} < D \}$. Each element $e \in \SD$ gives rise
to a reduced prototype  $[e] := (w,1,0,e) \in \Pcal_D$, where $w=(D-e^2)/8$.
\end{Definition}

We equip $\SD$ with the relation $e\sim e'$ if $[e']= B_q([e])$, for some $q\in
\N \cup \{\infty\}$. Note that this condition implies that $e'=-e-4q$,
and  $\gcd(w,q)=1$,  where $w  =(D-e^{2})/8$,  when $q\in  \N\setminus
\{0\}$, and  $e'=-e-4$, when $q=\infty$.  An equivalence class  of the
equivalence  relation generated  by  this relation  is  called a  {\it
  component} of $\SD$. Clearly, if $e\sim e'$ in $\SD$, then $[e]$ and
$[e']$ are equivalent in $\Pcal_D$. Note that the converse is not necessarily true, it can happen that 
$[e]\sim [e']$ in $\Pcal_D$, but $e$ and $e'$ do not belong to the same equivalence class in $\SD$.
Theorem~\ref{theo:main:PD} follows mainly from the following

\begin{Theorem}
\label{theo:main:SD}
Let $D>16$  be a  discriminant with  $D \equiv 0,1,4  \mod 8$.  Let us
assume that
$$                    D                    \not                    \in
\{20,36,41,73,97,112,148,196,244,292,304,436,484,676,1684\}.
$$ Then the set $\mathcal S_{D}$ is non empty and has either
\begin{itemize}
\item two components: $\{e\in \SD,\;  e \equiv 2 \mod 8 \}$ and
  $\{e\in \SD,\; e\equiv -2 \mod 8\}$, if $D \equiv 4 \mod 16$, or
\item only one component.
\end{itemize}
\end{Theorem}

\begin{Remark}
\label{rk:two:components}
There  is a simple  congruence condition  that explains  why $\mathcal
S_{D}$ has (at least) two  components when $D\equiv 4 \mod 16$. Indeed
in that case if $e\sim f$ then $e \equiv f \mod 8$.
\end{Remark}

\subsection{Exceptional cases}

Our number-theoretic analysis of  the connectedness of $\mathcal S_{D}$
only applies when $D$ is sufficiently large ({\it e.g.} $D \geq 83^{2}$). On
one  hand  it is  feasible  to compute  the  number  of components  of
$\mathcal S_{D}$ when  $D$ is reasonably small. This  reveals the $15$
exceptional    cases   of   Theorem~\ref{theo:main:SD},    listed   in
Table~\ref{tab:exceptional:cases}. On  the other hand, using computer assistance, one can easily
prove the following

\begin{Lemma}
\label{lm:exceptionnal:cases}
Theorem~\ref{theo:main:SD} is true for all $D \leq 83^{2} = 6889$.
\end{Lemma}

\begin{table}[htbp]
$$
\begin{array}{|c|c||c|c||c|c|}
\hline   D  &   \textrm{Components  of   }  \mathcal   S_{D}  &   D  &
\textrm{Components of }  \mathcal S_{D} & D &  \textrm{Components of }
\mathcal S_{D} \\ \hline 20 & 1 & 112 &  2 & 304 & 2 \\ 36 & 1 & 148 &
3 & 436 & 3 \\ 41 & 2 & 196 & 3  & 484 & 3 \\ 73 & 2 & 244 & 3 & 676 &
3 \\ 97 & 2 & 292 & 3 & 1684 & 3\\ \hline
\end{array}
$$
\caption{
\label{tab:exceptional:cases}
Exceptional cases of Theorem~\ref{theo:main:SD}.}
\end{table}

\subsection{Proof of Theorem~\ref{theo:main:PD}}\hfill\\
We     first    show     how     Theorem~\ref{theo:main:SD}    implies
Theorem~\ref{theo:main:PD}.

\begin{proof}[Proof of Theorem~\ref{theo:main:PD}]
Obviously  when $\mathcal  S_{D}$  has only  one  component, there  is
nothing to prove. Thus we only need to consider the cases
$$ \left\{ \begin{array}{l} D > 16 \qquad \textrm{and} \qquad D \equiv
  4           \mod           16,           \\          D           \in
  \{73,97,112,148,196,244,292,304,436,484,676,1684\}.
\end{array}
\right.
$$  We  first examine  the  general  case,  and then  the  exceptional
cases.  Since  for $D  \leq  100$,  Theorem~\ref{theo:main:PD} can  be
checked by hand, let us assume $D > 100$ and $D \equiv 4 \mod 16$. The
idea is  to connect  the two components  of $\mathcal S_{D}$  by using
non-reduced elements of $\mathcal P_{D}$. To be more precise one needs
to  connect $e\in  \mathcal S_{D}$  to some  $e'\in  \mathcal S_{D}$
where  $e\not \equiv  e' \mod  8$, by  using butterfly  moves $B_{q}$,
$q\in       \N        \cup       \{\infty\}$       (compare       with
Remark~\ref{rk:two:components}).

\begin{enumerate}
\item[(1)] {\bf First case: $D = 4 + 16k$, $k$ odd.} \\ Since $D>100$,
  we   have  $k\geq   7$.  We   start  from   the   reduced  prototype
  $(2k-4,1,0,-6)\in \Pcal_{D}$ or equivalently $e=-6\in \SD$.  Observe
  that  $B_2$  is  admissible  since  $e+4\cdot 2  =  2  <  \sqrt{16}<
  \sqrt{D}$.  Applying   the  followings  Butterfly   moves:  $B_{2}$,
  $B_{\infty}$ and $B_{1}$ in this order gives:
$$                        (2k-4,1,0,-6)\overset{B_{2}}{\longrightarrow}
  (k,2,0,-2)\overset{B_{\infty}}{\longrightarrow}(k-2,2,0,-6)
  \overset{B_{1}}{\longrightarrow}(2k,1,0,-2).
$$ Thus  $[-6]$ and $[-2]$ are  equivalent in $\Pcal_D$,  and $-6 \not
  \equiv -2 \mod 8$ as desired. \medskip

\item[(2)] {\bf Second case: $D = 4 + 32k$, $k$ odd.}\\ We have $k\geq
  5$,  since  $D>100$.  This  time  we  will  start from  the  reduced
  prototype  $(4k,1,0,2)\in \Pcal_{D}$  or equivalently  $e=2\in \SD$.
  We  apply  the  followings  Butterfly Moves:  $B_{2}$,  $B_{2}$  and
  $B_{1}$ in this order.  For the first move $q=2$ is admissible since
  $e+4\cdot 2  = 10  < \sqrt{D}$.  For the second  move $q=2$  is also
  admissible: $e+4\cdot 2 = -2 < \sqrt{D}$.
$$                           (4k,1,0,2)\overset{B_{2}}{\longrightarrow}
  (2k-6,2,1,-10)\overset{B_{2}}{\longrightarrow}(2k-2,2,1,-6)
  \overset{B_{1}}{\longrightarrow}(4k,1,0,-2).
$$ Thus $[2]$ is connected to  $[-2]$ in $\Pcal_D$, and $2 \not \equiv
  -2 \mod 8$ as desired. \medskip

\item[(3)] {\bf Third case: $D = 4 + 32k$, $k$ even.}\\ We have $k\geq
  4$. In  this case,  since $D\not \in  \{68,100\}$, one  has $D>100$.
  This time we will start from the reduced prototype $(4k-4,1,0,-6)\in
  \Pcal_{D}$ or  equivalently $e=-6\in \SD$.  We  apply the followings
  Butterfly moves:  $B_{4}$, $B_{\infty}$  and $B_{1}$ in  this order.
  The first move corresponding  to $q=4$ is admissible since $e+4\cdot
  4 = 10 < \sqrt{D}$.
$$                        (4k-4,1,0,-6)\overset{B_{4}}{\longrightarrow}
  (k-3,4,0,-10)\overset{B_{\infty}}{\longrightarrow}(k-1,4,0,-6)
  \overset{B_{1}}{\longrightarrow}(4k-12,1,0,-10)
$$ Thus  $[-6]$ is connected to  $[-10]$ in $\Pcal_{D}$,  and $-6 \not
  \equiv -10 \mod 8$ as desired. \medskip

\item[(4)]       {\bf       Exceptional       cases:      $D       \in
  \{73,97,112,148,196,244,292,304,436,484,676,1684\}.$}     \\     The
  strategy is  the same  as above. We  have collected  the information
  into              Table~\ref{table:exceptionnal:link}             in
  Appendix~\ref{appendix:exceptional}
  page~\pageref{appendix:exceptional}. 
  \label{Table:explanation}
  We explain  here the first line
  of this table.
$$
\begin{array}{|c|c|c|}
\hline D & \textrm{Components of $\mathcal S_{D}$} & \textrm{Butterfly
  Moves} \\ \hline \hline 73  & \{1,-5\} \textrm{ and } \{-1,-3,3,-7\}
&                 [-5]                \overset{B_{3}}{\longrightarrow}
(1,3,0,-7)\overset{B_{\infty}}{\longrightarrow}(2,3,0,-5)\overset{B_{1}}{\longrightarrow}
       [-7] \\ \hline
\end{array}
$$ The  first two  columns represent the  discriminant $D=73$  and the
components  of  $\mathcal   S_{73}$:  a  representative  elements  are
{\it e.g.}  $e=-5$  and $e=-7$.  In  the last  column  we  encode the  moves
connecting    the   two    corresponding    reduced   prototypes    in
$\Pcal_D$.  Hence,  whereas  $\mathcal  S_{73}$  has  two  components,
$\mathcal P_{73}$ has only one.
\end{enumerate}
The proof of our theorem is now complete.
\end{proof}

We can now turn into the proof of Theorem~\ref{theo:main:SD}. To prove
this  theorem,  we  use  almost   the  same  ideas  as  the  proof  of
Theorem~$10.1$ in~\cite{Mc4}, and do not wish to claim any originality.

\subsection{Small values of $q$} Surprisingly it is possible to show
that Theorem~\ref{theo:main:SD}  holds for most values of  $D$ only by
using   Butterfly   moves  $B_q$   with   small   $q$,  namely   $q\in
\{1,2,3,5,7\}$. If $q$ is a prime  number, we will use the following two
operations
$$ \left\{
\begin{array}{lll}
F_{q}(e) &=& e + 4(q-1), \\ F_{-q}(e) &=& e - 4(q-1).
\end{array}
\right.
$$


These two maps are useful to us, since we have
\begin{Proposition}
Let $e\in \mathcal S_{D}$, and assume that $q$ is an odd prime.
\begin{enumerate}
\item If $F_{q}(e)  \in \mathcal S_{D}$ and $D  \not \equiv e^{2} \mod
  q$ then $e \sim F_{q}(e)$.
\item If $F_{-q}(e)  \in \mathcal S_{D}$ and $D  \not \equiv (e+4)^{2}
  \mod q$ then $e \sim F_{-q}(e)$.
\end{enumerate}
\end{Proposition}

\begin{proof}
It  suffices  to  remark  that  $[F_q(e)]$  (resp.  $[F_{-q}(e)]$)  is
obtained   from   $[e]$   by   the   sequence   of   Butterfly   moves
($B_q,B_\infty$)   (resp.   ($B_\infty,B_q$)),   and  the   respective
conditions ensure the admissibility of the corresponding sequence.
\end{proof}

The  next   proposition  guaranties  that,  under   some  rather  mild
assumptions, one has $e \sim F_3(e)=e+8$.

\begin{Proposition}
\label{prop:equiv:step8}
Let $e\in  \mathcal S_{D}$  and let us  assume that $e-24$  and $e+32$
also belong to $\mathcal S_{D}$. Then one of the following two holds:
\begin{enumerate}
\item $e \sim e+8$, or
\item  $(D,e)$ is  congruent  to $(4,-10)$  or  $(4,-2)$ when  reduced
  modulo $105 = 3\cdot5\cdot7$.
\end{enumerate}
\end{Proposition}

\begin{proof}
We say that a sequence of integers $(q_1,q_2,\dots,q_n)$ is a strategy
for $(D,e)$ if for any $i=1,\dots,n-1$ the following holds:
$$   \left\{    \begin{array}{l}   e_{i+1}=F_{q_i}(e_i)   \in    e   +
  \{-24,-16,-8,0,8,16,24,32\} \textrm{ (where } e_1=e), \textrm{ and }
  \\ q_i \textrm{ is admissible for } (D,e_i), \\ e_{n} = e+8.
\end{array}\right.
$$ For instance,  if $(D,e) \equiv (0,3) \mod 105$  then $(5,-3)$ is a
strategy. Indeed  letting $e = 3$ we  see that $3 \sim  F_{5}(3) = 19$
since $5$  is admissible for $(D,3)$.  And $19 \sim F_{-3}(19)  = 11 =
3+8$   since  $-3$  is   admissible  for   $(D,19)$.  Hence   $3  \sim
3+8$.  \medskip

Thus in order to prove the proposition we only need to give a strategy
for every pair $(D,e) \mod 105$  with the two exceptions stated in the
theorem. In fact each of the  $105^2-2$ cases can be handled by one of
the following $12$ strategies.

\begin{enumerate}
\item There  are $7350$  pairs $(D,e)$ for  which $q=3$  is admissible
  ({\it i.e.}  $D\not \equiv  e^2 \mod 3$).  Since $F_3(e) =  e+8$ the
  sequence $(3)$ is a strategy for all of these cases.

\item Among the $105^2-2-7350=3673$  remaining pairs, there are $1960$
  pairs $(D,e)$ for which the sequence $(5,-3)$ is a common strategy.

\item We can continue in order to find strategies for all remaining
  pairs $(D,e)$ but two: $(4,-10)$ and $(4,-2)$. We find respectively the strategies:
$$
\begin{array}{l}
(7,-5),\  (-3,5),  \ (-5,7),  \\  (5,3,-5),  \  (-5,3,5), \  (5,5,-7),
  \ (-7,5,5), \ (-3,7,-3), \\ (-5,3,7,-3), \ (-3,7,3,-5).
\end{array}
$$
\end{enumerate}
Note  that the  conditions  that  $e-24$ and  $e+32$  belong to  $\SD$
guaranty the admissibility of the strategies. This completes the proof
of the proposition.
\end{proof}

\begin{Remark}
Since for  $(D,e) \equiv  (4,-2) \mod 105$  one has $D  \equiv (e+4)^2
\mod  105$, even  though one can enlarge the set  of primes to be
used in the strategies, there is no hope to get a similar conclusion to Proposition~\ref{prop:equiv:step8}
without the second case.
\end{Remark}

\begin{Remark}
\label{rk:crietrion:ends}
A simple criterion to be not  close to the ends of $\mathcal S_{D}$ is
the following.
\begin{center}
If $f\in \mathcal  S_{D}$ then for any $e>f,  \qquad (e+36 < \sqrt{D})
\implies (e+32\in \mathcal S_{D})$.
\end{center}
Indeed $e+32\in \mathcal S_{D}$ if and only if $(e+32)^{2} <
D$ and $(e+36)^{2} < D$. Thus the claim is obvious if $e+32\geq0$. Now
if $e < -32$ then since $e>f$ the inequalities
$$ 0> e + 32 > f + 32 > f \qquad \textrm{and} \qquad -(f+4)> 4> e + 36
> f + 36 > f + 4
$$ implies
$$ (e + 32)^{2} < f^{2} <  D \qquad \textrm{and} \qquad (e + 36)^{2} <
(f + 4)^{2} < D.
$$
\end{Remark}

\medskip

Let us define $\mathcal  T_{D} = \{  e \in \mathcal  S_{D}, \ e-24 \textrm{  and }
e+32\in \mathcal S_{D} \}$. 
The next proposition asserts that if $D$ is large then assumption of Proposition~\ref{prop:equiv:step8} 
actually holds.

\begin{Proposition}
\label{prop:ends}
If $D\geq 55^2$  then every element of $\mathcal  S_{D}$ is equivalent
to an element of $\mathcal T_{D}$.
\end{Proposition}

The proof will  use the following theorem (the notations have been adapted to our
situation)

\begin{NoNumberTheorem}[McMullen~\cite{Mc4} Theorem~$9.1$]
For any integer $w > 1$  there is an integer $q>1$ relatively prime to
$w$ with
$$ 1 < q < \cfrac{3\log(w)}{\log(2)}.
$$
\end{NoNumberTheorem}

\begin{proof}[Proof of Proposition~\ref{prop:ends}]
Let $f\in \mathcal S_{D}$. Since $f \sim -f - 4$ we can assume $f \leq
-2$.  If  $f>-6$ then  the proposition is  clearly true,  therefore we
only have to consider the case $f\leq -6$. Observe that if $ f\leq -6$
then  $(f+32)^{2} \leq (f-20)^{2}$  and $(f+36)^{2}  \leq (f-24)^{2}$,
hence $f  - 24 \in  \SD$ which implies $f  + 32 \in \mathcal  S_{D}$. Assume
that

\begin{equation}
\label{eq:f}
f ^{2} < D \leq (f-24)^{2}.
\end{equation}
We will  show that there always exists  $e>f$ with $e\sim f$
and  $e+32\in \mathcal S_{D}$,  or equivalently  $e+36 <  \sqrt{D}$ by
Remark~\ref{rk:crietrion:ends}. If $e-24 \not \in \mathcal S_{D}$ then
by definition, $e$ satisfies  the inequalities~(\ref{eq:f}) and thus we
can repeat the argument by replacing $f$ by $e$. \medskip

Since  $ D\geq 55^2$  we have $f\leq 24-55=-31$.  Now assume
that  there exists  some  prime $q  \leq  13$ such  that $\gcd(w,q)  =
1$. Then $f \sim F_{q}(f) > f$ and
\begin{equation*}
F_q(f)+36= f+4(q-1)+36 \leq -31 +48+36=53 < 55 \leq \sqrt{D}.
\end{equation*}
Hence $e = F_{q}(f)$ is convenient. \medskip

Thus assume that  $w $  is divisible  by all  primes $p\leq
13$.  Then $D  \geq  8 \cdot  w  \geq 10^{5}$.   Pick  an integer  $q$
relatively prime to $w$ such that
$$ 1 < q < \frac{3\log(w)}{\log(2)} \leq 5\log(D).
$$ Now $f\sim F_{q}(f)$ where
$$ f < F_{q}(f) = f + 4(q-1) < 20 \cdot \log(D).
$$ Since for $D \geq 10^{5}$, we have
$$ F_q(f)+36 < 20 \cdot \log(D) + 36 < \sqrt{D}.
$$ This completes the proof of Proposition~\ref{prop:ends}.
\end{proof}

\subsection{Case $D \equiv 4 \mod 105$}

From   Proposition~\ref{prop:equiv:step8},  we   know   that,  if   $D
\not\equiv   4  \mod  105$,   then  $e   \sim  e+8$,   whenever  $e\in
\mathcal{T}_D$,  but if  $D\equiv 4  \mod 105$,  we do  not  have this
property for  all $e\in \mathcal{T}_D$,  namely when $e  \equiv -10,-2
\mod 105$. Assume that $D \equiv 4 \mod 105$, we define
$$ \mathcal U_{D} =  \{ e \in \mathcal T_{D}, \ e  \not \equiv -2 \mod
105\},
$$

\begin{Lemma}
\label{lm:eq:UD}
For  $D  >  83^{2}  =  6889$  all elements  of  $\mathcal  S_{D}$  are
equivalent to an element of $\mathcal U_{D}$.
\end{Lemma}

%

\begin{proof}

%
Let $e\in \mathcal S_{D}$. Since $D>83^2$, Proposition~\ref{prop:ends}
implies  that one  can assume  $e\in  \mathcal T_{D}$.  Let us  assume
$e\not \in  \mathcal U_{D}$, {\it i.e.}   $e \equiv -2  \mod 105$, one
can  assume $e\leq  -2$ since  $e\sim -e-4$.  To prove the lemma, we need the following

\begin{Lemma}
\label{eq:claim}
For $D > 83^{2}$ there exists $q\not \equiv 1 \mod 105$ such that
$$ \gcd(w,q) = 1, \textrm{ and } 4q + 31 < \sqrt{D}.
$$
\end{Lemma}

Let us first complete the proof of Lemma~\ref{lm:eq:UD}. According to Lemma~\ref{eq:claim}, we can pick some  $q$  such that
$$
\gcd(w,q)=1 \text{ and } F_{q}(e) + 36 = e+4(q-1) + 36 = e + 4q + 32 \leq 4q + 30 < \sqrt{D}
$$ 
Thanks to Remark~\ref{rk:crietrion:ends},  we know that $F_{q}(e) +
36< \sqrt{D}$ implies $F_{q}(e) + 32 \in \mathcal S_{D}$. Consequently
$F_{q}(e) \in \mathcal T_{D}$.  Since $F_{q}(e) - e \equiv 4(q-1) \not
\equiv  0  \mod  105$ we  have  $F_{q}(e)  \not  \equiv -2  \mod  105$
i.e.  $F_{q}(e)  \in  \mathcal  U_{D}$.  We conclude  by  noting  that
$\gcd(w,q) = 1$,  which implies $e\sim F_{q}(e)$. Lemma~\ref{lm:eq:UD}
is now proven.
\end{proof}

To complete the proof of our statement, it remains to show

\begin{proof}[Proof of Lemma~\ref{eq:claim}]
One has to show that
\begin{equation}
\label{eq:lemma}
\left\{ \begin{array}{l}  \gcd(w,q) = 1,  \\ q\not \equiv 1  \mod 105,
  \\ 4q + 31 < \sqrt{D}.
\end{array} \right.
\end{equation}
Since  $D > 83^{2}$  the last  two conditions  of~(\ref{eq:lemma}) are
automatic for $q =  2,\ 3,\ 5,\ 7,\ 11$ and $13$.  Thus one can assume
$w$  is divisible  by  $30030 =  2\cdot 3  \cdot  5 \cdot  7 \cdot  11
\cdot13$.  But in this  case $\sqrt  D=\sqrt{e^{2}+8\cdot w\cdot  h} >
490$. \medskip

Again, the last two conditions  are fulfilled for all primes less than
$114$; thus  the claim  is proven  unless $w$ is  divisible by  all of
these $30$ primes, in which case we have $w \geq 10^{46}$.\medskip

To find   a   good    $q$   satisfying   the   first   condition
of~(\ref{eq:lemma}),  we will  use the  Jacobsthal's  function $J(n)$,
that  is  defined  to  be  largest gap  between  consecutive  integers
relatively prime  to $n$ ({\it e.g.} $J(10)=7-3=4$).   A convenient estimate
for  $J(n)$ is provided  by Kanold~\cite{Kanold1967}:  If none  of the
first   $k$   primes   divide   $n$,   then   one   has   $J(n)   \leq
n^{\log(2)/\log(p_{k+1})}$   where    $p_{k+1}$   is   the   $(k+1)th$
prime. \medskip

We will also use  the following inequality that can be found
in~\cite{Mc4}   (Theorem~$9.4$):  \\  For   any  $a,w,n\geq   1$  with
$\gcd(a,n) =  1$ there is a  positive integer $q \leq  n J(w//n)$ such
that
$$
q \equiv a \mod n \textrm{ and } \gcd(q,w) = 1,
$$  
where $w//n$  is obtained  by removing  from $w$  all  primes that
divide $n$. \bigskip

Applying  the above inequality with $a=13$ and  $n=210$, one can
find a positive integer $q$ satisfying

$$
q \leq 210J(w//210),
$$
with $\gcd(w, q) = 1$ and  $q \equiv 13 \mod 210$. In particular $q
\not  \equiv   1  \mod  105$   and  thus  the  first   two  conditions
of~(\ref{eq:lemma})   are  satisfied.   Let  us   see  for   the  last
condition. \medskip
 
\noindent Since the  first prime $p_{k+1}$ that divide  $w//210$ is at
least $13$, Kanold's estimates gives
$$ J(w//210) \leq (w//210)^{\log(2)/\log(p_{k+1})} \leq (w//210)^{1/3}
\leq w^{1/3}.
$$ 
Hence
$$
4\cdot q + 31 \leq 4\cdot 210 \cdot w^{1/3} + 31.
$$
But since $w > 10^{46}$ we have:

$$
4\cdot 210 \cdot w^{1/3} + 31 \leq w^{1/2} \leq \sqrt{D}.
$$ 
The lemma is proven.
\end{proof}


\subsection{Proof of Theorem~\ref{theo:main:SD}}\hfill
One can  assume that $D\geq 83^2$,  since by Lemma~\ref{lm:exceptionnal:cases}
the theorem is true  for $D<83^2$.

\subsubsection{Case $D\not \equiv 4 \mod 105$}

\begin{proof}
Thanks  to Proposition~\ref{prop:ends},  every component  of $\mathcal
S_{D}$ meets $\mathcal T_{D}$. Since $D=e^2+8w$ the possible values of
$D$ modulo $8$ are
$$
D \equiv 0,1,4 \mod 8.
$$
We will examine each case separately.\medskip

\noindent  {\bf Case one:  $D\equiv 0  \mod 8$.}  Let us  consider the
partition  $\mathcal  T_{D}   =  \mathcal  T^{0}_{D}  \sqcup  \mathcal
T^{1}_{D}$ where
$$ \mathcal  T^{i}_{D} = \{  e \in \mathcal  T_{D},\ e \equiv  4i \mod
8\}.
$$  By  Proposition~\ref{prop:equiv:step8} we  have  $e  \sim  e +  8$
whenever $e$ and $e + 8$  are both in $\mathcal T_{D}$.  Therefore all
elements of  $\mathcal T^{0}_{D}$ are equivalent, as  are all elements
of  $\mathcal  T^{1}_{D}$.   Thus Proposition~\ref{prop:ends}  implies
$\mathcal  S_{D}$  has  at  most  two  components.  But  $B_1(0)  \sim
0-4\times  1 =  -4$ thus  $0\in  \mathcal T^{0}_{D}$  is connected  to
$-4\in \mathcal T^{1}_{D}$. \bigskip

\noindent  {\bf Case two:  $D\equiv 4  \mod 8$.}  Let us  consider the
partition  $\mathcal  T_{D}   =  \mathcal  T^{0}_{D}  \sqcup  \mathcal
T^{1}_{D}$ where
$$ \mathcal T^{i}_{D}  = \{ e \in \mathcal T_{D},\  e \equiv 6+4i \mod
8\}.
$$ Again Propositions~\ref{prop:equiv:step8} and~\ref{prop:ends} imply
$\mathcal S_{D}$ has at most  two components. There are two sub-cases:
$D \equiv 4  \textrm{ or } 12  \mod 16$.  In the first  case there are
actually  two   components  (see  Remark~\ref{rk:two:components}).  So
assume $D \equiv 12 \mod 16$. Then $2\in \mathcal T^{1}_{D}$ and since
$w  = (D-2^2)/8$  is odd,  one has  $B_{2}(2) \sim  -2-4\times 8=-10$.
Hence we have connected  $2\in \mathcal T^{1}_{D}$ to $-10\in \mathcal
T^{0}_{D}$. \bigskip

\noindent {\bf  Case three: $D\equiv 1  \mod 8$.} Let  us consider the
partition  $\mathcal  T_{D}   =  \mathcal  T^{0}_{D}  \sqcup  \mathcal
T^{1}_{D}\sqcup \mathcal T^{2}_{D}\sqcup \mathcal T^{3}_{D}$ where
$$ \mathcal T^{i}_{D} = \{ e \in \mathcal T_{D},\ e \equiv 1 + 2i \mod
8\}.
$$ Again Propositions~\ref{prop:equiv:step8} and~\ref{prop:ends} imply
$\mathcal S_{D}$ has at most  four components. We will connect each of
these sets  by specific butterfly  moves.  First observe  that $B_1(1)
\sim -1-4\times  1 = -5\in  \mathcal T^{1}_{D}$. This  shows that $\mathcal
T^{0}_{D}$ is connected to $\mathcal T^{1}_{D}$. \medskip

\noindent  The same argument  shows $B_1(5)\sim  -5-4\times 1  = -9\in
\mathcal  T^{3}_{D}$.   Thus  $\mathcal  T^{2}_{D}$  is  connected  to
$\mathcal T^{3}_{D}$. \medskip

\noindent We need now to connect $\mathcal{T}^0_D\cup \mathcal{T}^1_D$
with $\mathcal{T}^2_D\cup\mathcal{T}^3_D$. We have two cases

\begin{itemize}
\item If $D \equiv  9 \mod 16$ then for $e=1$, one  has $w = (D-1^2)/8$
  is odd. Thus $\gcd(w,q)=1$ for  $q=2$ and $B_{2}(1) = -1-4\times 2 =
  -9$.    This   connects    $\mathcal    T^{0}_{D}$   to    $\mathcal
  T^{3}_{D}$. \medskip

\item If $D \equiv  1 \mod 16$ then for $e=3$ one  has $w = (D-3^2)/8$
  is odd. Thus $\gcd(w,q)=1$ for  $q=2$ and $B_{2}(3) = -3-4\times 2 =
  -11$. This connects $\mathcal T^{1}_{D}$ to $\mathcal T^{2}_{D}$.
\end{itemize}
This finishes the proof of Theorem~\ref{theo:main:SD} in the
case $D\not \equiv 4 \mod 105$.

\end{proof}

\subsubsection{Case $D\equiv 4 \mod 105$}

\begin{proof}
Recall  that in  this  case we  have  defined $\mathcal{U}_D:=\{e  \in
\mathcal{T}_D,  \; e  \not\equiv -2  \mod 105\}$.  We define  the sets
$\mathcal{T}^i_D$ in the same way as the previous case, namely

$$
\begin{array}{lll}
\mathcal T^{i}_{D} = \{ e \in  \mathcal T_{D},\ e \equiv 4i \mod 8\} &
i=0,\ 1, & \textrm{ if } D \equiv 0 \mod 8, \\ \mathcal T^{i}_{D} = \{
e \in \mathcal T_{D},\ e \equiv  6+4i \mod 8\} &i=0,\ 1, & \textrm{ if
}  D \equiv  4  \mod 8,  \\ \mathcal  T^{i}_{D}  = \{  e \in  \mathcal
T_{D},\ e \equiv 1 + 2i \mod 8\}  &i= 0,\ 1,\ 2,\ 3, & \textrm{ if } D
\equiv 1 \mod 8.
\end{array}
$$
and  consider the partition of  $\mathcal{U}_D$ by $\mathcal
U^{i}_{D} = \mathcal U_{D} \cap \mathcal T^{i}_{D}$.

\begin{Lemma}\label{lm:connect:UDi}
All  elements  of $\mathcal  U^{i}_{D}$  are  equivalent in  $\mathcal
S_{D}$.
\end{Lemma}

\begin{proof}[Proof of the lemma]
We  will apply Proposition~\ref{prop:equiv:step8}.  Since $D  \equiv 4
\mod  105$ and $e  \not \equiv  -2 \mod  105$, if  we can  not conclude
directly that  $e \sim e+8$  then this means  that $e \equiv  -10 \mod
105$. But in this case, since
$$ e^{2} \equiv 0 \not \equiv D \equiv 1 \mod 5
$$  one can  apply the  move $F_{q}$  with $q=5$.  This gives  $e \sim
F_{5}(e) = e + 16$. This proves the lemma.
\end{proof}

By Lemma~\ref{lm:eq:UD}  and Lemma~\ref{lm:connect:UDi}, we  only need
to   connect    elements   in   $\mathcal{U}^i_D$,    with   different
$i$.  Actually,  we  can  use  the  same strategies  as  the  case  $D
\not\equiv 4  \mod 105$  since they do  not involve any  element $e\in
\mathcal{T}^i_D$ such that $e \equiv  -2 \mod 105$. This completes the
proof of Theorem~\ref{theo:main:SD}.
\end{proof}


\section{Components of the Prym eigenforms locus}
\label{sec:diagrams}

In this section, we  give the proof of our main result (Theorem~\ref{MainTh1}) announced
in  Section~\ref{sec:background}.   Since  the  fact   that  the  Prym
eigenform loci  of different  discriminants are disjoint  follows from
Theorem~\ref{UqeThm} (see Corollary~\ref{cor:disjoint}), it remains to
show that  $\Omega E_D(4)$ has one  component when $D  \equiv 0,4 \mod
8$, and two components when $D \equiv 1 \mod 8$. \medskip

\noindent By          Theorem~\ref{theo:onto:map}         and
Theorem~\ref{theo:main:PD}, when $D \not\in \{41,68,100\}$, we have
$$    \#\   \{\textrm{Components   of    }   \Omega    E_D(4)\}   \leq
\#\      \left(\mathcal     Q_{D}/\sim \right)      \leq     2\cdot
\#\ \left(\mathcal P_{D}/\sim \right) = 2.
$$
When $D$ is  odd, by  Theorem~\ref{theo:disconnect:odd}, we
know  that  $\Omega  E_D(4)$  has  at  least  two  $\GL^+(2,\R)$-orbits,
therefore Theorem~\ref{MainTh1} is proven  for $D\equiv 1 \mod 8$, and
$D\neq 41$.

\begin{Remark}
There exists a simple congruence  relation that explains why it is not
possible to  connect $(p,+)$ to $(p,-)$  by Butterfly moves  $B_q, \ q
\in \N\cup\{\infty\}$ when $D$ is odd. Indeed, if it is the case, then
we would  have a sequence  of Butterfly moves in  $\Pcal_D$ connecting
$p$ to itself by an odd  number of steps. But this is impossible since
$e  \equiv \pm 1  \mod 4$  (since $D=e^2+8wh$),  and a  Butterfly move
sends $e$ to $e'=-e-4qh \equiv -e \not \equiv e \mod 4$.
\end{Remark}

For the remaining cases, Theorem~\ref{MainTh1} follows from

\begin{Theorem}[Generic even discriminants]
\label{theo:connect:odd:steps}
Let $D>16$ be  an even discriminant with $D\equiv 0,4  \mod 8$.  If $D
\not\in \{48,68,100\}$ then $\Qcal_D$ has only one component.
\end{Theorem}

and

\begin{Theorem}[Exceptional discriminants]
\label{theo:connect:4:cases}
\
\begin{enumerate}
\item $\Omega  E_{48}(4)$, $\Omega E_{68}(4)$  and $\Omega E_{100}(4)$
  consist of a single $\GL^+(2,\R)$-orbit;
\item $\Omega E_{41}(4)$ consists of two $\GL^+(2,\R)$-orbits.
\end{enumerate}
\end{Theorem}

\subsection{Proof of Theorem~\ref{theo:connect:odd:steps}}

We  will show  that  there exists  $e\in  \mathcal S_D$  which can  be
connected  to   itself  by  a   sequence  of  $1$  or   $3$  Butterfly
moves. Consider four different cases.

\begin{itemize}
\item[{\bf (1)}] $D  \equiv 4 \mod 8$ and  $D\not\in \{68,100\}$. Then
  $-2 \in \mathcal  S_D$ and $B_1(-2) = -2$.  Since $\Pcal_D$ has only
  one component, so is $\Qcal_D$, and we are done. \medskip

\item[{\bf  (2)}] $D =  8 +  16k, \  k\geq 1$.  Then $-4  \in \mathcal
  S_D$.  Note  that  $[-4]=(2k-1,1,0,-4)$.   Since  $e+4\cdot  2  =  4
  <\sqrt{D}$, $q=2$ is admissible, and $B_2(-4) = -4$. \medskip

\item[{\bf  (3)}] $D  =  32k$. Then  $-4  \in \mathcal  S_D$, we  have
  $[-4]=(4k-2,1,0,-4)$.  Since $e+4\cdot  2 =  4 <\sqrt{D}$,  $q=2$ is
  admissible, and
$$                        (4k-2,1,0,-4)\overset{B_{2}}{\longrightarrow}
  (2k-1,2,0,-4)\overset{B_{\infty}}{\longrightarrow}(2k-1,2,0,-4)
  \overset{B_{1}}{\longrightarrow}(4k-2,1,0,-4)
$$  is  a a  sequence  of three  Butterfly  moves  connecting $-4$  to
  itself.  \medskip

\item[{\bf (4)}] $D = 16+32k$ and $k>1$. Since $k\geq 2$, $-8
  \in  \mathcal S_D$ and  $[-8]=(4k-6,1,0,-8)$. This  time we  use the
  sequence
$$            (4k-6,1,0,-8)\overset{B_{2}}{\longrightarrow}(2k+1,2,0,0)
  \overset{B_{\infty}}{\longrightarrow}(2k-3,2,0,-8)
  \overset{B_{2}}{\longrightarrow}(4k-6,1,0,-8)
$$ to connect  $-8$ to itself with three steps.  Observe that $q=2$ is
  admissible in both cases.

\hfill $\square$
\end{itemize}

\subsection{Proof of Theorem~\ref{theo:connect:4:cases}}
\label{sec:four:discri}


\subsubsection{$D=100$}

Since  $D=100  =  10^{2}$  the  surfaces in  $\Omega  E_{100}(4)$  are
arithmetic surfaces (square-tiled  surfaces). The set $\mathcal Q_{D}$
has  exactly two  components, represented  by the  complete prototypes
$(12,1,0,-2,+)$        and       $(12,1,0,2,+)$        (see       also
Figure~\ref{fig:action:P68}   page~\pageref{fig:action:P68}   for  the
action of  Butterfly Moves  on $\mathcal P_{100}$).  Let $\Sigma_{-2}$
and  $\Sigma_{2}$  be  the  surface constructed  from  the  prototypes
$(12,1,0,-2,+)$   and  $(12,1,0,2,+)$,  respectively.    Observe  that
normalizing by $\GL^+(2,\Q)$, $\Sigma_{-2}$ and $\Sigma_{2}$ are
square-tiled surfaces, made of $10$ squares. \medskip

It turns  out there are exactly  $135$ square-tiled  surfaces (made of
$10$ squares)  in $\Omega\mathfrak{M}(4)$ and  they all belong  to the
same  Teichm\"uller  curve.  To  be  more precise,  if  we  denote  by
$L=\left( \begin{smallmatrix} 1& 1 \\ 0 & 1 \end{smallmatrix} \right)$
and  $R=\left( \begin{smallmatrix}  1& 0  \\ 1  &  1 \end{smallmatrix}
\right)$ the standard generators of $\textrm{SL}(2,\Z)$, then
$$ R^{2}\cdot (R \cdot L)^{3} \cdot \Sigma_{-2} = \Sigma_{2}.
$$ This shows that $\Omega E_{100}(4)$ is connected.

\subsubsection{$D=48$}
\label{soussec:excep1}

In this case
$$ \Qcal_{48}=\{(2,2,1,-4,\pm), (4,1,0,-4,\pm),(6,1,0,0,\pm)\}.
$$ The Butterfly moves connect  all the incomplete prototypes, that is
$\Pcal_{48}$  has only one  component (see  Figure~\ref{fig:P48}), but
$\Qcal_{48}$   has  two   components  since   none  of   prototypes  in
$\Pcal_{48}$ can be connected to  itself by an odd number of Butterfly
moves.
 \begin{figure}[htbp]
 \begin{tikzpicture}[scale=0.3, inner sep=1mm]
 \node[rectangle, draw, thick] (a) at (-8,0) {$\scriptstyle (2,2,1,-4)$};
 \node[rectangle, draw, thick] (b) at (0,0) {$\scriptstyle (4,1,0,-4)$};
 \node[rectangle, draw, thick] (c) at (8,0) {$\scriptstyle (6,1,0,0)$};
 
 \draw[->,  >=stealth]   (a)  to  node[below]   {$\scriptstyle B_1,B_\infty$}  (b);
 \draw[->, >=stealth] (b)  to [out=135, in=45] node[above] {$\scriptstyle B_2$}(a);
 \draw[->,  >=stealth]   (b)  to  node[below]   {$\scriptstyle B_1,B_\infty$}  (c);
 \draw[->,    >=stealth]    (c)    to   [out=135,in=45]    node[above]
      {$\scriptstyle B_1,B_\infty$} (b);
 \end{tikzpicture}
 \caption{
 \label{fig:P48}
Action of the Butterfly moves on $\Pcal_{48}$.}
\end{figure}
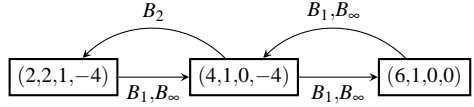

We label the components of $\Qcal_{48}$ as follows:
$$
\begin{array}{lll}
\Qcal^1_{48}  &  =   &  \{(2,2,1,-4,+),  (4,1,0,-4,-),  (6,1,0,0,+)\}, \\
\Qcal^2_{48} & = & \{(2,2,1,-4,-), (4,1,0,-4,+), (6,1,0,0,-)\} \\
\end{array}
$$ We will  show that there is actually  only one $\GL^+(2,\R)$-orbit in
$\Omega E_{48}(4)$.   To see this we  pick a prototype  for Model $B$,
that  is, a  quadruplet  $(w,h,t,e)$ of  integers satisfying  Property
$(\Pcal')$, and show that  the surface constructed from this prototype
admits  two   decompositions,  one  corresponds  to   a  prototype  in
$\Qcal^1_{48}$,  and   the  other   corresponds  to  a   prototype  in
$\Qcal^2_{48}$.   Note that,  for $D=48$  there are  $4$  solutions to
$(\Pcal')$, listed below
$$ \mathcal{S}'_{48}=\{(3,2,0,0), (4,1,0,4), (2,3,0,0), (1,4,0,-4)\}.
$$  Let  $\Sigma$  be  the  surface  constructed  from  the  prototype
$(3,2,0,0)$          following         Model          $B$         (see
Figure~\ref{fig:48:new:direction}).               We              have
$\displaystyle{\lambda=\frac{e+\sqrt{D}}{2}=\frac{\sqrt{48}}{2}=2\sqrt{3}}$.
The surface $\Sigma$ admits decompositions following Model $A-$ in the
directions              $v_1=(\lambda/2,h+\lambda/2)$,             and
$v_2=(w,-h-\lambda/2)$.    Direct   computations    show    that   the
decomposition  in the  direction  $v_1$ corresponds  to the  prototype
$(6,1,0,0,-) \in  \Qcal^2_{48}$, while the  decomposition in direction
$v_2$    corresponds    to    the    prototype    $(4,1,0,-4,-)    \in
\Qcal^1_{48}$.   Remark  that   in   this  case,   to  determine   the
corresponding  prototypes, it  suffices to  compute the  ratio  of the
heights of the cylinders in directions $v_1,v_2$.

\begin{figure}[htbp]
\begin{minipage}[t]{0.4\linewidth}
\centering
\begin{tikzpicture}[scale=0.6]

\filldraw[fill=gray!10]    (0,0)   --   (1.75,3.75)--    (1.75,2)   --
(intersection  of 1.75,2  --  0,-1.75  and 0,0  --  1.25,0) --  cycle;
\filldraw[fill=gray!10] (0,2) -- (intersection of 0,2 -- 1.75,5.75 and
0,3.75  -- 1.25,3.75)  -- (0,3.75)  --  cycle; \filldraw[fill=gray!10]
(1.25,0) --  (intersection of 1.25,0 --  3,3.75 and 1.75,2  -- 3,2) --
(3,2)  --  (1.25,-1.75)  --  cycle; \filldraw[fill=gray!10]  (3,0)  --
(intersection  of 3,0  --  1.25,-3.75 and  1.75,-1.75  -- 3,-1.75)  --
(3,-1.75) -- cycle;

\draw  (0,0) --  (1.25,0) --  (1.25,-1.75)  -- (3,-1.75)  -- (3,2)  --
(1.75,2) -- (1.75,3.75) -- (0,3.75)  -- cycle; \draw (1.25,0) -- (3,0)
(0,2) -- (1.75,2);

\filldraw[draw=black, fill=white]  (0,0) circle (2pt)  (1.25,0) circle
(2pt)  (1.25,-1.75) circle (2pt)  (1.75,-1.75) circle  (2pt) (3,-1.75)
circle  (2pt) (3,2) circle  (2pt) (1.75,2)  circle (2pt)  (1.75, 3.75)
circle  (2pt) (1.25,3.75)  circle  (2pt) (0,3.75)  circle (2pt)  (0,2)
circle (2pt) (3,0) circle (2pt) ;

\draw[thin,  <->,  >=angle  45]   (0,-2)  --  (3,-2);  \draw  (1.5,-2)
node[below] {$\scriptstyle w$};  \draw[thin, <->, >=angle 45] (-0.2,0)
-- (-0.2,2)   ;   \draw   (-0.2,1)  node[left]   {$\scriptstyle   h$};
\draw[thin, <->, >=angle 45]  (-0.2,2) -- (-0.2,3.75) ; \draw (-0.2,3)
node[left]  {$\scriptstyle \lambda/2$};  \draw[thin, <->,  >=angle 45]
(0,3.9)  --  (1.75,3.9);  \draw (0.9,3.9)  node[above]  {$\scriptstyle
  \lambda/2$};

\draw (1.5,-3) node {simple cylinders in direction $v_1$};

\end{tikzpicture}
\end{minipage} 
\begin{minipage}[t]{0.4\linewidth}
\centering
\begin{tikzpicture}[scale=0.6]
\filldraw[fill=gray!10] (0,3.75)  -- (3,0) --  (intersection of 1.75,2
-- 4.75,-1.75 and 3,0 -- 3,2) -- (intersection of 1.75,2 -- -1.25,5.75
and 0,3.75  -- 1.75,3.75)  -- cycle; \filldraw[fill=gray!10]  (0,2) --
(3,-1.75) --  (intersection of 1.25,0 -- 4.25,-3.75  and 1.25,-1.75 --
3,-1.75) -- (intersection  of 1.25,0 -- -1.75,3.75 and  0,0 -- 0,2) --
cycle;  \filldraw[fill=gray!10] (0,0) --  (intersection of  -1.25,2 --
1.75,-1.75  and  0,0  --   1.25,0)  --  (intersection  of  -1.25,2  --
1.75,-1.75 and 0,0 --  0,2) -- cycle; \filldraw[fill=gray!10] (3,2) --
(intersection of 1.25,3.75 -- 4.15,0  and 3,0 -- 3,2) -- (intersection
of 1.25,3.75 -- 4.15,0 and 3,2 -- 1.75,2) -- cycle;

\draw  (0,0) --  (1.25,0) --  (1.25,-1.75)  -- (3,-1.75)  -- (3,2)  --
(1.75,2) -- (1.75,3.75) -- (0,3.75)  -- cycle; \draw (1.25,0) -- (3,0)
(0,2) -- (1.75,2);

\filldraw[draw=black, fill=white]  (0,0) circle (2pt)  (1.25,0) circle
(2pt)  (1.25,-1.75) circle (2pt)  (1.75,-1.75) circle  (2pt) (3,-1.75)
circle  (2pt) (3,2) circle  (2pt) (1.75,2)  circle (2pt)  (1.75, 3.75)
circle  (2pt) (1.25,3.75)  circle  (2pt) (0,3.75)  circle (2pt)  (0,2)
circle (2pt) (3,0) circle (2pt);

\draw (1.5,-3) node {simple cylinders in direction $v_2$};

\end{tikzpicture}
\end{minipage}

\caption{
\label{fig:48:new:direction}
Two periodic directions  corresponding to two prototypes in $\Qcal^1_{48}$ and $\Qcal^2_{48}$ 
on the surface constructed from  the prototype
$(3,2,0,0)$.  }
\end{figure}
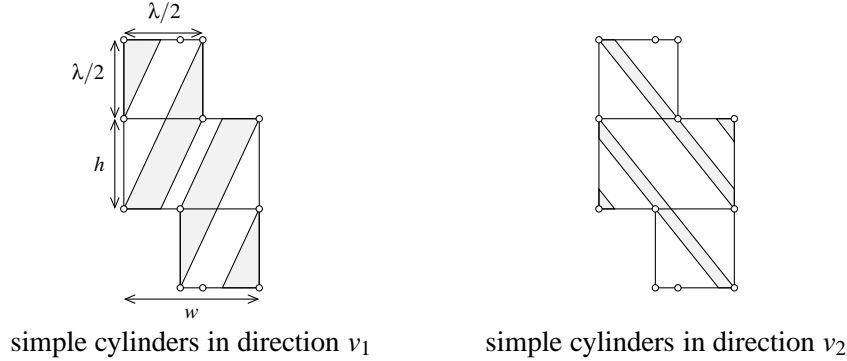

\subsubsection{$D=68$} $\Qcal_{68}$ has two components

$$\begin{array}{ccl}    

\Qcal^1_{68}   &    =    &   \{(2,1,2,-6,\pm), (8,0,1,-2,\pm\},\\
\Qcal^2_{68}    &     =    &\{(4,0,1,-6,\pm), (8,0,1,2,\pm), (4,1,2,-2,\pm)\}
\end{array}
$$
The strategy  is the same: we connect  the two components of
$\Qcal_{68}$  using two directions  on a  surface obtained  with Model
$B$. We have

$$\mathcal{S}'_{68}=\{(4,1,2,2),         (1,0,4,-6),        (4,0,1,6),
(2,1,4,-2)\}.
$$
Let  $\Sigma$ be the surface constructed  from the prototype
$(4,2,1,2)$    of     Model    $B$.     We     have    $\displaystyle{
  \lambda=1+\sqrt{17}}$.  This   surface  admits  decompositions  into
cylinders      in      directions      $v_1=(t,h+\lambda/2)$,      and
$v_2=(t+\lambda/2,h+\lambda/2)$.   Direct computations  show  that the
prototype  corresponding to  the decomposition  in direction  $v_1$ is
$(8,1,0,-2,-)\in \Qcal^1_{68}$, and the prototype corresponding to the
decomposition in direction $v_2$ is $(8,1,0,2,-)\in \Qcal^2_{68}$ (see
Figure~\ref{fig:D68:new:direction} for details).

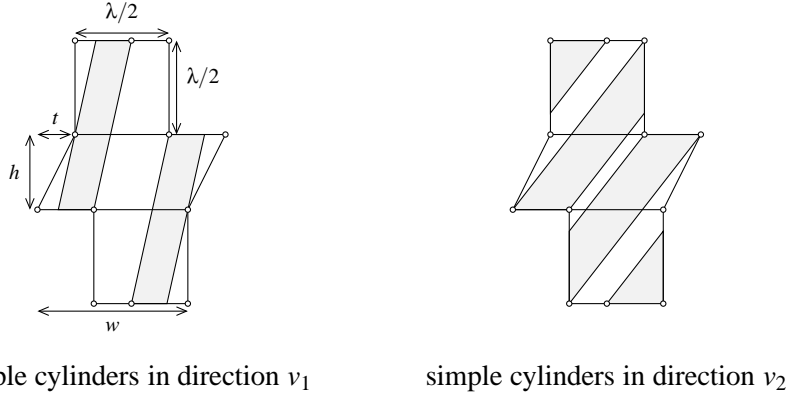
\begin{figure}[htbp]
\begin{minipage}[t]{0.4\linewidth}
\centering
\begin{tikzpicture}[scale=0.5]
\filldraw[fill=gray!10] (1.5,0)  -- (2.5,4.5) --  (intersection of 1,2
-- 2,6.5 and 1,4.5  -- 2.5,4.5) -- (intersection of  1,2 -- 0,-2.5 and
0,0 -- 1.5,0) -- cycle;

\filldraw[fill=gray!10] (3.5,2) --  (2.5,-2.5) -- (intersection of 4,0
-- 3,-4.5 and 2.5,-2.5 -- 4,-2.5) -- (intersection of 4,0 -- 5,4.5 and
3.5,2 -- 5,2) -- cycle;

\draw (0,0) -- (1.5,0) -- (1.5,-2.5)  -- (4,-2.5) -- (4,0) -- (5,2) --
(3.5,2)  -- (3.5,4.5) --  (1,4.5) --  (1,2) --  cycle; \draw  (1,2) --
(3.5,2) (1.5,0) -- (4,0);

\filldraw[draw=black,  fill=white] (0,0)  circle (2pt)  (1.5,0) circle
(2pt) (1.5,-2.5) circle (2pt)  (2.5,-2.5) circle (2pt) (4,-2.5) circle
(2pt) (4,0) circle (2pt) (5,2) circle (2pt) (3.5,2) circle (2pt) (3.5,
4.5) circle  (2pt) (2.5,4.5) circle  (2pt) (1,4.5) circle  (2pt) (1,2)
circle (2pt);

\draw[<->, >=angle  45] (3.7,2)  -- (3.7,4.5); \draw[<->,  >=angle 45]
(1,4.7)  --  (3.5, 4.7);  \draw[<->,  >=angle  45]  (0,2) --  (0.9,2);
\draw[<->,>=angle  45] (-0.2,0)  -- (-0.2,2);  \draw[<->,  >=angle 45]
(0,-2.7) -- (4,-2.7);

\draw  (3.7,3.5)  node[right]  {$\scriptstyle  \lambda/2$}  (2.25,4.7)
node[above]    {$\scriptstyle    \lambda/2$}    (0.5,2)    node[above]
{$\scriptstyle  t$} (-0.2,1)  node[left]  {$\scriptstyle h$}  (2,-2.7)
node[below] {$\scriptstyle w$};

\draw (2.5,-4.5) node {simple cylinders in direction $v_1$};
\end{tikzpicture}
\end{minipage}
\begin{minipage}[t]{0.4\linewidth}
\centering
\begin{tikzpicture}[scale=0.5]

\filldraw[fill=gray!10] (0,0)  -- (3.5,4.5) --  (intersection of 1.5,0
-- 5,4.5   and    3.5,4.5   --    3.5,2)   --   (1.5,0)    --   cycle;
\filldraw[fill=gray!10] (2.5,4.5) --  (intersection of 2.5,4.5 -- -1,0
and  1,4.5  --  1,2)  --  (1,4.5)  --  cycle;  \filldraw[fill=gray!10]
(1.5,-2.5) -- (5,2) -- (3.5,2) -- (intersection of 3.5,2 -- 0,-2.5 and
1.5,0  -- 1.5,-2.5)  -- cycle;  \filldraw[fill=gray!10]  (2.5,-2.5) --
(intersection of 2.5,-2.5 -- 6,2 and 4,-2.5 -- 4,0) -- (4,-2.5);

\draw (0,0) -- (1.5,0) -- (1.5,-2.5)  -- (4,-2.5) -- (4,0) -- (5,2) --
(3.5,2)  -- (3.5,4.5) --  (1,4.5) --  (1,2) --  cycle; \draw  (1,2) --
(3.5,2) (1.5,0) -- (4,0);

\filldraw[draw=black,  fill=white] (0,0)  circle (2pt)  (1.5,0) circle
(2pt) (1.5,-2.5) circle (2pt)  (2.5,-2.5) circle (2pt) (4,-2.5) circle
(2pt) (4,0) circle (2pt) (5,2) circle (2pt) (3.5,2) circle (2pt) (3.5,
4.5) circle  (2pt) (2.5,4.5) circle  (2pt) (1,4.5) circle  (2pt) (1,2)
circle (2pt);

\draw (2.5,-4.5) node { simple cylinders in direction $v_2$};
\end{tikzpicture}
\end{minipage}
\caption{
\label{fig:D68:new:direction}
Surface constructed from the prototype $(4,2,1,2)$: two periodic directions corresponding to
 prototypes in $\Qcal^1_{68}$ and $\Qcal^2_{68}$. }
\end{figure}

\subsubsection{$D=41$}

In this  case, the Butterfly Moves  do not connect  all the incomplete
prototypes  in   $\Pcal_{41}$,  therefore  $\Qcal_{41}(4)$   has  four
components:
$$
\begin{array}{lll}
\Qcal_{41}^1  &   =  &  \{(2,2,0,-3,+),   (1,2,0,-5,-),  (4,1,0,-3,+),
(5,1,0,-1,-)\},\\  \Qcal_{41}^2  &  = &  \{(2,1,0,-5,-),  (5,1,0,1,+),
(2,2,1,-3,+)\},\\  \Qcal_{41}^3 & =  & \{  (2,2,0,-3,-), (1,2,0,-5,+),
(4,1,0,-3,-),  (5,1,0,-1,+)\},\\ \Qcal_{41}^4  &  = &  \{(2,1,0,-5,+),
(5,1,0,1,-), (2,2,1,-3,-)\}.
\end{array}
$$
%

\begin{figure}[htbp]
\begin{minipage}[t]{0.4\linewidth}
\centering
\begin{tikzpicture}[scale=0.3]
\filldraw[fill=gray!10]  (0,6.5) -- (intersection  of 0,6.5  -- 9,-4.5
and 8,-4.5 --  8,0) -- (8,-4.5) -- ( intersection  of 8,-4.5 -- -1,6.5
and 0,2 -- 0,6) -- cycle;

\filldraw[fill=gray!10]  (3.5,6.5)  --  (intersection  of  3.5,6.5  --
12.5,-4.5 and 4.5,2 -- 4.5,6.5) -- (4.5,6.5) -- cycle;

\filldraw[fill=gray!10]  (4.5,-4.5) --  (intersection  of 4.5,-4.5  --
-4.5,6.5 and 3.5,0 -- 3.5,-4.5) -- (3.5,-4.5) -- cycle;

\draw (0,0) --  (3.5,0) -- (3.5,-4.5) -- (8,-4.5)  -- (8,2) -- (4.5,2)
-- (4.5,6.5) --  (0,6.5) -- cycle;  \draw (0,2) -- (4.5,2)  (3.5,0) --
(8,0);

\filldraw[draw=black,  fill=white] (0,0)  circle (3pt)  (3.5,0) circle
(3pt) (3.5,-4.5) circle (3pt)  (4.5,-4.5) circle (3pt) (8,-4.5) circle
(3pt) (8,0) circle (3pt) (8,2) circle (3pt) (4.5,2) circle (3pt) (4.5,
6.5) circle  (3pt) (3.5,6.5) circle  (3pt) (0,6.5) circle  (3pt) (0,2)
circle (3pt);

\draw (4,-5) node[below] {simple cylinder in direction $v_1$};
\end{tikzpicture}
\end{minipage} 
\begin{minipage}[t]{0.4\linewidth}
\centering
\begin{tikzpicture}[scale=0.3]

\filldraw[fill=gray!10]  (0,6.5) -- (intersection  of 0,6.5  -- 12.5,0
and 4.5,2 --  4.5,6.5) -- ( intersection of -3.5,8.5  -- 9,2 and 4.5,2
-- 4.5,6.5) -- ( intersection of -3.5,8.5 -- 9,2 and 0,6.5 -- 3.5,6.5)
-- cycle;

\filldraw[fill=gray!10] (8,0) -- (intersection  of 8,0 -- -4.5,6.5 and
0,2 -- 0,6.5) -- (intersection of 0,2 -- 0,6.5 and 4.5,2 -- -8,8.5) --
(intersection of 4.5,2 -- 17,-4.5 and 8,0 -- 8,2) -- cycle;

\filldraw[fill=gray!10] (0,0) -- (intersection of -3.5,2 -- 9,-4.5 and
0,0 -- 0,2) -- (intersection of  -3.5,2 -- 9,-4.5 and 0,0 -- 3.5,0) --
cycle;

\filldraw[fill=gray!10] (8,-4.5) --  (intersection of 8,-4.5 -- -4.5,2
and 3.5,0 -- 3.5,-4.5) -- (intersection of 11.5,-6.5 -- -1,0 and 3.5,0
-- 3.5,  -4.5) --  (intersection of  11.5,-6.5 --  -1,0 and  8,-4.5 --
4.5,-4.5) -- cycle;

\filldraw[fill=gray!10] (0,2) -- (intersection of 0,2 -- 12.5,-4.5 and
8,0 -- 8,-4.5) -- (intersection of 8,0 -- 8,-4.5 and 3.5,0 -- 16,-6.5)
-- (intersection of 3.5,0 -- -9,6.5 and 0,2 -- 0,0) -- cycle;

\filldraw[fill=gray!10] (8,2) -- (intersection of 11.5,0 -- -1,6.5 and
8,2 -- 8,0) -- (intersection of  11.5,0 -- -1,6.5 and 8,2 -- 4.5,2) --
cycle;

\draw (0,0) --  (3.5,0) -- (3.5,-4.5) -- (8,-4.5)  -- (8,2) -- (4.5,2)
-- (4.5,6.5) --  (0,6.5) -- cycle;  \draw (0,2) -- (4.5,2)  (3.5,0) --
(8,0);

\filldraw[draw=black,  fill=white] (0,0)  circle (3pt)  (3.5,0) circle
(3pt) (3.5,-4.5) circle (3pt)  (4.5,-4.5) circle (3pt) (8,-4.5) circle
(3pt) (8,0) circle (3pt) (8,2) circle (3pt) (4.5,2) circle (3pt) (4.5,
6.5) circle  (3pt) (3.5,6.5) circle  (3pt) (0,6.5) circle  (3pt) (0,2)
circle (3pt);

\draw (4,-5) node[below] {simple cylinders in direction $v_2$};
\end{tikzpicture}
\end{minipage}
\caption{
\label{fig:D41:1:new:direction}
Surface constructed from  the prototype $(4,1,0,3)$: two periodic directions  corresponding to prototypes 
in $\Qcal^1_{41}$ and $\Qcal^2_{41}$. }
\end{figure}
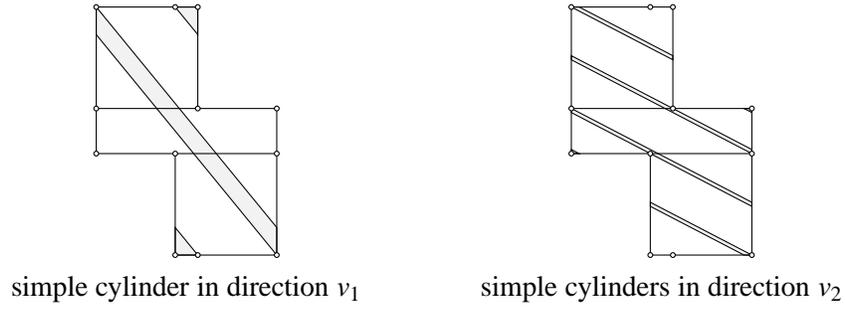

We  have  shown  that  when  $D$  is  odd,  there  are  at  least  two
$\GL^+(2,\R)$-orbits in  $\Omega E_D(4)$.  We  will show that  there are
exactly two  $\GL^+(2,\R)$-orbits in $\Omega  E_{41}(4)$. Let $\Sigma_1$
be  the surface constructed  from the  prototype $(4,1,0,3)$  of Model
$B$.  We have  $\displaystyle{ \lambda=\frac{3+\sqrt{41}}{2}}$.  This
surface   admits   decompositions   into   cylinders   in   directions
$\displaystyle{v_1=(\lambda,             -\lambda-h)}$,            and
$\displaystyle{v_2=(w+\lambda/2,-\lambda/2-h)}$.  The decomposition in
direction  $v_1$  corresponds   to  the  prototype  $(4,1,0,-3,+)  \in
\Qcal_{41}^1$, and the decomposition in direction $v_2$ corresponds to
the   prototype  $(2,1,0,-5,-)   \in  \Qcal_{41}^2$.   Therefore,  the
prototypes in $\Qcal^1_{41}$ and  $\Qcal_{41}^2$ give rise to the same
$\GL^+(2,\R)$-orbit    (see   Figure~\ref{fig:D41:1:new:direction}   for
details).

\begin{figure}[htbp]
\begin{minipage}{0.4\linewidth}
\centering
\begin{tikzpicture}[scale=0.7]
\filldraw[fill=gray!10] (0,2.8) --  (1,-2) -- (intersection of 0.2,2.8
-- 1.2,-2 and 1,-2 -- 1,2) -- (0.2,2.8) -- cycle;

\filldraw[fill=gray!10]  (0.2,-2)   --  (intersection  of   0.2,-2  --
-0.8,2.8 and 0,-2 -- 0,0) -- (0,-2) -- cycle;

\filldraw[fill=gray!10] (1,-2.8) -- (0,2) -- (intersection of 0.8,-2.8
-- -0.2,2 and 0,2 -- 0,-2) -- (0.8,-2.8) -- cycle ;

\filldraw[fill=gray!10] (0.8,2) --  (intersection of 0.8,2 -- 1.8,-2.8
and 1,2 -- 1,0) -- (1,2) -- cycle;

\draw (0,-2) -- (0.2,-2) -- (0.2,-2.8) -- (1,-2.8) -- (1,2) -- (0.8,2)
-- (0.8,2.8) --  (0,2.8) -- cycle;  \draw (0.2,-2) -- (1,-2)  (0,2) --
(0.8,2);

\filldraw[draw=black, fill=white] (0,-2)  circle (1pt) (0.2,-2) circle
(1pt) (0.2,-2.8) circle (1pt)  (0.8,-2.8) circle (1pt) (1,-2.8) circle
(1pt)  (1,-2) circle  (1pt) (1,2)  circle (1pt)  (0.8,2)  circle (1pt)
(0.8, 2.8)  circle (1pt) (0.2,2.8)  circle (1pt) (0,2.8)  circle (1pt)
(0,2) circle (1pt);

\draw (0.5,-3) node[below] {simple cylinders in direction $v_1$};

\end{tikzpicture}
\end{minipage} 
\begin{minipage}{0.4\linewidth}
\centering
\begin{tikzpicture}[scale=0.7]
\filldraw[fill=gray!10] (0,-2)  -- (intersection of 0,-2  -- 3,3.6 and
1,-2 --  1,2) -- (intersection of 2,2  -- -1,-3.6 and 1,-2  -- 1,2) --
(intersection of 2,2 -- -1,-3.6 and 0,-2 -- 0,2) -- cycle;

\filldraw[fill=gray!10]  (0,2) --  (intersection of  0,2 --  3,7.6 and
0,2.8 --  0.8, 2.8) --  (intersection of -2,-2  -- 1,3.6 and  0,2.8 --
0.8,2.8) -- (intersection of -2,-2 -- 1,3.6 and 0,0 -- 0,2) -- cycle;

\filldraw[fill=gray!10] (1,2)  -- (intersection of 1,2  -- -2,-3.6 and
0,2 --  0,-2) -- (intersection of -1,-2  -- 2,3.6 and 0,2  -- 0,-2) --
(intersection of -1,-2 -- 2,3.6 and 1,2 -- 1,-2) -- cycle;

\filldraw[fill=gray!10] (1,-2) -- (intersection of 1,-2 -- -2,-7.6 and
1,-2.8 --  0.2,-2.8) -- (intersection of  3,2 -- 0,-3.6  and 1,-2.8 --
0.2,-2.8) -- (intersection of 3,2 -- 0,-3.6 and 1,0 -- 1,-2) -- cycle;

%


\draw (0,-2) -- (0.2,-2) -- (0.2,-2.8) -- (1,-2.8) -- (1,2) -- (0.8,2)
-- (0.8,2.8) --  (0,2.8) -- cycle;  \draw (0.2,-2) -- (1,-2)  (0,2) --
(0.8,2);

\filldraw[draw=black, fill=white] (0,-2)  circle (1pt) (0.2,-2) circle
(1pt) (0.2,-2.8) circle (1pt)  (0.8,-2.8) circle (1pt) (1,-2.8) circle
(1pt)  (1,-2) circle  (1pt) (1,2)  circle (1pt)  (0.8,2)  circle (1pt)
(0.8, 2.8)  circle (1pt) (0.2,2.8)  circle (1pt) (0,2.8)  circle (1pt)
(0,2) circle (1pt);

\draw (0.5,-3) node[below] {simple cylinder in direction $v_2$};
\end{tikzpicture}

\end{minipage}
\caption{
\label{fig:D41:2:new:direction}
Surface constructed from the prototype $(1,4,0,-3)$: two periodic directions corresponding to prototypes in
 $\Qcal^3_{41}$ and $\Qcal^4_{41}$. }
\end{figure}
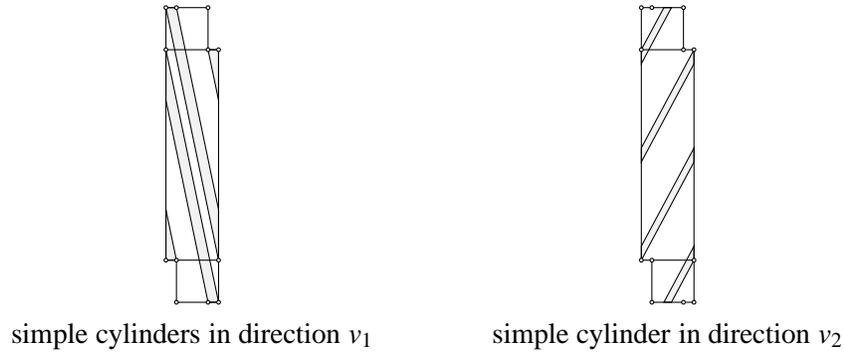

Let  $\Sigma_2$   be  the  surface  constructed   from  the  prototype
$(1,4,0,-3)$        of         Model        $B$.        We        have
$\displaystyle{\lambda=\frac{-3+\sqrt{41}}{2}}$.  This  surface admits
decompositions        into        cylinders       in        directions
$\displaystyle{v_1=(w,-h-\lambda/2)}$,                              and
$\displaystyle{v_2=(3w,h+\lambda)}$.  The  decomposition in  direction
$v_1$  corresponds to the  prototype $(4,1,0,-3,-)\in  \Qcal_{41}^3 $,
and the  decomposition indirection $v_2$ corresponds  to the prototype
$(2,1,0,-5,+)   \in  \Qcal_{41}^4$.   Therefore   the  prototypes   in
$\Qcal_{41}^3$   and   $\Qcal_{41}^4$    give   rise   to   the   same
$\GL^+(2,\R)$-orbit.  We  can  then  conclude  that  $\Omega  E_{41}(4)$
consists         of         two        $\GL^+(2,\R)$-orbits         (see
Figure~\ref{fig:D41:2:new:direction} for details).
The proof of Theorem~\ref{theo:connect:4:cases} is now complete. Theorem~\ref{MainTh1} is then 
proven \hfill $\square$

\newpage

\appendix

\section{Exceptional cases in Theorem~\ref{theo:main:PD}}
\label{appendix:exceptional}

The table below encodes the strategies that connect the 
different orbits in the proof of Theorem~\ref{theo:main:PD} for 
exceptional cases: $D       \in  \{73,97,112,148,196,244,292,304,436,484,676,1684\}$.
See page~\pageref{Table:explanation} for an explain of this table.

\begin{table}[htbp]
$$
\begin{array}{|c|c|c|}
\hline
D & \textrm{Components of $\mathcal S_{D}$} & \textrm{Butterfly Moves} \\
\hline 
\hline
73 & {\scriptstyle \{1,-5\} \textrm{ and } \{-1,-3,3,-7\}} & {\scriptstyle [-5] \overset{B_{3}}{\longrightarrow} (1,3,0,-7)\overset{B_{\infty}}{\longrightarrow}(2,3,0,-5)\overset{B_{1}}{\longrightarrow} [-7] }\\
\hline
97 & {\scriptstyle \{-7,3\} \textrm{ and } \{-9,-5,-3,-1,1,5\}} & {\scriptstyle [-7] \overset{B_{4}}{\longrightarrow} (1,2,0,-9)\overset{B_{1}}{\longrightarrow} [1]} \\
\hline
112 & {\scriptstyle \{-8,4\} \textrm{ and } \{-4,0\}} & {\scriptstyle [0] \overset{B_{2}}{\longrightarrow} (3,2,0,-8)\overset{B_{2}}{\longrightarrow} [-8] }\\
\hline
148 & {\scriptstyle \{-2\},\ \{-6,2\} \textrm{ and } \{-10,6\}} & 
\begin{array}{l}
{\scriptstyle [-2] \overset{B_{2}}{\longrightarrow} (7,2,0,-6) \overset{B_{2}}{\longrightarrow} [-10]}\\
{\scriptstyle [-6] \overset{B_{4}}{\longrightarrow} (3,2,0,-10)\overset{B_{\infty}}{\longrightarrow} (9,2,0,2) \overset{B_{1}}{\longrightarrow} [-10]}
\end{array} \\
\hline
196 & {\scriptstyle \{-2\},\ \{-6,2\} \textrm{ and } \{-10,6\} }& 
\begin{array}{l}
{\scriptstyle [-2] \overset{B_{3}}{\longrightarrow} (4,3,0,-10)\overset{B_{\infty}}{\longrightarrow} (8,3,0,-2) \overset{B_{1}}{\longrightarrow} [-10]}\\
{\scriptstyle [-6] \overset{B_{4}}{\longrightarrow} (3,4,0,-10)\overset{B_{\infty}}{\longrightarrow} (5,4,0,-6) \overset{B_{1}}{\longrightarrow} [-10]}
\end{array} \\
\hline
244 & {\scriptstyle \{-2\},\ \{-14,-6,2,10\} \textrm{ and } \{-10,6\} }& 
\begin{array}{l}
{\scriptstyle [-2] \overset{B_{\infty}}{\longrightarrow} (13,2,0,-6) \overset{B_{1}}{\longrightarrow} [-14]}\\
{\scriptstyle [6] \overset{B_{3}}{\longrightarrow} (3,2,0,-14)\overset{B_{2}}{\longrightarrow} [-2]}
\end{array} \\
\hline
292 & {\scriptstyle \{-2\},\ \{-14,-6,2,10\} \textrm{ and } \{-10,6\}} & 
\begin{array}{l}
{\scriptstyle [2] \overset{B_{2}}{\longrightarrow} (12,2,1,-10) \overset{B_{2}}{\longrightarrow} (16,2,1,-6)\overset{B_{1}}{\longrightarrow} [-2]}\\
{\scriptstyle [-6] \overset{B_{2}}{\longrightarrow} (6,2,1,-14)\overset{B_{2}}{\longrightarrow} (9,4,0,-2)\overset{B_{1}}{\longrightarrow} [-14]}
\end{array} \\
\hline
304 & {\scriptstyle \{-5,4\} \textrm{ and } \{-16,-12,-4,0,8,12\} }& {\scriptstyle 4 \overset{B_{3}}{\longrightarrow} (2,3,0,-16)\overset{B_{2}}{\longrightarrow}(15,2,0,-8)\overset{B_{1}}{\longrightarrow} 0 }\\
\hline
436 & %
\begin{array}{l}
{\scriptstyle \{-18,-10,-2,6,14\},\ \{-6,2\}, \textrm{ and} }\\ 
{\scriptstyle \{-14,10\}}
\end{array} &
\begin{array}{l}
{\scriptstyle [-6] \overset{B_{4}}{\longrightarrow} (21,2,0,-10) \overset{B_{\infty}}{\longrightarrow} (27,2,0,2)\overset{B_{1}}{\longrightarrow} [-10]}\\
{\scriptstyle [-6] \overset{B_{2}}{\longrightarrow} (27,2,0,-2)\overset{B_{2}}{\longrightarrow} [-14]}
\end{array} \\
\hline
484 & %
\begin{array}{l}
{\scriptstyle \{-2\},\ \{-14,-6,2,10\} , \textrm{ and} }\\ 
{\scriptstyle \{-18,-10,6,14\}}
\end{array} &
\begin{array}{l}
{\scriptstyle [2] \overset{B_{2}}{\longrightarrow} (24,2,1,-10) \overset{B_{2}}{\longrightarrow} (28,2,1,-6)\overset{B_{1}}{\longrightarrow} [-2]}\\
{\scriptstyle [-6] \overset{B_{2}}{\longrightarrow} (30,2,1,-2)\overset{B_{2}}{\longrightarrow} (9,4,0,-14)\overset{B_{1}}{\longrightarrow} [-2]}
\end{array} \\
\hline
676 & %
\begin{array}{l}
{\scriptstyle \{-18,-10,-2,6,14\},\ \{-14,10\}, \textrm{ and} }\\ 
{\scriptstyle \{-22,-6,2,18\}}
\end{array} &
\begin{array}{l}
{\scriptstyle [2] \overset{B_{2}}{\longrightarrow} (36,2,1,-10) \overset{B_{2}}{\longrightarrow} (40,2,1,-6)\overset{B_{1}}{\longrightarrow} [-2]}\\
{\scriptstyle [-6] \overset{B_{2}}{\longrightarrow} (42,2,1,-2)\overset{B_{2}}{\longrightarrow} (15,4,0,-14)\overset{B_{1}}{\longrightarrow} [-2]}
\end{array} \\
\hline
1684 & %
\begin{array}{l}
{\scriptstyle \{-2\},\ \{-34,-26,-18,-10,6,14,22,30\},\textrm{ and} } \\
\scriptstyle {\{-38,-34,-30,-22,-14,-6,2,10,18,26,34\}}
\end{array} & 
\begin{array}{l}
{\scriptstyle [-6] \overset{B_{2}}{\longrightarrow} (105,2,0,-2)\overset{B_{\infty}}{\longrightarrow} (103,2,0,-6)\overset{B_{1}}{\longrightarrow} [-2] }\\
{\scriptstyle [-6] \overset{B_{2}}{\longrightarrow} (105,2,0,-2)\overset{B_{\infty}}{\longrightarrow} (103,2,0,-6)\overset{B_{2}}{\longrightarrow} [-10]}
\end{array} \\
\hline
\end{array}
$$
\caption{\label{table:exceptionnal:link}
Connecting  components of  $\mathcal S_D$  thought $\mathcal  P_D$ for
exceptional cases of Theorem~\ref{theo:main:SD}. Recall that for $e\in
\mathcal S_D$ we define an incomplete prototype $[e]=(w,1,0,e)\in \mathcal P_D$, 
where $w=(D^{2}-e^{2})/8$.
}
\end{table}

\section{Square-tiled surfaces}

A square-tiled surface is a form $(X,\omega)$ such that $\omega(\gamma) \in \Z^{2}$ 
for any $\gamma \in H_{1}(X,\Z)$. For such a surface, integration of the form $\omega$ 
gives a holomorphic map $X \longrightarrow \C /\Z^{2}$ which can be normalized so it is branched only 
the origin. The $n$ preimages of the square $[0,1]^{2}$ provide a tiling of the surface $X$. We say 
that $(X,\omega)$ is primitive if $\left\{ \omega(\gamma),\  \gamma \in H_{1}(X,\Z) \right\}= \Z^{2}$. 
Observe that a surface $(X,\omega)\in \Omega E_{D}(4)$ is square-tiled if and only if $D=d^{2}$ is a square. 
The following elementary proposition relates $d$ and $n$. 

\begin{Proposition}
\label{prop:square:tilde}
Let $(X,\omega)\in \Omega E_{d^{2}}(4)$ be a Prym eigenform.  Assume that $(X,\omega)$ is a 
primitive square-tiled surface made of $n$ squares, then
\begin{enumerate}
\item $n=d$, if $d$ is even; \medskip
\item $n=d$ or $n=2d$ depending on the $\GL^+(2,\R)$-orbit of  $(X,\omega)$, if $d$ is odd.


\end{enumerate}
\end{Proposition}

Theorem~\ref{theo:main:intro} allows us to get properties for the topology of the branched 
covers:

\begin{Corollary}
Fix $n\geq 5$. If $n\equiv 2 \mod 4$ then there are exactly two $\GL^+(2,\R)$-orbits of degree $n$, primitive 
square-tiled surfaces which are Prym eigenforms in $\Omega\mathfrak{M}(4)$, otherwise there is only one $\GL^+(2,\R)$-orbit.
\end{Corollary}

\begin{proof}
From              Proposition~\ref{prop:square:tilde}              and
Theorem~\ref{theo:main:intro}, 
the only possibility to get two $\GL^+(2,\R)$-orbits of square-tiled
surfaces made of $n$ squares is given when $n$ is even, and $n/2$ is odd, {\it i.e.} $n\equiv 2 \mod 4$.
\end{proof}

\section{Cusps of the Teichm\"uller curves in genus $3$}
\label{appendix:cusps}

The projection of the $\GL^+(2,\R)$-orbit of a Veech surface $(X,\omega)$ into the moduli space $\mathfrak{M}_g$ 
of Riemann surfaces  is a Teichm\"uller curve. Let $\SL(X,\omega)$ denote the Veech group of $(X,\omega)$ which 
is a lattice of $\SL(2,\R)$. A Teichm\"uller curve can never be compact, since any periodic direction of $(X,\omega)$ 
gives rise to a cusp, each cusps corresponds to the $\SL(X,\omega)$-orbits of a periodic direction of $(X,\omega)$.\medskip

Let $W_D(4)$ denote the projection of $\Omega E_D(4)$ into $\mathfrak{M}_3$. By Theorem~\ref{MainTh1}, we 
know that $W_D(4)$ is either a single Teichm\"uller curve, or the union of two Teichm\"uller curves. In both cases, we 
denote by $C(W_D(4))$ the total number of cusps in $W_D$, and by $C^{(k)}(W_D(4)), \; k=1,2,3,$ the number of cusps corresponding to decompositions into $k$ cylinders. Recall that each decomposition into three cylinders is characterized 
by a prototype in $\Qcal_D$, or in $\Pcal'_D$ up to the action of $\GL^+(2,\R)$, and clearly, if two cylinder 
decompositions correspond to  the same prototype then they are related by an element of $\SL(X,\omega)$. It follows 
that we have a bijection from the set of prototypes ($\Qcal_D\cup \Pcal'_D)$ and the set of cusps corresponding to 
decompositions into three cylinders. \medskip

If $D$ is not a square,  since $(X,\omega)$ does not admit any decomposition into one or two cylinders, we have
$$
C(W_D(4))=C^{(3)}(W_D(4))=|\Qcal_D| +|\Pcal'_D|=2|\Pcal_D|+|\Pcal'_D|.
$$
When $D=d^2, d\in \N$, the curve(s) in $W_D(4)$ has cusps corresponding to decompositions into one or two 
cylinders. It turns out that one can characterize the decompositions into one or two cylinders in a similar manner 
to  the decompositions  into  three cylinders,  and  therefore we  can
associate to each of such decompositions a prototype.

\begin{Theorem}
\label{theo:12cyl:prototype}
Let us define 
$$
\Pcal^s_D:=\left\{ (p,q)\in \N^2;\ 0<q<p < d/2 \qquad \textrm{and} \qquad \gcd(p,q,d)=1 \right\}.
$$
Then $C^{(1)}(W_D(4))=C^{(2)}(W_D(4))=|\Pcal^s_D|$. In particular:
$$
C(W_D(4))=2|\Pcal_D|+|\Pcal'_D|+2|\Pcal^s_D|.
$$
\end{Theorem}

To prove Theorem~\ref{theo:12cyl:prototype} we introduce the prototype
for cylinder decompositions into $1$ and $2$ cylinders:
\begin{Proposition}
\label{prop:12cyl:prototype}
Let $(X,\omega)$ be a surface in $\Omega E_{d^2}(4)$ for which the horizontal direction is completely periodic.
\begin{enumerate}

\item Suppose that $(X,\omega)$ has only one cylinder in the horizontal direction.\\
Let $\alpha_1, \beta_1, \alpha_{2,1}, \beta_{2,1}, \alpha_{2,2}, \beta_{2,2}$ be as in Figure~\ref{fig:1cyl:prototype}. 
Set $\alpha_2=\alpha_{2,1}+\alpha_{2,2},\beta_2=\beta_{2,1}+\beta_{2,2}$. Observe that 
$(\alpha_1,\beta_1,\alpha_2,\beta_2)$ is a symplectic basis for $H_1(X,\Z)^-$. Then there exists a unique generator
$T$ of $\Ord_D$ such that $T^*(\omega)=\lambda(T)\cdot\omega$, with $\lambda(T)>0$, and $T$ is written in the basis 
$(\alpha_1,\beta_1,\alpha_2,\beta_2)$ by the matrix $\displaystyle{\left(\begin{smallmatrix}
e & 0 & 2p & 2q\\ 
0 & e & 0 & 2s \\
s & -q & 0 & 0\\
0 & p & 0 & 0\\
\end{smallmatrix}\right)}$, where $(e,p,q,s)\in \Z^4$ satisfies
$$(\Pcal^s_D) \left\{\begin{array}{l}
e+4p =\sqrt{D},\\
\lambda(T)=s=e+2p>0,\\
0<q<p, \gcd(e,p,q,s)=1.\\
\end{array}\right. $$
Up to the action of $\GL^+(2,\R)$, we have
$$\left\{ \begin{array}{l}
\omega(\alpha_1)=(1,0), \; \omega(\beta_1)=(0,1)\\
\omega(\alpha_2)=(2p/s,0), \; \omega(\beta_2)=(2q/s,2).\\
\end{array}\right.$$

\item Suppose that $(X,\omega)$ is decomposed into two cylinders in the horizontal direction.\\
Let $\alpha_{1,1}, \beta_{1,1}, \alpha_{1,2}, \beta_{1,2}, \alpha_2, \beta_2$ be as in Figure \ref{fig:2cyl:prototype}. Set $\alpha_1=\alpha_{1,1}+\alpha_{1,2}, \beta_1=\beta_{1,1}+\beta_{1,2}$. Observe that $(\alpha_1,\beta_1,\alpha_2,\beta_2)$ is a symplectic basis for $H_1(X,\Z)^-$. Then there exists a unique generator $T$ of $\Ord_D$ such that $T^*(\omega)=\lambda(T)\cdot \omega$, with $\lambda(T)>0$, and $T$ is written in the basis $(\alpha_1,\beta_1,\alpha_2,\beta_2)$ by the following matrix $\displaystyle{\left(\begin{smallmatrix}
e & 0 & p & q\\ 
0 & e & 0 & s \\
2s & -2q & 0 & 0\\
0 & 2p & 0 & 0\\
\end{smallmatrix}\right)}$, where $(e,p,q,s)\in \Z^4$ also satisfies $(\Pcal^s_D)$. Up to the action of $\GL^+(2,\R)$, we have
$$\left\{ \begin{array}{l}
\omega(\alpha_1)=(1,0), \ \omega(\beta_1)=(0,1)\\
\omega(\alpha_2)=(p/s,0), \ \omega(\beta_2)=(q/s,1).\\
\end{array}\right.$$

\end{enumerate}
Conversely,  let $(X,\omega)$ be an  Abelian differential in
$\Omega\mathfrak{M}(4)$ having a cylinder decomposition into Models 
presented  above.  Suppose that  there  exists $(p,q,s,e)$  satisfying
$(\Pcal^s_D)$ such  that,  after  normalizing by  $\GL^+(2,\R)$,  all  the
conditions are satisfied, then $(X,\omega)$ belongs to $\Omega E_D(4)$.
\end{Proposition}

\begin{proof}[Proof of Proposition~\ref{prop:12cyl:prototype}]
We distinguish the two cases separately. \medskip

\noindent {\bf Case 1: decomposition into one cylinder}.

In the basis $(\alpha_1,\beta_1,\alpha_2,\beta_2)$ of $H_1(X,\Z)^-$, the intersection form is given by $\displaystyle{\left(\begin{smallmatrix} J & 0 \\ 0 & 2J \\ \end{smallmatrix}\right)}$. There exists a unique generator $T$ of $\Ord_D$ such that $T^*(\omega)=\lambda\cdot\omega$, with $\lambda>0$, which is written in this basis by a matrix of the form $\displaystyle{\left(\begin{smallmatrix} e & 0 & 2p & 2q \\ 0 & e & 2r & 2s \\ s & -q & 0 & 0 \\ -r & p & 0 & 0\\ \end{smallmatrix}\right)}$ (see Proposition~\ref{NormA1Prop}). Using $\GL^+(2,\R)$ we can assume that 
$$\left\{ \begin{array}{l}
\omega(\alpha_1)=(1,0),\ \omega(\beta_1)= (0,1), \\ 
\omega(\alpha_2)=(x+y,0),\ \omega(\beta_2)=(x,2), \text{ with } x>0, y>0
\end{array} \right.
$$
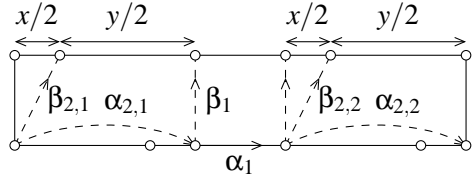
\begin{figure}[htbp]

\begin{tikzpicture}[scale=0.6]
\draw (0,0) -- (10,0) -- (10,2) -- (0,2) -- cycle;

\draw[->, >=angle 45] (4,0) -- (5.5,0); 
\draw[dashed, ->, >=angle 45] (4,0) -- (4,1.5); \draw[dashed] (4,1.5) -- (4,2);
\draw[dashed, ->, >=angle 45] (6,0) -- (6,1.5); \draw[dashed] (6,1.5) -- (6,2);

\draw[dashed, ->, >=angle 45] (0,0) -- (0.75,1.5); \draw[dashed] (0.75,1.5) -- (1,2); 
\draw[dashed, ->, >=angle 45] (6,0) -- (6.75,1.5); \draw[dashed] (6.75,1.5) -- (7,2);
\draw[dashed, ->,  >=angle 45](0,0)  .. controls (1,0.6)  and (3, 0.6) ..   (4,0);
\draw[dashed, ->,  >=angle 45](6,0)  .. controls (7,0.6)  and (9, 0.6) ..   (10,0);

\foreach \x in {(0,0), (3,0), (4,0), (6,0), (9,0), (10,0), (10,2), (7,2), (6,2), (4,2), (1,2), (0,2)} \filldraw[fill=white] \x circle (3pt);

\draw[<->, >=angle 45] (0,2.3) -- (1,2.3) ; \draw[<->, >=angle 45] (1,2.3) -- (4,2.3); \draw[<->, >=angle 45] (6,2.3) -- (7,2.3); \draw[<->, >=angle 45] (7,2.3) -- (10,2.3);

\draw (0.5,1) node[right] {$\beta_{2,1}$} (2.5,0.4) node[above] {$\alpha_{2,1}$}; \draw (4,1) node[right] {$\beta_1$} (5,0) node[below] {$\alpha_1$}; \draw (6.5,1) node[right] {$\beta_{2,2}$} (8.5,0.4) node[above] {$\alpha_{2,2}$};

\draw (0.5,2.3) node[above] {$x/2$} (2.5,2.3) node[above] {$y/2$} (6.5,2.3) node[above] {$x/2$} (8.5,2.3) node[above] {$y/2$};

\end{tikzpicture}
\caption{Decomposition into one cylinder: $(\alpha_1,\beta_1,\alpha_2,\beta_2)$, where $\alpha_2=\alpha_{2,1}+\alpha_{2,2}, \beta_2=\beta_{2,1}+\beta_{2,2}$, is a symplectic basis for $H_1(X,\Z)^-$.}
\label{fig:1cyl:prototype}
\end{figure}

In other words, $\mathrm{Re}(\omega)=(1,0,x+y,x)$ and $\mathrm{Im}(\omega)=(0,1,0,2)$ in the basis dual to $(\alpha_1,\beta_1,\alpha_2,\beta_2)$. We must have
\begin{equation}\label{1cyl:prototype:eq:Re}
(1,0,x+y,x)\cdot T  =  \lambda(1,0,x+y,x) 
\end{equation} 
and 
\begin{equation}
\label{1cyl:prototype:eq:Im}
(0,1,0,2)\cdot T  =  \lambda(0,1,0,2) 
\end{equation} 
It follows immediately from (\ref{1cyl:prototype:eq:Im}) that $r=0$ and $e+2p=s=\lambda$. Note that $\lambda$ is the positif root of the characteristic polynomial of $T$, therefore $\lambda^2=e\lambda+2ps$. The condition (\ref{1cyl:prototype:eq:Re}) then implies that $2p=\lambda(x+y)$ and $2q=\lambda x$ from which we deduce in particular that $0<q<p$. Since $T$ is a generator of $\Ord_D$, we have $D=e^2+8ps=(e+4p)^2$, and by the properness of $\Ord_D$ in $\mathrm{End}(\Prym(X,\rho))$, we have $\gcd(e,p,q,s)=1$. All the conditions in $(\Pcal^s_D)$ are now fulfilled. \medskip

\noindent {\bf Case 2: decomposition into two cylinders.} 

In the basis $(\alpha_1,\beta_1,\alpha_2,\beta_2)$ of $H_1(X,\Z)^-$, the intersection form is given by $\displaystyle{\left(\begin{smallmatrix} 2J & 0 \\ 0 & J \\ \end{smallmatrix}\right)}$. There exists a unique generator $T$ of $\Ord_D$ such that $T^*(\omega)=\lambda\cdot\omega$, with $\lambda>0$, which is written in this basis by a matrix of the form $\displaystyle{\left(\begin{smallmatrix} e & 0 & p & q \\ 0 & e & r & s \\ 2s & -2q & 0 & 0 \\ -2r & 2p & 0 & 0\\ \end{smallmatrix}\right)}$ (see Proposition~\ref{NormA2Prop}).

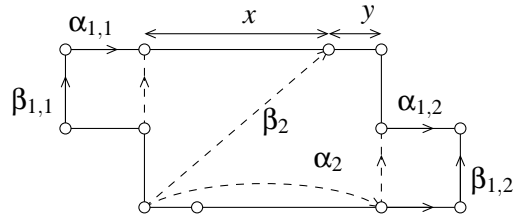
\begin{figure}[htbp]
\begin{tikzpicture}[scale=0.7]
\draw (0,0) -- (6,0) -- (6,1.5) -- (4.5,1.5) -- (4.5,3) -- (-1.5,3) -- (-1.5,1.5) -- (0,1.5) -- cycle;
\draw[->, >=angle 45] (-1.5,3) -- (-0.5,3); \draw[->, >=angle 45] (-1.5, 1.5) -- (-1.5,2.5); 
\draw[dashed, ->,>=angle 45] (0,1.5) -- (0,2.5); \draw[dashed] (0,2.5) -- (0,3);
\draw[->, >=angle 45] (4.5,0) -- (5.5,0); \draw[->, >=angle 45] (4.5,1.5) -- (5.5,1.5); \draw[->, >=angle 45] (6,0) -- (6,1); \draw[dashed, ->,>=angle 45] (4.5,0) -- (4.5,1); \draw[dashed] (4.5,1) -- (4.5, 1.5);

\draw[dashed, ->,>=angle 45] (0,0) -- (3.5,3); \draw[dashed, ->,  >=angle 45](0,0)  .. controls (1,0.6)  and (3.5, 0.6) ..   (4.5,0);

\foreach \x in {(0,0), (1,0), (4.5,0), (6,0), (6,1.5), (4.5,1.5), (4.5,3), (3.5,3), (0,3), (-1.5,3), (-1.5,1.5), (0,1.5)} \filldraw[fill=white] \x circle (3pt);

\draw (-1,3) node[above] {$\alpha_{1,1}$} (-1.5, 2) node[left] {$\beta_{1,1}$} (5.25,1.5) node[above] {$\alpha_{1,2}$} (6,0.5) node[right] {$\beta_{1,2}$};
\draw (3.5,0.5) node[above] {$\alpha_2$} (2.5,2.1) node[below] {$\beta_2$};

\draw[<->, >=angle 45] (0,3.3) -- (3.5,3.3); \draw[<->, >=angle 45] (3.5,3.3) -- (4.5,3.3);
\draw (2,3.3) node[above] {$x$} (4.25,3.3) node[above] {$y$};
\end{tikzpicture}

\caption{Decomposition into two cylinders: $(\alpha_1,\beta_1,\alpha_2,\beta_2)$, where $\alpha_1=\alpha_{1,1}+\alpha_{1,2}, \beta=\beta_{1,1}+\beta_{1,2}$, is a symplectic basis for $H_1(X,\Z)^-$.}
\label{fig:2cyl:prototype}
\end{figure}
Using $\GL^+(2,\R)$ we can assume that 
$$\left\{ \begin{array}{l}
\omega(\alpha_1)=(1,0), \ \omega(\beta_1)= (0,1), \\ 
\omega(\alpha_2)=(x+y,0), \ \omega(\beta_2)=(x,1), \text{ with } x>0, y>0
\end{array} \right.
$$
The remainder of the proof for this case follows the same lines as the previous case. 
\end{proof}

We are now ready to prove Theorem~\ref{theo:12cyl:prototype}:

\begin{proof}[Proof of Theorem~\ref{theo:12cyl:prototype}]
Let $s=\sqrt{D}-2p>0$ and $e=\sqrt{D}-4p>0$. 
It is easy to check that the tuple $(e,p,q,s)\in \Z^4$ satisfies the conditions in $(\Pcal^s_D)$ if and
only if $(p,q)\in \Pcal^s_D$. 
From Proposition~\ref{prop:12cyl:prototype}, we know that each decomposition into one or two cylinders of 
the surfaces in $\Omega E_D(4)$ gives rise to an element of $\Pcal^s_D$.  If two decompositions (with the same 
number of cylinders) give the same element in $\Pcal^s_D$ then there exists an element of the Veech group which
maps one decomposition to other. Conversely, given a pair $(p,q)$ in $\Pcal^s_D$, we can construct a surface in
$\Omega E_D(4)$ which admits a decomposition into one or two cylinders in the horizontal direction. Therefore, 
we have a bijection from $\Pcal^s_D$ to the set of cusps corresponding to decomposition into one cylinder, 
and a bijection from $\Pcal^s_D$ to the set of cusps corresponding to decompositions into two cylinders of $W_D(4)$. 
\end{proof}

\section{Components of the Prym eigenforms locus in genus $4$}
\label{appendix:other:strata}

The approach we use in this paper, namely prototypes and Butterfly moves, can also be employed to investigate the 
connectedness of the locus $\Omega E_D(6)$. Recall that $\Omega E_D(6)$ is the intersection of the Prym eigenform 
locus and the stratum $\Omega \mathfrak{M}(6)$. Following~\cite{Mc7} 
$\Omega E_D(6)$ is the union of finitely many $\GL^+(2,\R)$-orbits of Veech surfaces. We provide 
the following classification:

\begin{Theorem}
\label{theo:H6:appendix}
For any $D\in \N$, $D\equiv 0,1 \mod 4$, and $D \not\in \{4,9\}$, the loci $\Omega E_D(6)$ are 
non empty and pairwise disjoints. Moreover if $D \not \in \{8,12,16,36,41,52,68,84,100\}$ then $\Omega E_D(6)$
\begin{enumerate}
\item is connected if $D$ is odd,
\item has at most two components if $D$ is even.
\end{enumerate}
For the exceptional cases $D=41,52,68,84$ the locus $\Omega E_D(6)$ has at most three components. \\
For the exceptional cases $D=8,12,16,6^{2},10^{2}$ the locus $\Omega E_D(6)$ is connected.
\end{Theorem}

\subsection{Strategy of a proof}

We briefly sketch a proof of Theorem~\ref{theo:H6:appendix}.
Surfaces in $\Omega E_D(6)$ admit two types of decomposition into four cylinders, which will be called 
Model $A$, and Model $B$. The Model $A$ is characterized by the existence of simple cylinders 
(see Figure~\ref{fig:H6:ModA}) while the Model $B$ is characterized by the absence of such cylinders 
(see Figure~\ref{fig:H6:ModB}).

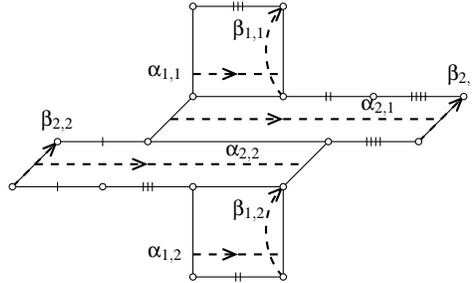
\begin{figure}[htbp]
\centering \subfloat{%
\begin{tikzpicture}[scale=0.6]

\draw (0,0) -- (-2,0) -- (-3,-1) -- (1,-1) -- (1,-3) -- (3,-3) -- (3,-1) -- (4,0) -- (6,0)
-- (7,1) -- (3,1) -- (3,3) --  (1,3) -- (1,1) -- cycle; 

\draw (1,1) -- (3,1) (0,0) -- (4,0) (1,-1) -- (3,-1);

\draw[thick, dashed, ->, >= angle 45] (1,1.5) -- (2,1.5); \draw[thick,
  dashed] (2,1.5)  -- (3,1.5); \draw[thick,  dashed, ->, >=  angle 45]
(3,1) .. controls (2.5,1.5) and (2.5,2.5) .. (3,3);

\draw[thick,  dashed,   ->,  >=   angle  45]  (0.5,0.5)   --  (3,0.5);
\draw[thick, dashed] (3,0.5) -- (6.5,0.5); 
\draw[thick, dashed, ->, >=  angle 45] (6,0) -- (7,1);

\draw[thick,  dashed,   ->,  >=   angle  45]  (-2.5,-0.5)   --  (0,-0.5);
\draw[thick, dashed] (0,-0.5) -- (3.5,-0.5); 
\draw[thick, dashed, ->, >=  angle 45] (-3,-1) -- (-2,0);

\draw[thick,  dashed,   ->,  >=   angle  45]  (1,-2.5)   --  (2,-2.5);
\draw[thick, dashed] (2,-2.5) -- (3,-2.5); 

\draw[thick, dashed, ->, >= angle 45] (3,-3) .. controls (2.5,-2.5) and (2.5,-1.5) .. (3,-1);

\filldraw[fill=white,draw=black] (0,0) circle (2pt)  (4,0)  circle (2pt)  (6,0)
circle (2pt) (7,1) circle (2pt)  (5,1) circle (2pt) (3,1) circle (2pt)
(3,3) circle (2pt) (1,3) circle (2pt) (1,1) circle (2pt) (-2,0) circle (2pt)
(-3,-1) circle (2pt) (-1,-1) circle (2pt) (1,-1) circle (2pt) (1,-3) circle (2pt)
(3,-3) circle (2pt) (3,-1) circle (2pt);

\draw   (2,-3)   +(-0.05,-0.1)   --  +(-0.05,0.1)   +(0.05,-0.1)   -- +(0.05,0.1);
\draw   (4,1)   +(-0.05,-0.1)   --  +(-0.05,0.1)   +(0.05,-0.1)   -- +(0.05,0.1);

\draw   (2,3)   +(0,-0.1)  -- +(0,0.1)  +(-0.1,-0.1) -- +(-0.1,0.1)  +(0.1,-0.1) --  +(0.1,0.1);
\draw   (0,-1)   +(0,-0.1)  -- +(0,0.1)  +(-0.1,-0.1) -- +(-0.1,0.1)  +(0.1,-0.1) --  +(0.1,0.1);

\draw  (-2,-1) +(0,-0.1)  -- +(0,0.1) ;
\draw  (-1,0) +(0,-0.1)  -- +(0,0.1) ;

\draw  (5,0) +(-0.15,-0.1)   --  +(-0.15,0.1)   +(-0.05,-0.1)   -- +(-0.05,0.1) 
+(0.05,-0.1)   --  +(0.05,0.1)   +(0.15,-0.1)   -- +(0.15,0.1);
\draw  (6,1) +(-0.15,-0.1)   --  +(-0.15,0.1)   +(-0.05,-0.1)   -- +(-0.05,0.1) 
+(0.05,-0.1)   --  +(0.05,0.1)   +(0.15,-0.1)   -- +(0.15,0.1);

\draw   (1,1.5)  node[left]  {$\scriptstyle   \alpha_{1,1}$}  (2.25,2)
node[above]    {$\scriptstyle   \beta_{1,1}$}    (1,-2.5)   node[left]
{$\scriptstyle  \alpha_{1,2}$}  (2.25,-2)  node[above]  {$\scriptstyle
  \beta_{1,2}$} (4.5,0.75)  node[right] {$\scriptstyle \alpha_{2,1}$} (7,1)
node[above] {$\scriptstyle \beta_{2,1}$}
(1.5,-0.25)  node[right] {$\scriptstyle \alpha_{2,2}$} (-2,0)
node[above] {$\scriptstyle \beta_{2,2}$};
\end{tikzpicture}
}
\caption{ Cylinder decomposition in $\Omega E_D(6)$:  Model  $A$. 
For $i=1,2$, setting $\alpha_{i}:=\alpha_{i,1}+\alpha_{i,2}$ and
$\beta_i:=\beta_{i,1}+\beta_{i,2}$ observe that $\{\alpha_1,\beta_1,\alpha_2,\beta_2\}$     is     a    basis     of
  $H_{1}(X,\Z)^{-}$.  }
\label{fig:H6:ModA}
\end{figure}

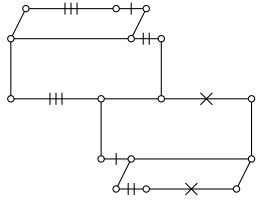
\begin{figure}[htbp]
\begin{center}
\begin{tikzpicture}[scale=0.4]

\draw (0,0) -- (3,0) -- (3,-2) -- (4,-2) -- (3.5,-3) -- (7.5,-3) -- (8,-2) -- (8,0) -- (5,0) -- (5,2) -- (4,2) -- (4.5,3) --(0.5,3) -- (0,2) -- cycle;

\draw (0,2) -- (4,2) (3,0) -- (5,0) (4,-2) -- (8,-2);

\draw (4,3) +(0,0.2) -- +(0,-0.2); \draw (3.5,-2) +(0,0.2) -- +(0,-0.2);
\draw (4.5,2) +(-0.1,0.2) -- +(-0.1,-0.2) +(0.1,0.2) -- +(0.1,-0.2); \draw (4,-3) +(-0.1,0.2) -- +(-0.1,-0.2) +(0.1,0.2) -- +(0.1,-0.2);
\draw (2,3) +(-0.2,0.2) -- +(-0.2,-0.2) +(0,0.2) -- +(0,-0.2) +(0.2,0.2) -- +(0.2,-0.2); \draw (1.5,0) +(-0.2,0.2) -- +(-0.2,-0.2) +(0,0.2) -- +(0,-0.2) +(0.2,0.2) -- +(0.2,-0.2);
\draw (6.5,0) +(-0.2,0.2) -- +(0.2,-0.2) +(-0.2,-0.2) -- +(0.2,0.2); \draw (6,-3) +(-0.2,0.2) -- +(0.2,-0.2) +(-0.2,-0.2) -- +(0.2,0.2);

\foreach \x in {(0,0), (3,0), (3,-2), (4,-2), (3.5,-3), (4.5,-3), (7.5,-3), (8,-2), (8,0), (5,0), (5,2), (4,2), (4.5,3), (3.5,3), (0.5,3), (0,2)} \filldraw[fill=white] \x circle (3pt);

\end{tikzpicture}
\end{center}
\caption{Cylinder decomposition in $\Omega E_D(6)$: Model $B$.}
\label{fig:H6:ModB}
\end{figure}

We first need to normalize the decompositions in model $A$. 

\begin{Proposition}
\label{NormAProp}
Let $(X,\omega) \in \Omega E_D(6)$  be a Prym eigenform which admits a
cylinders       decomposition      of       Model       $A$, equipped with 
the   symplectic   basis    presented   in Figure~\ref{fig:H6:ModA}.

\begin{itemize}
\item[(i)] There exists a unique generator $T$ of $\Ord_D$ such that the matrix of $T$ in the basis $(\alpha_1,\beta_1, \alpha_2,\beta_2)$ has the form $\left(\begin{smallmatrix}
e\cdot\id_2 & B \\ B^* & 0\\
\end{smallmatrix}\right)$, and $T^*(\omega)=\lambda(T)\omega$ with $\lambda(T)>0$.

\item[(ii)] Up    to    the   action $\GL^+(2,\R)$ and Dehn twists,
there exist $w,h,t \in \N$ such that  the          tuple         $(w,h,t,e)$         satisfies

$$(\tilde{\Pcal})\left\{\begin{array}{l}        w>0,h>0,\;        0\leq
  t<\gcd(w,h),\\    \gcd   (w,h,t,e)    =1,\\    D=e^2+4w   h,\\    0<
  \lambda:=\frac{e+\sqrt{D}}{2}<\frac{w}{2} \  \textrm{ (or, equivalently, } w>2(e+2h)) \\
\end{array}
\right.,
$$
and the matrix of $T$ in  the  basis  $\{\alpha_1,\beta_1,\alpha_2,\beta_2\}$  is
  $\left(%
\begin{smallmatrix}
  e & 0 & w & t \\ 0 & e & 0 & h \\ h & -t & 0 & 0 \\ 0 & w & 0 & 0
  \\
\end{smallmatrix}%
\right)$. Moreover,  in these coordinates we  have
$$
\left\{                                              \begin{array}{l}
  \omega(\Z\alpha_{2,1}+\Z\beta_{2,1})=\omega(\Z\alpha_{2,2}+\Z\beta_{2,2})=\Z(\frac{w}{2},0)+\Z(\frac{t}{2},\frac{h}{2})
  \\ \omega(\Z\alpha_{1,1}+\Z\beta_{1,1})=\omega(\Z\alpha_{1,2}+\Z\beta_{1,2})=\frac{\lambda}{2}\cdot
  \Z^{2}
\end{array}
\right.
$$
\end{itemize}

Conversely,  let $(X,\omega) \in \Omega\mathfrak{M}(6)$ having a four-cylinder decomposition. 
Assume there exists $(w,h,t,e)  \in \Z^4$ satisfying 
$(\tilde{\Pcal})$,  such  that  after  normalizing  by  $\GL^+(2,\R)$,  all  the above
conditions are fulfilled. Then $(X,\omega) \in \Omega E_D(6)$.
\end{Proposition}

The proof of Proposition~\ref{NormAProp} is analogous to the proof of  Proposition~\ref{NormA2Prop}, 
the only difference is that the intersection form in $H_1(X,\Z)^-$ is now given by $\DS{\left( \begin{smallmatrix}
2J & 0 \\ 0 & 2J \\
\end{smallmatrix} \right)}$.

\begin{Remark}
One can also provide prototypes for Model B as
$$(\tilde{\Pcal'})\left\{\begin{array}{l} w>0,h>0,\; 0\leq
  t<\gcd(w,h),\\ \gcd (w,h,t,e) =1,\\ D=e^2+4w h,\\ \frac{w}{2} < \lambda:=\frac{e+\sqrt{D}}{2}<w \\
\end{array}
\right.,
$$
\end{Remark}

Putting  prototypes  for  Model A  and  Model  B  together, we  get  a
parametrisation  for  cusps associated  to  Models  A  and B  (compare
with~\cite{Mc4}):

\begin{Proposition}
\label{prop:cusps:egalite}
For any  surface of  $\Omega E_D(6)$, the  set of  periodic directions
associated to Models A and B is parametrised by
$$\left\{ (w,h,t,e) \in \Z^4,
\begin{array}{l}
w>0,  \ h>0,\  0\leq  t<\gcd(w,h),  \ h+e  <  w, \\  \gcd(w,h,t,e)=1,
\textrm{ and } D=e^2+4hw.\\
\end{array}
\right\}
$$
\end{Proposition}

The next proposition tells us that, except the case $D=5$, the surfaces in $\Omega E_D(6)$ always admit a decomposition in Model $A$, its proof is similar to Proposition~\ref{prop:D8:topo}.

\begin{Proposition}\label{prop:H6:noModA}
Let $(X,\omega)$ be an eigenform in $\Omega E_D(6)$. Then the flat surface associated to $(X,\omega)$ has no decompositions in model $A$ if and only if $D=5$.
\end{Proposition}

For a fixed $D$, we denote by $\tilde{\Pcal}_D$ the set of $(w,h,t,e)$ satisfying $(\tilde{\Pcal})$, the elements of 
$\tilde{\Pcal}_D$ are called {\em prototypes}. We can define the {\em Butterfly moves} $B_q, q\in \N\cup\{\infty\}$, 
for decompositions of type $A$ in the same way as in Section~\ref{sec:butterfly}. Note that in this case the 
Butterfly moves preserve the type of the decomposition.  The admissibility condition now becomes

\begin{equation*}
0 < \lambda q < \frac{w}{2} \Leftrightarrow (e+4qh)^2<D 
\end{equation*}

The actions of the Butterfly moves on $(\tilde{\Pcal})$ are the same as the case $\Omega E_D(4)$ (see 
Propositions~\ref{BM1Prop} and~\ref{BM2Prop}),  namely

\begin{enumerate}
\item If $q\in \N$ then

$$\left\{\begin{array}{ccl}
e' & = & -e-4qh,\\
h' & = & \gcd(qh,w+qt)\\ 
\end{array}\right. $$

\item If $q=\infty$ then 

$$\left\{\begin{array}{ccl}
e' & = & -e-4h,\\
h' & = & \gcd(t,h)\\ 
\end{array}\right. $$
\end{enumerate}

We can parametrize the set of {\em reduced  prototypes} (see Section~\ref{sec:reduced:prot}), by the set
$$
\tilde{\mathcal{S}}_{D} = \left\{ e\in \mathbb Z : e \equiv D \mod 2 \textrm{ and } e^{2},\ (e+4)^{2} < D   \right\}.
$$

We call an equivalence class of the equivalence relation generated by the Butterfly moves on $\tilde{\Pcal}_D$
(respectively, $\tilde{\mathcal{S}}_D$) a component of $\tilde{\Pcal}_D$ (respectively, $\tilde{\mathcal{S}}_D$). 
We then have

\begin{Theorem}
\label{theo:H6:connect:SD}
Let $D\geq 12$  be a  discriminant.  Let us assume that
$$ D   \not \in
\left\{
\begin{array}{l}
12,16,17,20,25,28,36,73,88,97,105,112,121,124,136,145,148, \\
169,172,184,193,196,201,217,220,241,244,265,268,292,304, \\
316, 364,385,436,484,556,604,676,796,844,1684
\end{array}
\right\}.
$$ 

Then the set $\tilde{\mathcal S}_{D}$ is non empty and has either
\begin{itemize}
\item three components, $\{e\in \tilde{\mathcal S}_{D},\ e\equiv 0 \textrm{ or
} 4 \mod 8\}$, $\{e\in \tilde{\mathcal S}_{D},\ e\equiv 2 \mod 8\}$ and \\
$\{e\in \tilde{\mathcal S}_{D},\ e\equiv -2 \mod 8\}$, if $D \equiv 4 \mod 8$,
\item two components,
\begin{itemize}
\item $\{e\in \tilde{\mathcal S}_{D},\ e\equiv 1 \textrm{ or } 3 \mod 8\}$ and $\{e\in \tilde{\mathcal S}_{D},\ e\equiv 
-1 \textrm{ or } -3 \mod 8\}$ if $D \equiv 1 \mod 8$,
\item $\{e\in \tilde{\mathcal S}_{D},\ e\equiv 0 \textrm{ or } 4 \mod 8\}$ and $\{e\in \tilde{\mathcal S}_{D},\ e\equiv 
+ 2 \textrm{ or } -2 \mod 8\}$ if $D \equiv 0 \mod 8$,
\end{itemize}
\item only one component, otherwise.
\end{itemize}
\end{Theorem}

 \begin{Remark}
 There is a simple congruence condition that explains why 
$\tilde{\mathcal S}_D$ is not connected for some values of $D$.
\end{Remark}

As a corollary we draw: 

\begin{Theorem}\label{theo:H6:connect:PD}
Let $D\geq 12$  be a  discriminant. If $D \not \in \{36,41,52,68,84,100\}$ then $\tilde{\mathcal P}_{D}$ is non empty and has either
\begin{itemize}
\item only one component if $D=12$ or $D=16$,
\item two components, $\{p\in \mathcal P_{D},\ e\equiv 0 \mod 4\}$ and $\{p\in \tilde{\mathcal P}_{D},\ e\equiv 2 \mod 4\}$, 
if $D$ is even, or
\item only one component otherwise.
\end{itemize}
For the exceptional cases mentioned above, $\tilde{\mathcal P}_{D}$ has three components.
\end{Theorem}

\begin{Remark}
Again there is a simple congruence relation that explain why 
there is (at least) two components when $D$ is even. Indeed since $e'=-e-4qh$ and $e$ is even, the value 
of $e$ modulo $4$ is constant.
\end{Remark}

To prove the previous theorems, we use similar ideas to the proofs of  Theorem~\ref{theo:main:SD} and 
Theorem~\ref{theo:main:PD}. This is straightforward. Theorem~\ref{theo:H6:appendix} is then a 
direct consequence of these results.


\bibliographystyle{amsalpha}
\bibliography{bib/journals_abbrev,bib/articles}


\end{document}